\documentclass[10pt,reqno]{amsart}

\usepackage{amsmath,amssymb,amsfonts,latexsym,amsthm}
\usepackage{mathtools}
\usepackage{mathrsfs}
\usepackage{hyperref}
\numberwithin{equation}{section}

\usepackage{graphicx}

\usepackage{ifthen} 

\calclayout
\provideboolean{shownotes} 
\setboolean{shownotes}{false} 
\newcommand{\margnote}[1]{
	\ifthenelse{\boolean{shownotes}}%
	{\marginpar{\raggedright\tiny\texttt{#1}}}%
	{}%
}

\newcommand{\hole}[1]{
	\ifthenelse{\boolean{shownotes}}%
	{\begin{center} \fbox{ \rule {.25cm}{0cm}
				\rule[-.1cm]{0cm}{.4cm} \parbox{.85\textwidth}{\begin{center}
						\texttt{#1}\end{center}} \rule {.25cm}{0cm}}\end{center}}
	{}
}

\theoremstyle{plain}
\newtheorem{defi}{Definition}[section]
\newtheorem{theo}{Theorem}[section]
\newtheorem{lema}{Lemma}

\newtheorem{prop}{Proposition}

\newtheorem{cond}{Condition}
\newtheorem*{condition}{Condition}
\newtheorem{coro}{Corollary}
\newtheorem{rem}{Remark}

\theoremstyle{remark}

\theoremstyle{remark}

\begin{document}

	\title[Systems with fractional dissipation]{Local solutions to hyperbolic systems with fractional dissipation}
	
\author[F. Angeles]{Felipe Angeles}


\makeatletter
\def\@setauthors{%
	\begingroup
	\centering
	{\authors\par}
	\vspace{0.5em}
	\normalfont\small
	Instituto de Investigaciones en Matem\'aticas
	Aplicadas y en Sistemas\\
	Universidad Nacional Aut\'onoma de M\'exico\\
	Circuito Escolar s/n, Ciudad Universitaria\\
	C.P. 04510, Ciudad de M\'exico, Mexico\\[0.3em]
	\texttt{felipe.angeles@iimas.unam.mx}\\[0.3em]
	ORCID: \href{https://orcid.org/0000-0003-0309-9475}
	{0000-0003-0309-9475}\par
	\endgroup
}
\makeatother

\begin{abstract}
We study the local Cauchy problem for a class of hyperbolic quasilinear evolution systems with fractional dissipation and commutators of the fractional Laplacian $(-\Delta)^{\alpha}$ with arbitrary $\alpha\in(0,1)$ in the framework of Sobolev spaces. The analysis is motivated by the compressible Euler alignment system with singular communication weights. We show that the contribution of the commutators to the energy is absorbed by the dissipation. Our approach does not rely on the cancellation condition previously used to recover the coercivity of the highest order nonlocal terms. We further show that the commutator term induces a loss of regularity in the stability estimates, preventing the standard contraction argument from being applied. Therefore, we instead establish the relative compactness of the solution operator and apply Schauder's fixed point theorem to provide local existence for large initial data. The resulting theory identifies the obstacles to deduce local uniqueness in the presence of commutators. 
\end{abstract}
	\maketitle
	
	\setcounter{tocdepth}{1}

	\section{Introduction}
	The present work deals with the local well-posedness of the following Cauchy problem for a nonlinear system of evolution equations
\begin{equation}
	\label{eq:quasilinearsystem1}
	\begin{aligned}
		A^{0}(U)\partial_{t}U+A^{j}(U)\partial_{j}U+D(U)U&+B(U)\Lambda^{2\alpha}U+H(U)[\Lambda^{2\alpha},G(U)]\mathcal{L}U=0,\\
		U\rvert_{t=0}&=U_{0},
	\end{aligned}
\end{equation}
where the notation $\Lambda^{\sigma}=(-\Delta)^{\sigma/2}$ and the summation convention have been used. Here, $U$ stands as the variable to be determined and is such that $U=U(x,t)\in\mathbb{R}^{N}$ for all $(x,t)\in\mathbb{R}^{d}\times[0,T]$, for some $T>0$. The coefficients $A^{0}(U)$, $A^{j}(U)$, $D(U)$, $B(U)$, $H(U)$ and $G(U)$ are real matrices of order $N\times N$ and $\mathcal{L}$ is a constant real matrix of the same order.  The fractional Laplace operator $\Lambda^{2\alpha}:=(-\Delta)^{\alpha}$ is the Fourier multiplier with symbol $|\xi|^{2\alpha}$, namely
\begin{align*}
	\mathcal{F}\left(\Lambda^{2\alpha}f\right)(\xi)=|\xi|^{2\alpha}\widehat{f}(\xi).
\end{align*}
Equivalently, for smooth enough functions, 
\begin{align}
	\Lambda^{2\alpha}f(x)=c_{\alpha}\operatorname{P.V.}\int_{\mathbb{R}^{d}}\frac{f(x)-f(y)}{|x-y|^{d+2\alpha}}dy,\quad 0<\alpha<1,\quad c_{\alpha}:=\frac{2^{2\alpha}\Gamma(\alpha+\tfrac{d}{2})}{\pi^{\tfrac{d}{2}}|\Gamma(-\alpha)|}.\label{eq:fractionalPV}
\end{align}
We use the notation
\begin{align*}
	[\Lambda^{2\alpha},f]g:=\Lambda^{2\alpha}(fg)-f\Lambda^{2\alpha}g,
\end{align*}
for the commutator of the operator $\Lambda^{2\alpha}$. 

We investigate the local well-posedness of the Cauchy problem associated with \eqref{eq:quasilinearsystem1}. We assume a block structure of the matrix coefficients that is reminiscent of the structure of a quasilinear hyperbolic-parabolic composite system of evolution equations (see, \cite{kawa}). In particular, there is a partition $U=(u,v)^{\top}\in\mathbb{R}^{k}\times\mathbb{R}^{m}$ ($k+m=N$), where $v$ is the variable that experiences the fractional diffusion and $u$ is the \emph{hyperbolic} variable.

Many examples of systems of the type \eqref{eq:quasilinearsystem1} can be formally derived from a particle system of evolution equations for $N$ individuals known as Cucker-Smale model \cite{cucker}:
\begin{equation}
	\label{eq:cuckersmale}
	\begin{aligned}
		\frac{d}{dt}x_{i}&=v_{i}(t),\\
		\frac{d}{dt}v_{i}&=\frac{1}{N}\sum_{j\neq i}(v_{j}-v_{i})\phi(|x_{i}-x_{j}|)-\beta v_{i},
	\end{aligned}
\end{equation}
where $x_{i},v_{i}\in\mathbb{R}^{d}$ denote the spatial position and velocity of the $i$-th individual (or particle). In this system, each individual adjusts (or aligns) its velocity to a weighted average of velocities of its neighbors. The function $\phi$ measures the strength of this alignment interactions and is referred as the \emph{communication weight} (se also \cite{motsch}).

When $N\rightarrow\infty$, the ensemble of particles can be described by a density function $f=f(t,x,v)$ at the spatial-velocity position $(x,v)\in\mathbb{R}^{d}\times\mathbb{R}^{d}$ at time $t$. Then, the evolution of $f$ is determined by the Vlasov type kinetic equation
\begin{align}
	\partial_{t}f+v\cdot\nabla_{x}f+\nabla_{v}\cdot\left(\iint_{\mathbb{R}^{d}\times\mathbb{R}^{d}}(v^{\prime}-v)\phi(|x-y|)f(t,y,v^{\prime})dydv^{\prime}f(t,x,v)\right)-\beta\nabla_{v}\cdot(vf)=0.\label{eq:vlasov}
\end{align}
In this work we are concern with the case 
\begin{align}
	\label{eq:fractionalkernel}
	\phi(x):=\frac{1}{|x|^{d+2\alpha}}\quad\mbox{with}\quad\alpha\in(0,1),
\end{align}
associated with the action of the fractional Laplacian. One of the main motivations for the study of singular alignment kernels is the collision avoidance between particles (\cite{haliu}, \cite{hauray}, \cite{carrillo}, \cite{cuckerdong}, \cite{ahn}).

When "monokinetic" solutions of the form $f(t,x,v)=\rho(x,t)\delta_{w=v(x,t)}$ are considered in \eqref{eq:vlasov} and $\beta=0$, $\rho$ and $v$ solve the associated limit system given by the pressureless Euler system with nonlocal flocking dissipation
\begin{equation}
	\label{eq:pressurelessEuler}
	\begin{aligned}
		\partial_{t}\rho+\nabla\cdot(\rho v)&=0,\\
		\partial_{t}(\rho v)+\nabla\cdot(\rho v\otimes v)&=\int_{\mathbb{R}^{d}}\left(\frac{v(y,t)-v(x,t)}{|y-x|^{d+2\alpha}}\right)\rho(x,t)\rho(y,t)dy.
	\end{aligned}
\end{equation}
In general, a different choice for the distribution $f$ leads to different pressure laws in the resulting hydrodynamic system. For example, the choice $f(t,x,v)=\rho(x,t)(2\pi)^{-d/2}e^{-|w-v|^{2}/2}$ leads to the Euler system with nonlocal dissipation and linear pressure $p(\rho)=\rho$. For rigorous derivations of the hydrodynamic limit system see, \cite{karper} and \cite{figalli}.

Another example was proposed by L. Chen, C. Tan and L. Tong in \cite{chenalign},
\begin{equation}
	\label{eq:Eulervelal}
	\begin{aligned}
		\partial_{t}\rho+\nabla\cdot(\rho v)&=0,\\
		\partial_{t}(\rho v)+\nabla\cdot(\rho v\otimes v)+\nabla p(\rho)&=-\beta\rho v-c_{\alpha}\int_{\mathbb{R}^{d}}\left(\frac{v(x,t)-v(y,t)}{|y-x|^{d+2\alpha}}\right)\rho(x,t)\rho(y,t)dy,
	\end{aligned}
\end{equation}
where the pressure is given by the power law $p(\rho)=\rho^{\gamma}$ with $\gamma\geq 1$, and the damping coefficient $\beta\geq 0$. We refer to systems \eqref{eq:pressurelessEuler} and \eqref{eq:Eulervelal} as compressible \emph{Euler alignment systems}, with \eqref{eq:pressurelessEuler} being pressureless.

In \cite{chenalign}, the authors consider the initial value problem for \eqref{eq:Eulervelal} near the constant equilibrium state $(\rho_{c},v_{c})=(1,0)$. Following \cite{sideris}, they introduce a change of variables $\sigma=\sigma(\rho)$, and reformulate \eqref{eq:Eulervelal} into the following equivalent system
\begin{equation}
	\label{eq:Eulertantong}
	\begin{aligned}
		&\partial_{t}\sigma+u\cdot\nabla\sigma+\left(\frac{\gamma-1}{2}\sigma+\sqrt{\gamma}\right)\nabla\cdot u=0,\\
		&\partial_{t}u+u\cdot\nabla u+\left(\frac{\gamma-1}{2}\sigma+\sqrt{\gamma}\right)\nabla\sigma=-\beta u-\Lambda^{2\alpha}u-[\Lambda^{2\alpha},u](\rho(\sigma)-1),
	\end{aligned}
\end{equation}
subject to the initial data $	\sigma\rvert_{t=0}=\sigma(\rho_{0})$, $u\rvert_{t=0}=u_{0}$. The authors establish local and global energy estimates in Sobolev spaces (see,\cite[Lemma 3.1, Lemma 3.2 and Theorem 2.2]{chenalign}), under the assumption of initial data in $H^{s}$ with $s>\frac{d}{2}+\max\{1,2\alpha\}$.

For Euler pressureless systems, it has been shown that the alignment non-linearity enhances the dissipation in the sense that globally regular solutions exist (see, \cite{dokiselev}, \cite{tadmorsh} and \cite{shvy}). 

In \cite{danchin}, Danchin et al., renormalize \eqref{eq:pressurelessEuler} by introducing the variable $\sigma=\rho-1$, to obtain
\begin{equation}
	\label{eq:danchinrenorm}
	\begin{aligned}
		\sigma_{t}+\nabla\cdot(\sigma v)&=-\nabla\cdot u,\\
		v_{t}+c_{\alpha}^{-1}\Lambda^{2\alpha}v&=I_{\alpha}(v,\sigma)-c_{\alpha}^{-1}\sigma\Lambda^{2\alpha}v-(v\cdot\nabla)v,
	\end{aligned}
\end{equation}
with 
\begin{equation}
	\label{eq:danchinnonlinear}
	I_{\alpha}(v,\sigma)-c_{\alpha}^{-1}\sigma\Lambda^{2\alpha}v=[\Lambda^{2\alpha},v]\sigma.
\end{equation}
By assuming initial data $(\sigma_{0},v_{0})$ so that $(\sigma_{0},\nabla\sigma_{0})\in \dot{B}_{d,1}^{1}$ and $(v_{0},\nabla v_{0})\in\dot{B}^{2-2\alpha}_{d,1}$ together with the smallness assumption
\begin{align}
	\|v_{0}\|_{\dot{B}_{d,1}^{2-\alpha}}+\|\sigma_{0}\|_{\dot{B}_{d,1}^{1}}<\epsilon\quad\mbox{for some}~\epsilon>0,\label{eq:danchinsmallness}
\end{align}
the authors show the existence and uniqueness of small global-in-time solutions for the Cauchy problem for \eqref{eq:danchinrenorm}. Moreover, if condition \eqref{eq:danchinsmallness} is fulfilled only by $\sigma_{0}$, they showed local existence of solutions with large initial data for the velocity. Although the smallness condition \eqref{eq:danchinsmallness} is at the critical level of regularity, one more derivative in the initial data allows the authors to control the convergence and uniqueness estimates (see the discussion in \cite[p.92]{danchin}). As they explained, the hyperbolic nature of the continuity equation in \eqref{eq:danchinrenorm} and the nonlinear nonlocal term \eqref{eq:danchinnonlinear}, prevent the stability estimates (i.e., estimates on the difference between two solutions) to be closed for too low regularity. This is true even in the local well-posedness setting. More precisely, the authors construct a sequence of global approximate solutions $(\sigma^{n},v^{n})$ of \eqref{eq:danchinrenorm}. Then, by imposing the condition $\|\sigma_{0}\|_{\dot{B}_{d,1}^{1}}\leq\eta$ (for small enough $\eta>0$), they show that $(\sigma^{n},v^{n})$ converges to some $(\sigma,v)$ in the Banach space 
\begin{align*}
	\mathcal{C}([0,T];\dot{B}_{d,1}^{1})\times\left(\mathcal{C}([0,T];\dot{B}_{d,1}^{2-\alpha}\cap L_{T}^{1}\dot{B}_{p,1}^{2})\right)
\end{align*}
for some $T>0$ with $\|\sigma\|_{L_{T}^{\infty}\dot{B}_{d,1}^{\infty}}\leq\eta$. Furthermore, if $(\overline{\sigma},\overline{v})$ is another solution of \eqref{eq:danchinrenorm} on $[0,T]$, the smallness of $\sigma$ then allows to close the a priori estimates for $(\overline{\sigma}-\sigma,\overline{v}-v)$ and conclude uniqueness (see, \cite[Step 8]{danchin}).

In this work we study the local Cauchy problem \eqref{eq:quasilinearsystem1} within the framework of symmetric hyperbolic systems in Sobolev spaces. Our goal is to provide a general local existence result in the space $\mathcal{C}([0,T];H^{s})\cap X_{s,\alpha,T}^{\infty}$, where
\begin{align*}
	X_{s,\alpha,T}^{\infty}:=\{U=(u,v)^{\top}\in L^{\infty}([0,T];H^{s})~|~\Lambda^{\alpha}v\in L^{2}([0,T];H^{s})\}
\end{align*}
and $s\in\mathbb{N}$ is sufficiently large, in particular $s>\frac{d}{2}+1$. Our analysis includes Euler velocity alignment systems with pressure, like \eqref{eq:Eulervelal}. Nonetheless, it does not include pressureless Euler alignment systems like \eqref{eq:pressurelessEuler} and \eqref{eq:danchinrenorm}, since these type of systems do not have a Friedrichs symmetrizer. 

We find that the fractional commutator term induces a loss of regularity in the stability estimates. In general, the typical iterative scheme developed to approximate the solutions of a standard hyperbolic-parabolic quasilinear system of equations (see \cite{adogbo}, \cite{felipecontraction}, \cite{kawa} and \cite{serre2}) also exhibits a loss of regularity in the convergence estimates. Nevertheless, contraction-type estimates can be established in \emph{lower norms} that yield the existence of solutions. When this last property is combined with uniform a priori estimates in \emph{higher norms}, one recovers the required regularity for the solution. 

However, our analysis reveals that, the presence of the fractional commutator prevents the solution map from satisfying contraction type estimates in the norm of the space 
\begin{align*}
	X_{m-1,\alpha,T}^{\infty}:=\{U=(u,v)^{\top}\in L^{\infty}([0,T];H^{m-1})~|~\Lambda^{\alpha}v\in L^{2}([0,T];H^{m-1})\},
\end{align*}
for any integer $1\leq m\leq s-1$ (see the discussion prior to theorem \ref{holdercontinuity}). Consequently, the system \eqref{eq:quasilinearsystem1} cannot be solved by means of contractive iterations, unless additional assumptions are imposed.  

To overcome this difficulty, we establish the relative compactness of the solution operator in the weaker topology induced by the space 
\begin{align*}
Y_{s-1,\alpha,T}=\{U=(u,v)^{\top}\in L^{2}([0,T];H^{s-1})~|~\Lambda^{\alpha}v\in L^{2}([0,T];H^{s-1})\}.
\end{align*}
This is achieved by combining the Aubin-Lions lemma with a tail control argument (i.e. a tightness of the $L^{2}$-norm) to recover compactness on the whole space $\mathbb{R}^{d}$ (see lemma \ref{compactness}). Furthermore, in theorem \ref{holdercontinuity}, we establish a H\"older continuity estimate of order $1/2$ for the solution operator with respect to the norm of $Y_{s-1,\alpha,T}$. Then, by Schauder's fixed point theorem, we show the local existence of solutions for large initial data (see section \ref{fixedlocal}).

In the presence of commutators, the contraction inequality required for uniqueness can only be obtained under additional assumptions, such as smallness of the initial data in $H^{s}$ (see theorem \ref{localwellposedness}). In fact, to establish uniqueness we adopt a similar strategy to that employed in \cite{danchin}. Namely, we use the smallness assumption to absorb the high order terms into the coercive contributions of the energy estimates.

It is important to highlight that the loss of regularity in the a priori estimates is only present when estimating the difference between two solutions. For general symmetric systems of the form \eqref{eq:quasilinearsystem1} (including the Euler alignment system \eqref{eq:Eulervelal}), the term
\begin{align*}
	\mathcal{A}_{\alpha}(U):=B(U)\Lambda^{2\alpha}U+H(U)[\Lambda^{2\alpha},G(U)]\mathcal{L}U,
\end{align*}
has a coercive contribution to the a priori estimates for \eqref{eq:quasilinearsystem1} under suitable assumptions for the coefficients (see Condition \ref{cond1}). More precisely, it holds that
\begin{align}
	\langle\mathcal{A}_{\alpha}U,U\rangle_{H^{m}}\geq c_{0}\|\Lambda^{\alpha}v\|_{H^{m}}^{2}-\omega\|U\|_{H^{m}}^{2},\quad m\in[0,s]\cap\mathbb{N}_{0},\label{eq:coercivezero}
\end{align}
for some positive constant $c_{0}$ and a non-negative function $\omega\in L^{1}([0,T])$ depending on the coefficients.

In \cite{chenalign}, for example, Chen et al. prove estimate \eqref{eq:coercivezero} for the nonlocal terms in the Euler alignment system \eqref{eq:Eulertantong} by using the so called \emph{cancellation property}. Hence, the contribution of the commutator term to the energy is dominated by the linear dissipation. It is easy to verify, however, that this cancellation mechanism is highly sensitive to changes in the coefficient structure of the equation and thus requires a precise combination of constants in the coefficients of the nonlocal terms in \eqref{eq:Eulertantong}.

In theorem \ref{fractionalellipticity} we prove the coercive estimate \eqref{eq:coercivezero} for the general system \eqref{eq:quasilinearsystem1} as a consequence of commutator estimates developed in \ref{mainFcommutator} and \ref{fractionalcommutatorfeli}. Our approach does not rely on the cancellation property used by Chen et al. \cite{chenalign}, and therefore, it can be applied to general systems with variable coefficients.

Instead, in section \ref{commutatorsfeli} (see theorems \ref{mainFcommutator}, \ref{fractionalcommutatorfeli} and \ref{secondFcommest}),  we develop an alternative method based on the \emph{Jacobi identity} applied to Fourier multiplier operators $T_{1}$ and $T_{2}$ (see \cite[Chapter 3]{liehall}), namely
\begin{align}
	T_{1}[T_{2},f]g=T_{2}[T_{1},f]g+[T_{2},f]T_{1}g-[T_{1},f]T_{2}g.\label{eq:maincommutatoridea} 
\end{align}
The key advantage of this approach is that it allows differential operators acting on Fourier multiplier commutators to be decomposed into terms that can be controlled using lower order commutator estimates for the operators involved. We apply \eqref{eq:maincommutatoridea} to several operators, including the non-homogeneous dyadic block $\Delta_{j}$, the Friedrichs' mollifier $\mathbb{J}_{\epsilon}$, differential operators $D^{\gamma}$, and the fractional Laplacian $\Lambda^{\sigma}$ (see section \ref{commutatorsfeli}). In the latter case, we use different refined Kato-Ponce type inequalities derived by D. Li in \cite{katoponceli}. 

Our theory is not limited to Euler aligment systems with pressure. We consider two general cases for \eqref{eq:quasilinearsystem1}: We assume that $\mathcal{L}$ is either the identity $\mathbb{I}_{N}$ or 
\begin{align}
	\label{eq:diffusionmatrix}
	\mathcal{L}_{N}:=\left(\begin{array}{cc}
		\mathbb{O}_{k\times k}&\mathbb{O}_{k\times m}\\
		\mathcal{L}_{0}&\mathbb{O}_{m\times m}
	\end{array}\right),
\end{align}
where $\mathcal{L}_{0}$ is a constant $m\times k$ matrix. The choice $\mathcal{L}=\mathcal{L}_{N}$ is designed to capture the structure of the Euler alignment system \eqref{eq:Eulervelal}, where the nonlocal conmmutator acts only through the hyperbolic variables. In section \ref{alignment section}, the reader will find the matrix coefficients that allow system \eqref{eq:Eulervelal} to be written in the form \eqref{eq:quasilinearsystem1}. By contrast, the choice $\mathcal{L}=\mathbb{I}_{N}$ corresponds to a situation in which both the fractional dissipation and the commutator act directly on the diffusive variable $v$. This situation applies to a system with the form
\begin{align*}
	\widetilde{A}^{0}(U)U_{t}+\widetilde{A}^{j}(U)\partial_{j}U+\widetilde{D}(U)U+\Lambda^{2\alpha}(\widetilde{B}(v)U)=0,
\end{align*}
where,
\begin{align*}
	\Lambda^{2\alpha}(\widetilde{B}(v)U)=\widetilde{B}(v)\Lambda^{2\alpha}U+[\Lambda^{2\alpha},\widetilde{B}(v)]U.
\end{align*}

In the case $\mathcal{L}=\mathcal{L}_{N}$, we require initial data in $H^{s}$ with $s>\frac{d}{2}+\max\{1,2\alpha\}$, as in \cite{chenalign}. On the other hand, when $\mathcal{L}=\mathbb{I}_{N}$, the regularity assumptions can be relaxed to $s>\frac{d}{2}+1$, independently of $\alpha\in(0,1)$. This is a consequence of the commutator acting only on the variable that experiences the regularizing effect of the fractional diffusion.

\section{Plan of the paper}
In section \ref{commutatorsfeli} we establish the basic commutators estimates that allow to recover the coercivity of the nonlocal term \eqref{eq:coercivezero}. In particular, estimate \ref{secondFcommest} is the key ingredient in proving the energy estimates for solutions of \eqref{eq:quasilinearsystem1}. In section \ref{matrixcond} we lay the basic structural assumptions on the matrix coefficients of \eqref{eq:quasilinearsystem1} that yield the energy estimates. Then, in sections \ref{approxlinearsystem}-\ref{approxsolutions} we apply the Friedrichs method to provide a detailed construction of approximate solutions for a linearized version of \eqref{eq:quasilinearsystem1}. In particular, in theorem \ref{mollifiedapprox} we prove energy estimates for solutions of \eqref{eq:quasilinearsystem1}. Then, through sections \ref{nonlinear}-\ref{fixedlocal} we establish the local existence of solutions for \eqref{eq:quasilinearsystem1}. In section \ref{uniqueness}, we show the how to recover uniqueness under the smallness condition \eqref{eq:constantuniqueness}. In section \ref{alignment section}, we apply our results to the Euler alignment system \eqref{eq:Eulervelal}. Finally, in Appendices A-C we left a series of technical lemmas used throughout the paper.

\section{Basic Lemmas}
Throughout the paper, we use the standard notation $\dot{H}^{\sigma}:=\dot{H}^{\sigma}(\mathbb{R}^{d})$ and $H^{\sigma}:= H^{\sigma}(\mathbb{R}^{d})$ for the homogeneous and non-homogeneous Sobolev space, respectively. Let $m\in\mathbb{N}_{0}$. We denote $\widehat{H}^{m}:=\widehat{H}^{m}(\mathbb{R}^{d})$ as the Banach space,
\begin{align*}
\widehat{H}^{m}(\mathbb{R}^{d}):=\{f\in L^{\infty}(\mathbb{R}^{d})~|~\nabla u\in H^{m-1}(\mathbb{R}^{d})\},
\end{align*}
with norm $\|f\|_{\widehat{H}^{m}}=\|f\|_{L^{\infty}}+\|\nabla f\|_{H^{m-1}}$. When $m=0$ we define $\widehat{H}^{0}:=L^{\infty}$ and $\|f\|_{\widehat{H}^{0}}=\|f\|_{L^{\infty}}$.

The following result can be found in \cite{milani}.
\begin{lema}
\label{sobolevprodinteger}
Let $s$ and $r$ be integers. If $s>\frac{d}{2}$ the space $\widehat{H}^{s}$ is an algebra under point by point multiplications. That is, if $f,g\in\widehat{H}^{s}$ their product $fg$ belongs to $\widehat{H}^{s}$ and 
\begin{align*}
\|fg\|_{\widehat{H}^{s}}\leq C\|f\|_{\widehat{H}^{s}}\|g\|_{\widehat{H}^{s}},
\end{align*}
where $C$ is a constant independent of $f$ and $g$. More generally, if $f\in\widehat{H}^{s}$ and $g\in H^{r}$, $0\leq r\leq s$, then $fg\in H^{r}$, and 
\begin{align*}
\|fg\|_{H^{r}}\leq C\|f\|_{\widehat{H}^{s}}\|g\|_{H^{r}},
\end{align*}
where $C$ is a constant independent of $f$ and $g$. 
\end{lema}

Next, we state various cases of the fractional Leibniz rule and commutator estimates for the fractional Laplacian that will be used throughout this work.
\begin{lema}[\cite{katoponceli}]
\label{fractionalprodest}
Let $\sigma>0$. 
\begin{itemize}
\item [(i)] Assume that $1\leq p_{1}, p_{4}\leq\infty$ and $1<p,p_{2}, p_{3}<\infty$ satisfy
\begin{align*}
\frac{1}{p}=\frac{1}{p_{1}}+\frac{1}{p_{2}}=\frac{1}{p_{3}}+\frac{1}{p_{4}}.
\end{align*}
Suppose that $u\in L^{p_{1}}$, $\Lambda^{\sigma}u\in L^{p_{3}}$, $v\in L^{p_{4}}$ and $\Lambda^{\sigma}v\in L^{p_{2}}$. Then $\Lambda^{\sigma}(uv)\in L^{p}$ and there is a universal constant $C>0$ such that
\begin{align}
\|\Lambda^{\sigma}(uv)\|_{L^{p}}\leq C\left(\|u\|_{L^{p_{1}}}\|\Lambda^{\sigma}v\|_{L^{p_{2}}}+\|\Lambda^{\sigma}u\|_{L^{p_{3}}}\|v\|_{L^{p_{4}}}\right)\label{eq:fractionalleibnizrule}
\end{align}
\item [(ii)] Assume that $u,v\in\dot{H}^{\sigma}$. If $\sigma\in(0,1]$ and $p\in(1,\infty)$, there is a positive constant $C$ (depending on $\sigma$), such that 
\begin{align}
\label{eq:fractionalcommutator}
\|[\Lambda^{\sigma},u]v\|_{L^{p}}\leq C\|\Lambda^{s}u\|_{L^{p}}\|v\|_{L^{\infty}}.
\end{align}
If, on the other hand, $\sigma>1$, there exists a positive constant $C$ (depending on $\sigma$) such that 
\begin{align}
\label{eq:fractionalcommutator2}
\|[\Lambda^{\sigma},u]v\|_{L^{p}}\leq C\left(\|D u\|_{L^{p_{1}}}\|\Lambda^{\sigma-1}v\|_{L^{p_{2}}}+\|\Lambda^{\sigma}u\|_{L^{p_{3}}}\|v\|_{L^{p_{4}}}\right).
\end{align}
\end{itemize}
\end{lema}
The commutator estimate \eqref{eq:fractionalcommutator} follows from \cite[Theorem 1.2]{katoponceli}. In particular, see \cite[Eq. (1.8)]{katoponceli}. The following corollary is a direct consequence of the fractional Leibniz rule \eqref{eq:fractionalleibnizrule} and we use it explicitly throughout the text.
\begin{coro}
\label{leibnizcorollary}
Let $p=2$ and $\sigma>0$ in lemma \ref{fractionalprodest}.
\begin{itemize}
\item [(a)] If $p_{1}=p_{4}=\infty$ and $p_{2}=p_{3}=2$ in \eqref{eq:fractionalleibnizrule}, we have 
\begin{align}
	\label{eq:fractionalproduct2}
	\|uv\|_{\dot{H}^{\sigma}}\leq C\left(\|u\|_{L^{\infty}}\|v\|_{\dot{H}^{\sigma}}+\|u\|_{\dot{H}^{\sigma}}\|v\|_{L^{\infty}}\right).
\end{align}
\item [(b)] If $p_{1}=\infty$, $p_{2}=2$, $p_{3}=p_{4}=4$ in \eqref{eq:fractionalleibnizrule} and $\Lambda^{\sigma}u\in H^{1}$ and $v\in H^{1}$, we have
\begin{align}
\|uv\|_{\dot{H}^{\sigma}}\leq C\left(\|u\|_{L^{\infty}}\|\Lambda^{\sigma}v\|_{L^{2}}+\|\Lambda^{\sigma}u\|_{H^{1}}\|v\|_{H^{1}}\right).\label{eq:leibemb}
\end{align}
\end{itemize}
\end{coro}
\begin{proof}
The proof of $(a)$ follows immediately from \eqref{eq:fractionalleibnizrule}. To prove statement $(b)$ we first apply \eqref{eq:fractionalleibnizrule}. Then, we use the embedding $H^{1}\hookrightarrow L^{4}$ if $d=3$ and $H^{1/2}\hookrightarrow L^{4}$ if $d=2$, in order to conclude. 
\end{proof}

Let $\sigma>0$. For any pair of given functions $f$ and $g$ we define its Kato-Ponce-Li commutator as
\begin{align}
\mathbb{K}_{\sigma}(f,g):=\Lambda^{\sigma}(fg)-\sum_{|\beta|\leq\sigma}\frac{1}{\beta!}D^{\beta}f\Lambda^{\sigma,\beta}g,\label{eq:KatoPonceLicommutator}
\end{align}
where, for any multi-index $\beta$, the operator $\Lambda^{\sigma,\beta}$ is defined via Fourier transform as 
\begin{align*}
	\widehat{\Lambda^{\sigma,\beta}g}(\xi)=\widehat{\Lambda^{\sigma,\beta}}(\xi)\widehat{g}(\xi)\quad\mbox{with}\quad\widehat{\Lambda^{\sigma,\beta}}(\xi):=i^{-|\beta|}D^{\beta}_{\xi}(|\xi|^{\sigma}).
\end{align*}
Observe that if $\sigma\in(0,1)$ we have that $\mathbb{K}_{\sigma}(f,g)=[\Lambda^{\sigma},f]g$. The following result can be found in \cite[Theorem 1.2, eq. (1.6)]{katoponceli}.
\begin{lema}
\label{dongli}
Let $\sigma>0$ and $1<p<\infty$. Then, there is a positive constant $C=C(d,\sigma)$ such that, for any $f,g\in\mathcal{S}$, 
\begin{align}
\|\mathbb{K}_{\sigma}(f,g)\|_{L^{2}}\leq C\|\Lambda^{\sigma}f\|_{BMO}\|g\|_{L^{2}},\label{eq:dongliC1}
\end{align}
In particular, if $\sigma\in(0,1]$, it holds that
\begin{align}
\|[\Lambda^{\sigma},f]g\|_{L^{2}}\leq C\|\Lambda^{\sigma}f\|_{L^{\infty}}\|g\|_{L^{2}}. \label{eq:dongliC2}
\end{align}
\end{lema}	
The following results provide chain rule estimates in Sobolev spaces. Lemma \ref{chainruleestimates} can be found in \cite{kawa} and \cite{milani}. Meanwhile, the fractional chain rule estimates in lemma \ref{fractionalchainrulesest} can be found stated in \cite{chenalign} and \cite{wanginitial}. See also, \cite{christ} and \cite{leetan}.
\begin{lema}
\label{chainruleestimates}
Let $s\geq 1$ be an integer and assume that $v=(v_{1},...,v_{N})\in\widehat{H}^{s}$. Let $F=F(v)$ be a $\mathcal{C}^{\infty}$-function of $v\in\mathbb{R}^{N}$. Then for $1\leq j\leq s$, we have $D_{x}F(v)\in H^{j-1}$ and 
\begin{align}
\label{eq:crest}
\|D_{x}F(v)\|_{H^{j-1}}\leq CM(1+\|v\|_{L^{\infty}})^{j-1}\|D_{x}v\|_{H^{j-1}},
\end{align}
where $C$ is a positive constant and $M:=\sum_{1\leq k\leq j}\sup\{|D_{v}^{k}F(v)|~|~|v|\leq\lambda=\|v\|_{L^{\infty}}\}$. 
\end{lema}
	
\begin{lema}
\label{fractionalchainrulesest}
Let $s>0$. Suppose that $g\in H^{s}\cap L^{\infty}$ and $f\in\mathcal{C}^{[s]}(\operatorname{Range}(g))$. Then $f(g)\in H^{s}\cap L^{\infty}$ and there is a constant $C>0$ depending on $s$ and $\|g\|_{L^{\infty}}$, such that 
\begin{align}
\label{eq:fractionalchain1}
\|D_{x}^{s}f(g)\|_{L^{2}}\leq C\|f\|_{\mathcal{C}^{[s]}}\|D^{s}g\|_{L^{2}}.
\end{align}
In particular, for $s\in(0,1]$,
\begin{align*}
\|\Lambda^{s}f(g)\|_{L^{2}}\leq C\|f\|_{\mathcal{C}^{[s]}}\|\Lambda^{s}g\|_{L^{2}},
\end{align*}
where the constant $C$ depends only on $s$ and $\|g\|_{L^{\infty}}$. 
\end{lema}
Finally, we use the non-homogeneous Littlewood-Paley decomposition as in \cite[Chapter 2]{bahouri}. More precisely, if $\Delta_{j}$ denotes the non-homogeneous dyadic block operator and $u\in\mathcal{S}^{\prime}$, we have that
\begin{align*}
u=\sum_{j\geq -1}\Delta_{j}u\quad\mbox{in}\quad\mathcal{S}^{\prime}. 
\end{align*}
In the following section we use the fact that $H^{\sigma}$ coincides with the Besov space $B_{2,2}^{\sigma}$ and use the equivalent norm
\begin{align}
\|u\|_{B_{2,2}^{\sigma}}:=\left(\sum_{j\geq -1}2^{2j\sigma}\|\Delta_{j} u\|_{L^{2}}^{2}\right)^{1/2}\quad\mbox{for all}\quad u\in\label{eq:littlewoodpaley} H^{\sigma}. 
\end{align}
\section{Commutator estimates}
\label{commutatorsfeli}
In this section we establish some commutator estimates involving the operators $D^{\gamma}$, $\Lambda^{\sigma}$ and $\mathbb{J}_{\epsilon}=\eta_{\epsilon}\ast\cdot$. 
\begin{theo}
\label{mainFcommutator}
Let $s>\frac{d}{2}+1$ be an integer and $\alpha\in(0,1)$. Assume that $f\in \dot{H}^{\alpha}\cap\widehat{H}^{s}\cap\dot{H}^{s+\alpha}$ and $g\in H^{m}$ for some integer $1\leq m\leq s$. Then, for any multi-index $\gamma$ with $|\gamma|=m$, there is a positive constant $C$, depending only on $\alpha$, $m$, $d$ and $s$ such that,
\begin{align}
\label{eq:Fcommutatorestimate1}
\|[D^{\gamma},f]g\|_{H^{\alpha}}\leq C\left(\|f\|_{\dot{H}^{\alpha}}+\|f\|_{\dot{H}^{s+\alpha}}\right)\|g\|_{H^{m}}.
\end{align}
\end{theo}	
\begin{proof}
First, we apply \eqref{eq:maincommutatoridea} with the non-homogeneous dyadic blocks $\Delta_{j}$ (see,\cite{bahouri}) and the differential operator $D^{\gamma}$ to obtain the decomposition
\begin{align*}
\Delta_{j}\left\lbrace[D^{\gamma},f]g\right\rbrace&=D^{\gamma}\left\lbrace[\Delta_{j},f]g\right\rbrace+[D^{\gamma},f]\Delta_{j} g-[\Delta_{j},f]D^{\gamma}g\\
&=:\mathcal{K}_{1}^{j}+\mathcal{K}_{2}^{j}+\mathcal{K}_{3}^{j}.
	\end{align*}
Thus,
\begin{align*}
2^{j\alpha}\|\Delta_{j}\left\lbrace[D^{\gamma},f]g\right\rbrace\|_{L^{2}}\leq 2^{j\alpha}\|\mathcal{K}_{1}^{j}\|_{L^{2}}+2^{j\alpha}\|\mathcal{K}_{2}^{j}\|_{L^{2}}+2^{j\alpha}\|\mathcal{K}_{3}^{j}\|_{L^{2}}\quad\mbox{for all}\quad j\in\mathbb{Z}.
\end{align*}
By standard commutator estimates (see, \cite[Lemma 2.97]{bahouri} and \cite[Proposition 1.5.10]{milani}), there is a positive constant $C$ such that
\begin{align*}
\|\mathcal{K}_{2}^{j}\|_{L_{2}}\leq C\|\nabla f\|_{L^{\infty}}2^{j(m-1)}\|\Delta_{j}g\|_{L^{2}}\quad\mbox{and}\quad\|\mathcal{K}_{3}^{j}\|_{L^{2}}\leq C2^{-j}\|\nabla f\|_{L^{\infty}}\|D^{\gamma}g\|_{L^{2}}.
\end{align*}
From the assumptions, $\nabla f\in H^{s-1}$ and $s-1>\frac{d}{2}$, we deduce that
	\begin{align}
		\label{eq:felipedecomp1}
		2^{j\alpha}\|\Delta_{j}\left\lbrace[D^{\gamma},f]g\right\rbrace\|_{L^{2}}\leq 2^{j\alpha}\|\mathcal{K}_{1}^{j}\|_{L^{2}}+C\|f\|_{\widehat{H}^{s}}\left(2^{j(m-1+\alpha)}\|\Delta_{j}g\|_{L^{2}}+2^{j(\alpha-1)}\|g\|_{H^{m}}\right).
	\end{align}
In order to estimate $\mathcal{K}_{1}^{j}$ we use the Leibniz rule to decompose it as
\begin{align*}
\mathcal{K}_{1}^{j}=D^{\gamma}\left\lbrace[\Delta_{j},f]g\right\rbrace=[\Delta_{j},f]D^{\gamma}g+\sum_{0<\beta\leq\gamma}C_{\beta,\gamma}\mathcal{M}(\beta,j),
\end{align*}
where $C_{\beta,\gamma}=\frac{\gamma!}{\beta!(\alpha-\beta)!}$ and for each $\beta\in\mathbb{N}_{0}$, with $0<\beta\leq\gamma$,
\begin{align*}
\mathcal{M}(\beta,j):=\Delta_{j}(D^{\beta}fD^{\gamma-\beta}g)-D^{\beta}fD^{\gamma-\beta}\Delta_{j}g.
\end{align*}
	Note that the first term in $\mathcal{K}_{1}^{j}$ and $-\mathcal{K}_{3}^{j}$ coincide, so that 
	\begin{align}
		2^{j\alpha}\|\mathcal{K}_{1}^{j}\|_{L^{2}}\leq C2^{(\alpha-1)j}\|\nabla f\|_{L^{\infty}}\|D^{\gamma}g\|_{L^{2}}+\sum_{0<\beta\leq \gamma}C_{\beta,\gamma}2^{j\alpha}\|\mathcal{M}(\beta,j)\|_{L^{2}}.\label{eq:K1decomp}
	\end{align}
Although the terms within $\mathcal{M}(\beta,j)$ are commutators of the dyadic block operator $\Delta_{j}$, we cannot apply to them the standard commutator estimate (\cite[Lemma 2.97]{bahouri}) since it requires the $L^{\infty}$ control on $D^{|\beta|+1}f$ for $|\beta|\geq 1$. Instead, we use a combination of the fractional Leibniz rule and the Sobolev embedding to estimate these terms. Notice that
\begin{equation}
\label{eq:Fdecomposition2}
\begin{aligned}
\sum_{0<\beta\leq \gamma}C_{\beta,\gamma}2^{j\alpha}\|\mathcal{M}(\beta,j)\|_{L^{2}}\leq\sum_{0<\beta\leq \gamma}C_{\beta,\gamma}\left[2^{j\alpha}\|\Delta_{j}(D^{\beta}fD^{\gamma-\beta}g)\|_{L^{2}}+2^{j\alpha}\|D^{\beta}fD^{\gamma-\beta}\Delta_{j}g\|_{L^{2}}\right]
\end{aligned}
\end{equation}
For the terms with $|\beta|=1$ apply $(b)$ in corollary \ref{leibnizcorollary},
\begin{align*}
\|D^{\beta}fD^{\gamma-\beta}g\|_{\dot{H}^{\alpha}}&\leq C\left(\|\Lambda^{\alpha}D^{\beta}f\|_{H^{1}}\|D^{\gamma-\beta}g\|_{H^{1}}+\|D^{\beta}f\|_{L^{\infty}}\|\Lambda^{\alpha}D^{\gamma-\beta}g\|_{L^{2}}\right)\\
&\leq C\left(\left[\|f\|_{\dot{H}^{\alpha}}+\|f\|_{\dot{H}^{2+\alpha}}\right]\|g\|_{H^{m}}+\|\nabla f\|_{L^{\infty}}\|g\|_{H^{m}}\right)
\end{align*}
and
\begin{align*}
\|D^{\beta}fD^{\gamma-\beta}g\|_{L^{2}}&\leq\|D^{\beta}f\|_{L^{4}}\|D^{\gamma-\beta}g\|_{L^{4}}\leq C\|D^{\beta}f\|_{H^{1}}\|D^{\gamma-\beta}g\|_{H^{1}}\\
&\leq C\left[\|f\|_{\dot{H}^{\alpha}}+\|f\|_{\dot{H}^{2}}\right]\|g\|_{H^{m}}.
\end{align*}
Given that $2+\alpha<3\leq s$, \eqref{eq:littlewoodpaley} gives that
\begin{equation}
\label{eq:felipedecomp3}
\begin{aligned}
\mbox{for all}~\beta\in\mathbb{N}_{0}^{d}~\mbox{with}~|\beta|=1,&\\
\left(\sum_{j\geq -1}2^{2j\alpha}\|\Delta_{j}(D^{\beta}fD^{\gamma-\beta}g)\|_{L^{2}}^{2}\right)^{1/2}&\leq\|D^{\beta}fD^{\gamma-\beta}g\|_{L^{2}}+\|D^{\beta}fD^{\gamma-\beta}g\|_{\dot{H}^{\alpha}}\\
&\leq C\left(\|f\|_{\dot{H}^{\alpha}}+\|f\|_{\dot{H}^{s}}\right)\|g\|_{H^{m}}.
\end{aligned}
\end{equation}
Similarly, by the fractional Leibniz rule with $p_{i}=4$ for all $i=1,2,3,4$ and the embedding $H^{1}\hookrightarrow L^{4}$, it holds that
\begin{equation}
\label{eq:felipedecomp5}
\begin{aligned}
\mbox{for all}~\beta\in\mathbb{N}_{0}^{d}~\mbox{with}~2\leq|\beta|\leq m-1&,\\
\left(\sum_{j\geq -1}2^{2j\alpha}\|\Delta_{j}(D^{\beta}fD^{\gamma-\beta}g)\|_{L^{2}}^{2}\right)^{1/2}&\leq C\left(\|D^{\beta}f\|_{H^{1}}\|D^{\gamma-\beta}g\|_{H^{1}}+\|\Lambda^{\alpha}D^{\beta}f\|_{H^{1}}\|D^{\gamma-\beta}g\|_{H^{1}}\right.\\
&+ \left.\|D^{\beta}f\|_{H^{1}}\|\Lambda^{\alpha}D^{\gamma-\beta}g\|_{H^{1}}\right)\\
&\leq C\left(\|f\|_{\dot{H}^{\alpha}}+\|f\|_{\dot{H}^{m+\alpha}}\right)\|g\|_{H^{m}}.
\end{aligned}
\end{equation}
For the case in which $\beta=\gamma$ in \eqref{eq:Fdecomposition2} we have to distinguish three situations: $|\gamma|=1$, $2\leq |\gamma|\leq s-1$ and $|\gamma|=s$. The fractional Leibniz rule combined with the embedding $H^{1}\hookrightarrow L^{4}$, yields
\begin{align*}
\|\Lambda^{\alpha}(D^{\gamma}fg)\|_{L^{2}}&\leq C\left\lbrace\begin{array}{cc}
\|D^{\gamma}f\|_{L^{\infty}}\|\Lambda^{\alpha}g\|_{L^{2}}+\|\Lambda^{\alpha}D^{\gamma}f\|_{H^{1}}\|g\|_{H^{1}}&\mbox{if}\quad|\gamma|=m=1,\\
\|D^{\gamma}f\|_{H^{1}}\|\Lambda^{\alpha}g\|_{H^{1}}+\|\Lambda^{\alpha}D^{\gamma}f\|_{H^{1}}\|g\|_{H^{1}}&\mbox{if}\quad 2\leq|\gamma|=m\leq s-1.
\end{array}\right.
\end{align*}
If $|\gamma|=m=s$, we use the fractional commutator estimate \eqref{eq:fractionalcommutator} together with the embeddings $H^{s}\hookrightarrow L^{\infty}$ and $H^{s-\alpha}\hookrightarrow L^{\infty}$ to obtain
\begin{align*}
\|\Lambda^{\alpha}(D^{\gamma}fg)\|_{L^{2}}&\leq\|[\Lambda^{\alpha},D^{\gamma}f]g\|_{L^{2}}+\|D^{\gamma}f\Lambda^{\alpha}g\|_{L^{2}}\\
&\leq C[\|f\|_{\dot{H}^{\alpha}}+\|f\|_{\dot{H}^{s+\alpha}}]\|g\|_{H^{s}}.
	\end{align*}
Therefore, for any $1\leq |\gamma|\leq s$, the following estimate holds
\begin{equation}
\label{eq:felipedecomp8}
\begin{aligned}
\left(\sum_{j\geq -1}2^{2j\alpha}\|\Delta_{j}(D^{\gamma}fg)\|_{L^{2}}\right)\leq C[\|f\|_{\dot{H}^{\alpha}}+\|f\|_{\dot{H}^{s+\alpha}}]\|g\|_{H^{m}}.
\end{aligned}
\end{equation}
By applying H\"older's inequality in the second term of \eqref{eq:Fdecomposition2}, we obtain that
\begin{align}
\label{eq:felipedecomp4}
\mbox{for all}~\beta\in\mathbb{N}_{0}^{d}~\mbox{with}~|\beta|=1,\quad\|D^{\beta}fD^{\gamma-\beta}\Delta_{j}g\|_{L^{2}}\leq C\|\nabla f\|_{L^{\infty}}2^{j(m-1)}\|\Delta_{j}g\|_{L^{2}}
\end{align}
and
\begin{equation}
\label{eq:felidecomp6}
\begin{aligned}
\mbox{for all}~\beta\in\mathbb{N}_{0}^{d}~\mbox{with}~&2\leq|\beta|\leq m-1,\\
\|D^{\beta}fD^{\gamma-\beta}\Delta_{j}g\|_{L^{2}}&\leq\|D^{\beta}f\|_{L^{4}}\|D^{\gamma-\beta}\Delta_{j}g\|_{L^{4}}\leq C\|D^{\beta}f\|_{H^{1}}\|D^{\gamma-\beta}\Delta_{j}g\|_{H^{1}}\\
&\leq C\left[\|f\|_{\dot{H}^{\alpha}}+\|f\|_{\dot{H}^{s+\alpha}}\right]2^{j(m-1)}\|\Delta_{j}g\|_{L^{2}}.
\end{aligned}
\end{equation}
Similarly, if $\beta=\gamma$ we have
\begin{align*}
\|D^{\gamma}f\Delta_{j}g\|_{L^{2}}\leq C\left\lbrace\begin{array}{cc}
\|\nabla f\|_{L^{\infty}}\|\Delta_{j}g\|_{L^{2}}&\mbox{if}\quad|\gamma|=m=1\\
\|D^{\gamma}f\|_{H^{1}}\|\Delta_{j}g\|_{H^{1}}&\mbox{if}\quad2\leq|\gamma|=m\leq s-1,\\
\|D^{\gamma}f\|_{L^{2}}\|\Delta_{j}g\|_{L^{\infty}}&\quad m=|\gamma|=s.
\end{array}\right.
\end{align*}
Since $s>\frac{d}{2}+1$, we deduce that, for any $1\leq|\gamma|\leq s$, 
\begin{align}
\label{eq:felipedecomp7}
\|D^{\gamma}f\Delta_{j}g\|_{L^{2}}\leq C\|f\|_{\dot{H}^{m}}\|\Delta_{j}g\|_{H^{m-1}}\leq C\|f\|_{\dot{H}^{m}}2^{j(m-1)}\|\Delta_{j}g\|_{L^{2}}.
\end{align}
Combining \eqref{eq:felipedecomp4}, \eqref{eq:felidecomp6} and \eqref{eq:felipedecomp7} shows that
\begin{align*}
\sum_{0<\beta\leq \gamma}C_{\beta,\gamma}2^{j\alpha}\|D^{\beta}fD^{\gamma-\beta}\Delta_{j}g\|_{L^{2}}\leq C\left(\|f\|_{\dot{H}^{\alpha}}+\|f\|_{\dot{H}^{s+\alpha}}\right)2^{j(m-1+\alpha)}\|\Delta_{j}g\|_{L^{2}}.
\end{align*}
Applying this estimate in \eqref{eq:Fdecomposition2} and using the resulting inequality in \eqref{eq:K1decomp}, we infer that 
\begin{equation}
\label{eq:Fdecomposition3}
\begin{aligned}
2^{j\alpha}\|\mathcal{K}_{1}^{j}\|_{L^{2}}&\leq C 2^{(\alpha-1)j}\|\nabla f\|_{L^{\infty}}\|D^{\gamma}g\|_{L^{2}}+\sum_{0<\beta\leq \gamma}C_{\beta,\gamma}2^{j\alpha}\|\Delta_{j}(D^{\beta}fD^{\gamma-\beta}g)\|_{L^{2}}\\
&+C\left(\|f\|_{\dot{H}^{\alpha}}+\|f\|_{\dot{H}^{s+\alpha}}\right)2^{j(m-1+\alpha)}\|\Delta_{j}g\|_{L^{2}}.
\end{aligned}
\end{equation}
By virtue of \eqref{eq:Fdecomposition3}, estimate \eqref{eq:felipedecomp1} yields 
\begin{align*}
2^{j\alpha}\|\Delta_{j}\left\lbrace[D^{\gamma},f]g\right\rbrace\|_{L^{2}}&\leq C\left(\|f\|_{\dot{H}^{\alpha}}+\|f\|_{\dot{H}^{s+\alpha}}\right)\left[2^{j(\alpha-1)}\|g\|_{H^{m}}+2^{j(m-1+\alpha)}\|\Delta_{j}g\|_{L^{2}}\right]\\
&+\sum_{0<\beta\leq \gamma}C_{\beta,\gamma}2^{j\alpha}\|\Delta_{j}(D^{\beta}fD^{\gamma-\beta}g)\|_{L^{2}}.
\end{align*}
Finally, since $\alpha\in(0,1)$, estimates \eqref{eq:felipedecomp3}, \eqref{eq:felipedecomp5} and \eqref{eq:felipedecomp8} imply that
\begin{align*}
&\left(\sum_{j\geq -1}2^{2j\alpha}\|\Delta_{j}\left\lbrace[D^{\gamma},f]g\right\rbrace\|_{L^{2}}^{2}\right)^{1/2}\\
&\leq C\left(\|f\|_{\dot{H}^{\alpha}}+\|f\|_{\dot{H}^{s+\alpha}}\right)\left[\left(\sum_{j\geq -1}2^{2j(\alpha-1)}\right)^{1/2}\|g\|_{H^{m}}+\left(\sum_{j\geq -1}2^{2j(m-1+\alpha)}\|\Delta_{j}g\|_{L^{2}}^{2}\right)^{1/2}\right]\\
&\leq C \left(\|f\|_{\dot{H}^{\alpha}}+\|f\|_{\dot{H}^{s+\alpha}}\right)\|g\|_{H^{m}},
\end{align*}
	where the constant $C$ depends on $\alpha$, $d$, $m$ and $s$. We conclude the result. 
\end{proof}	
\begin{theo}
\label{fractionalcommutatorfeli}
Let $s>\frac{d}{2}+1$ be an integer, $\gamma\in\mathbb{N}_{0}^{d}$ with $|\gamma|=m\in[1,s]$ and $\alpha\in(0,1)$ be given.  Assume that $Q\in\dot{H}^{\alpha}\cap\dot{H}^{s+\alpha}\cap\widehat{H}^{s}$ is a $p\times p$ matrix-valued function and that $w\in H^{m}$, $h\in H^{\alpha}$ are $\mathbb{R}^{p}$-valued vector fields. Then, there is a positive constant $C$, depending only on $\alpha,m,d$ and $s$ such that the following statements hold:
\begin{itemize}
\item [(I1)] If $w\in H^{m+\alpha}$,
\begin{align}
\label{eq:secondFeliCommEst}
|\left\langle D^{\gamma}\left\lbrace[\Lambda^{2\alpha},Q]w\right\rbrace,h\right\rangle_{L^{2}}|\leq C\left(\|Q\|_{\dot{H}^{\alpha}}+\|Q\|_{\dot{H}^{s+\alpha}}\right)\left(\|w\|_{H^{m}}\|h\|_{H^{\alpha}}+\|\Lambda ^{\alpha}w\|_{H^{m}}\|h\|_{L^{2}}\right).
\end{align}
\item [(I2)] If $s>\frac{d}{2}+s_{\alpha}$ with $s_{\alpha}=\max\{1,2\alpha\}$ and $w\in H^{m}$, then
\begin{equation}
\label{eq:nocancellation1}
|\left\langle D^{\gamma}\left\lbrace[\Lambda^{2\alpha},Q]w\right\rbrace,h\right\rangle_{L^{2}}|\leq C(\|Q\|_{\dot{H}^{\alpha}}+\|Q\|_{\dot{H}^{s+\alpha}})\|w\|_{H^{m}}\|h\|_{H^{\alpha}}.
\end{equation}
	\end{itemize}
\end{theo}

\begin{proof}
We begin by applying \eqref{eq:maincommutatoridea} to obtain the identity
\begin{align*}
D^{\gamma}\left\lbrace[\Lambda^{2\alpha},Q]w\right\rbrace=\Lambda^{2\alpha}\left\lbrace[D^{\gamma},Q]w\right\rbrace+[\Lambda^{2\alpha},Q]D^{\gamma}w-[D^{\gamma},Q]\Lambda^{2\alpha}w.
\end{align*}
Taking the inner product in $L^{2}$ and using the self-adjointness of $\Lambda^{\alpha}$, yields the following decomposition 
\begin{align*}
\left\langle D^{\gamma}\left\lbrace[\Lambda^{2\alpha},Q]w\right\rbrace,h\right\rangle_{L^{2}}&=\left\langle\Lambda^{\alpha}\left\lbrace[D^{\gamma},Q]w\right\rbrace,\Lambda^{\alpha}h\right\rangle_{L^{2}}+\left\langle[\Lambda^{2\alpha},Q]D^{\gamma}w,h\right\rangle_{L^{2}}-\left\langle[D^{\gamma},Q]\Lambda^{2\alpha}w,h\right\rangle_{L^{2}}\\
&=:I_{1}+I_{2}+I_{3}.
\end{align*}
In all cases, estimate \eqref{eq:Fcommutatorestimate1} implies that, 
\begin{align*}
|I_{1}|\leq&C\left(\|Q\|_{\dot{H}^{\alpha}}+\|Q\|_{\dot{H}^{s+\alpha}}\right)\|w\|_{H^{m}}\|h\|_{H^{\alpha}}.
\end{align*}
For $I_{2}$ we use the self-adjointness of $\Lambda^{2\alpha}$ in order to have 
\begin{align*}
|I_{2}|\leq\|D^{\gamma}w\|_{L^{2}}\|[\Lambda^{2\alpha},Q^{\top}]h\|_{L^{2}}.
\end{align*}
Assume that $\alpha\in\left(0,\frac{1}{2}\right]$. By setting $\sigma=2\alpha\leq 1$ in lemma \ref{dongli}, we obtain the existence of a constant $C=C(\alpha)>0$ such that,
\begin{align}
\|[\Lambda^{2\alpha},Q^{\top}]h\|_{L^{2}}\leq C\|\Lambda^{2\alpha}Q\|_{L^{\infty}}\|h\|_{L^{2}}.\label{eq:dongli1}
\end{align}
On the other hand, for $\alpha\in[\tfrac{1}{2},1)$, we use \eqref{eq:KatoPonceLicommutator} to write 
\begin{align*}
[\Lambda^{2\alpha},Q^{\top}]h=\mathbb{K}_{2\alpha}(Q^{\top},h)+\sum_{0<|\beta|\leq 2\alpha}\frac{1}{\beta!}D^{\beta}Q^{\top}\Lambda^{2\alpha,\beta}h
\end{align*}
so that, by applying the Kato-Ponce-Li inequality in Lemma \ref{dongli},
\begin{equation}
\label{eq:dongli2}
\begin{aligned}
\|[\Lambda^{2\alpha},Q^{\top}]h\|_{L^{2}}&\leq C\left[\|\Lambda^{2\alpha}Q\|_{BMO}\|h\|_{L^{2}}+\|\nabla Q\|_{L^{\infty}}\|\Lambda^{2\alpha-1}h\|_{L^{2}}\right].
\end{aligned}
\end{equation}
By the embedding $L^{1}_{loc}(\mathbb{R}^{d})\cap\dot{H}^{d/2}\hookrightarrow BMO$ and the assumption $s>\frac{d}{2}+1$, there is a constant $C>0$ such that
\begin{align*}
\|\Lambda^{2\alpha}Q\|_{BMO}\leq C\|\Lambda^{2\alpha}Q\|_{\dot{H}^{d/2}}\leq C\|Q\|_{\dot{H}^{2\alpha+\frac{d}{2}}}\leq C\left(\|Q\|_{\dot{H}^{\alpha}}+\|Q\|_{\dot{H}^{s+\alpha}}\right).
	\end{align*}
As a result, \eqref{eq:dongli1} and \eqref{eq:dongli2}, imply that
\begin{align*}
\|[\Lambda^{2\alpha},Q^{\top}]h\|_{L^{2}}\leq C\left(\|Q\|_{\dot{H}^{\alpha}}+\|Q\|_{\dot{H}^{s+\alpha}}\right)\|\Lambda^{\alpha}h\|_{L^{2}}.
\end{align*}
Hence,
\begin{align*}
|I_{2}|\leq C(\|Q\|_{\dot{H}^{\alpha}}+\|Q\|_{\dot{H}^{s+\alpha}})\|w\|_{H^{m}}\|h\|_{H^{\alpha}}.
\end{align*}
It remains to estimate $I_{3}$. Apply the standard commutator estimates for $D^{\gamma}$ (\cite[Proposition 1.5.10]{milani}) in $I_{3}$ to obtain
\begin{align*}
|I_{3}|\leq C\|\nabla Q\|_{L^{\infty}}\|\Lambda^{2\alpha}w\|_{H^{m-1}}\|h\|_{L^{2}}.
\end{align*}
Then, the embedding $H^{s-1}\hookrightarrow L^{\infty}$ yields,
\begin{align*}
|I_{3}|\leq C\|\nabla Q\|_{H^{s-1}}\left\lbrace\begin{array}{cc}
\|w\|_{H^{m}}\|h\|_{L^{2}}&0<\alpha\leq\frac{1}{2},\\
\|\Lambda^{\alpha}w\|_{H^{m}}\|h\|_{L^{2}}&\frac{1}{2}<\alpha<1.
\end{array}\right.
\end{align*}
After combining the estimates for $I_{1}$, $I_{2}$ and $I_{3}$, statement (I1) follows for arbitrary $\alpha\in(0,1)$ and (I2) for any $\alpha\in(0,\tfrac{1}{2}]$.
	
Now assume that $\alpha\in(\tfrac{1}{2},1)$ and $s>\frac{d}{2}+2\alpha$. By the Leibniz rule, 
\begin{align*}
I_{3}=-\sum_{0<\beta\leq\gamma}C_{\beta,\gamma}\left\langle D^{\beta}QD^{\gamma-\beta}\Lambda^{2\alpha}w,h\right\rangle_{L^{2}}.
\end{align*}
We split this sum according to the size of $|\beta|$:
\begin{align*}
I_{3}=I_{3,1}+I_{3,2}+I_{3,3}
\end{align*}
where
\begin{align*}
I_{3,1}:=-\sum_{\substack{\beta\leq\gamma\\|\beta|=1}}C_{\beta,\gamma}\left\langle D^{\beta}QD^{\gamma-\beta}\Lambda^{2\alpha}w,h\right\rangle_{L^{2}}&,\quad I_{3,2}:=-\sum_{\substack{\beta\leq\gamma\\2\leq|\beta|\leq|\gamma|-1}}C_{\beta,\gamma}\left\langle D^{\beta}QD^{\gamma-\beta}\Lambda^{2\alpha}w,h\right\rangle_{L^{2}}\\
\mbox{and}\quad I_{3,3}&:=-\left\langle D^{\gamma}Q\Lambda^{2\alpha}w,h\right\rangle_{L^{2}}.
\end{align*}
Here and below, empty sums are understood to be zero. By the self-adjointness of $\Lambda^{\alpha}$,
\begin{align*}
I_{3,1}=-\sum_{\substack{\beta\leq\gamma\\|\beta|=1}}C_{\beta,\gamma}\left\langle D^{\gamma-\beta}\Lambda^{\alpha}w,\Lambda^{\alpha}(D^{\beta}Q^{\top}h)\right\rangle_{L^{2}}.
\end{align*}
Thus, by the Cauchy-Schwarz inequality and the commutator estimate in lemma \eqref{eq:dongliC2},
\begin{align*}
|I_{3,1}|&\leq\sum_{\substack{\beta\leq\gamma\\|\beta|=1}}C_{\gamma,\beta}\|D^{\gamma-\beta}\Lambda^{\alpha}w\|_{L^{2}}\left(\|[\Lambda^{\alpha},D^{\beta}Q^{\top}]h\|_{L^{2}}+\|D^{\beta}Q^{\top}\Lambda^{\alpha}h\|_{L^{2}}\right)\\
&\leq C\|w\|_{H^{m}}\left(\|\Lambda^{\alpha}\nabla Q\|_{L^{\infty}}\|h\|_{L^{2}}+\|\nabla Q\|_{L^{\infty}}\|h\|_{H^{\alpha}}\right).
\end{align*}
Since, by hypothesis, $Q\in\dot{H}^{s+\alpha}$ and $\nabla Q\in H^{s-1}$ we infer that $\nabla Q\in H^{s+\alpha-1}$, which in turn implies that $\Lambda^{\alpha}\nabla Q\in H^{s-1}$. Therefore, by the embedding $H^{s-1}\hookrightarrow L^{\infty}$, we obtain
\begin{align*}
\|\Lambda^{\alpha}\nabla Q\|_{L^{\infty}}\leq C\|\Lambda^{\alpha}\nabla Q\|_{H^{s-1}}\leq C(\|Q\|_{\dot{H}^{\alpha}}+\|Q\|_{\dot{H}^{s+\alpha}}).
\end{align*}
Hence, we arrive at the estimate
\begin{align*}
|I_{3,1}|\leq C(\|Q\|_{\dot{H}^{\alpha}}+\|Q\|_{\dot{H}^{s+\alpha}})\|w\|_{H^{m}}\|h\|_{H^{\alpha}}.
\end{align*}
For $I_{3,2}$, we write
\begin{align*}
I_{3,2}:=-\sum_{\substack{\beta\leq\gamma\\2\leq|\beta|\leq|\gamma|-1}}C_{\beta,\gamma}\left\langle D^{\gamma-\beta}\Lambda^{2\alpha}w,D^{\beta}Q^{\top}h\right\rangle_{L^{2}}
\end{align*}
and since, $|\beta|\geq 2$, the Cauchy-Schwarz inequality implies the estimate
\begin{align*}
|I_{3,2}|\leq C\|w\|_{H^{m}}\sum_{\substack{\beta\leq\gamma\\2\leq|\beta|\leq|\gamma|-1}}\|D^{\beta}Q^{\top}h\|_{L^{2}}.
\end{align*}
By applying H\"older's inequality with $p=\tfrac{d}{d-2\alpha}$ and $q=\tfrac{d}{2\alpha}$ we obtain
\begin{align*}
\|D^{\beta}Q^{\top}h\|_{L^{2}}\leq\left(\int_{\mathbb{R}^{d}}|D^{\beta}Q|^{\frac{d}{\alpha}}dx\right)^{\frac{\alpha}{d}}\left(\int_{\mathbb{R}^{d}}|h|^{\frac{2d}{d-2\alpha}}dx\right)^{\frac{d-2\alpha}{2d}}.
\end{align*}
Given that $\alpha\in(\tfrac{1}{2},1)$ and $d=2,3$, it holds that, $2<\frac{d}{\alpha}<6$ and we can apply the continuous embeddings $\dot{H}^{\alpha}\hookrightarrow L^{\frac{2d}{d-2\alpha}}$ and $H^{1}\hookrightarrow L^{\frac{d}{\alpha}}$ to the right-hand side of the last inequality and thus obtain,
\begin{align*}
\|D^{\beta}Q^{\top}h\|_{L^{2}}\leq C\|D^{\beta}Q\|_{H^{1}}\|h\|_{\dot{H}^{\alpha}}.
\end{align*}
Since, in $I_{3,2}$, $|\beta|\leq|\gamma|-1$, we conclude that
\begin{align*}
|I_{3,2}|\leq C(\|Q\|_{\dot{H}^{\alpha}}+\|Q\|_{\dot{H}^{s+\alpha}})\|w\|_{H^{m}}\|h\|_{H^{\alpha}}. 
\end{align*}
For $I_{3,3}$ we have to distinguish three cases, namely, $|\gamma|=1$, $2\leq|\gamma|\leq s-1$ and $|\gamma|=s$. We assume first that $|\gamma|=1$ and proceed as in the estimate for $I_{3,1}$, that is, 
\begin{align*}
|I_{3,3}|=|\left\langle\Lambda^{\alpha}w,\Lambda^{\alpha}(D^{\gamma}Q^{\top}h)\right\rangle_{L^{2}}|&\leq C\|w\|_{H^{1}}\left(\|\Lambda^{\alpha}\nabla Q\|_{L^{\infty}}\|h\|_{L^{2}}+\|\nabla Q\|_{L^{\infty}}\|h\|_{H^{\alpha}}\right)\\
&\leq(\|Q\|_{\dot{H}^{\alpha}}+\|Q\|_{\dot{H}^{s+\alpha}})\|w\|_{H^{1}}\|h\|_{H^{\alpha}}.
\end{align*}
If $2\leq|\gamma|\leq s-1$ we argue as in the estimate for $I_{3,2}$, 
\begin{align*}
|I_{3,3}|=|\left\langle\Lambda^{2\alpha}w,D^{\gamma}Q^{\top}h\right\rangle_{L^{2}}|&\leq C\|w\|_{H^{2}}\|D^{\gamma}Q\|_{H^{1}}\|h\|_{\dot{H}^{\alpha}}\\
&\leq(\|Q\|_{\dot{H}^{\alpha}}+\|Q\|_{\dot{H}^{s+\alpha}})\|w\|_{H^{m}}\|h\|_{H^{\alpha}}.
\end{align*}
Finally, if $|\gamma|=s$ we use the Cauchy-Schwarz inequality to obtain,
\begin{align*}
|I_{3,3}|\leq\|D^{\gamma}Q\Lambda^{2\alpha}w\|_{L^{2}}\|h\|_{L^{2}}\leq\|D^{\gamma}Q\|_{L^{2}}\|\Lambda^{2\alpha}w\|_{L^{\infty}}\|h\|_{L^{2}}.
\end{align*}
Therefore, by the assumption $s>\frac{d}{2}+2\alpha$, we apply the embedding $H^{s-2\alpha}\hookrightarrow L^{\infty}$ and thus, 
\begin{align*}
|I_{3,3}|\leq(\|Q\|_{\dot{H}^{\alpha}}+\|Q\|_{\dot{H}^{s+\alpha}})\|w\|_{H^{s}}\|h\|_{L^{2}}. 
\end{align*}
Hence, for any $1\leq |\gamma|=m\leq s$, 
\begin{align*}
|I_{3,3}|\leq(\|Q\|_{\dot{H}^{\alpha}}+\|Q\|_{\dot{H}^{s+\alpha}})\|w\|_{H^{m}}\|h\|_{H^{\alpha}}.
\end{align*}
Thus, statement (I3) is satisfied. This concludes the proof. 
\end{proof}
\begin{theo}
\label{secondFcommest}
Let $s>\frac{d}{2}+1$ be an integer, $m\in\mathbb{N}_{0}$ with $m\in[0,s]$ and $\alpha\in(0,1)$ be given. Assume that $Q\in\dot{H}^{\alpha}\cap\dot{H}^{s+\alpha}\cap\widehat{H}^{s}$ is a square matrix of order $p$ and $w\in H^{m}$ is a vector field with range in $\mathbb{R}^{p}$. There is a positive constant $C$ depending only on $d$, $s$ and $\alpha$ such that,
\begin{align*}
\|\Lambda^{\alpha}[\mathbb{J}_{\epsilon},Q]w\|_{H^{m}}\leq C(\|Q\|_{\dot{H}^{\alpha}}+\|Q\|_{\dot{H}^{s+\alpha}})\|w\|_{H^{m}}
\end{align*}
\end{theo}
\begin{proof}
First observe that 
\begin{equation}
\label{eq:usefulidentityHm}
\begin{aligned}
\|\Lambda^{\alpha}[\mathbb{J}_{\epsilon},Q]w\|_{H^{m}}^{2}\leq\sum_{|\gamma|\leq m}\|D^{\gamma}\{[\mathbb{J}_{\epsilon},Q]w\}\|_{H^{\alpha}}^{2}=\sum_{|\gamma|\leq m}\sum_{j\geq -1}2^{2\alpha j}\|\Delta_{j}D^{\gamma}\{[\mathbb{J}_{\epsilon},Q]w\}\|_{L^{2}}^{2}.
\end{aligned}
\end{equation}
Hence, we concentrate on estimating the $L^{2}$ norms of $\Delta_{j}D^{\gamma}\{[\mathbb{J}_{\epsilon},Q]w\}$ for $j\geq -1$ and $\gamma\in\mathbb{N}_{0}^{d}$ with $|\gamma|\leq m$. Let $1\leq|\gamma|\leq m$. We use \eqref{eq:maincommutatoridea} with $T_{1}=D^{\gamma}$ and $T_{2}=\mathbb{J}_{\epsilon}$. Thus, for each $j\geq -1$, we have
\begin{align*}
\Delta_{j}D^{\gamma}\{[\mathbb{J}_{\epsilon},Q]w\}&=\Delta_{j}\{\mathbb{J}_{\epsilon}[D^{\gamma},Q]w\}+\Delta_{j}\{[\mathbb{J}_{\epsilon},Q]D^{\gamma}w\}-\Delta_{j}\{[D^{\gamma},Q]\mathbb{J}_{\epsilon}w\}\\
&=:\mathcal{L}_{1}^{j}+\mathcal{L}_{2}^{j}+\mathcal{L}_{3}^{j}.
\end{align*}
To control the terms $\mathcal{L}_{k}^{j}$, we apply \eqref{eq:maincommutatoridea} once again to obtain a second decomposition,
\begin{align*}
&\mathcal{L}_{1}^{j}=\mathbb{J}_{\epsilon}\left( D^{\gamma}\{[\Delta_{j},Q]w\}+[D^{\gamma},Q]\Delta_{j}w-[\Delta_{j},Q]D^{\gamma}w\right),\\
&\mathcal{L}_{2}^{j}=\mathbb{J}_{\epsilon}\{[\Delta_{j},Q]D^{\gamma}w\}+[\mathbb{J}_{\epsilon},Q]\Delta_{j}D^{\gamma}w-[\Delta_{j},Q]\mathbb{J}_{\epsilon}D^{\gamma}w,\\
&\mathcal{L}_{3}^{j}=D^{\gamma}\{[\Delta_{j},Q]\mathbb{J}_{\epsilon}w\}+[D^{\gamma},Q]\Delta_{j}\mathbb{J}_{\epsilon}w-[\Delta_{j},Q]\mathbb{J}_{\epsilon}D^{\gamma}w.
\end{align*}
The commutator estimates for the operators $D^{\gamma}$ and $\Delta_{j}$ (see, \cite[Lemma 2.97]{bahouri} and \cite[Proposition 1.5.10]{milani}) yield
\begin{align*}
2^{j\alpha}\|\mathcal{L}_{1}^{j}\|_{L^{2}}\leq 2^{j\alpha}\|D^{\gamma}\{[\Delta_{j},Q]w\}\|_{L^{2}}+C\|\nabla Q\|_{L^{\infty}}2^{j(|\gamma|-1+\alpha)}\|\Delta_{j}w\|_{L^{2}}+C2^{j(\alpha-1)}\|\nabla Q\|_{L^{\infty}}\|D^{\gamma}w\|_{L^{2}}.
\end{align*}
Notice that the first term in $\mathcal{L}_{1}^{j}$ coincides with $\mathcal{K}_{1}^{j}$ defined in the proof of Theorem \ref{mainFcommutator} and thus, by \eqref{eq:Fdecomposition3}, 
\begin{align*}
2^{j\alpha}\|\mathcal{L}_{1}^{j}\|_{L^{2}}&\leq C(\|Q\|_{\dot{H}^{\alpha}}+\|Q\|_{\dot{H}^{s+\alpha}})\left(2^{(\alpha-1)j}\|D^{\gamma}w\|_{L^{2}}+2^{j(|\gamma|-1+\alpha)}\|\Delta_{j}w\|_{L^{2}}\right)\\
&+\sum_{0<\beta\leq\gamma}C_{\beta,\gamma}2^{j\alpha}\|\Delta_{j}(D^{\beta}QD^{\gamma-\beta}w)\|_{L^{2}}.
\end{align*}
In view of \eqref{eq:felipedecomp3}, \eqref{eq:felipedecomp5} and \eqref{eq:felipedecomp8}, we obtain
\begin{align*}
\left(\sum_{j\geq -1}2^{2j\alpha}\|\mathcal{L}_{1}^{j}\|_{L^{2}}^{2}\right)^{1/2}\leq C(\|Q\|_{\dot{H}^{\alpha}}+\|Q\|_{\dot{H}^{s+\alpha}})\|w\|_{H^{m}}.
\end{align*}
An analogous argument yields the same bound for $\mathcal{L}_{3}^{j}$.
	
Meanwhile, the commutator estimates for the operator $\Delta_{j}$ (see, \cite[Lemma 2.97]{bahouri}), yield
\begin{align*}
2^{j\alpha}\|\mathcal{L}_{2}^{j}\|_{L^{2}}\leq C(\|Q\|_{\dot{H}^{\alpha}}+\|Q\|_{\dot{H}^{s+\alpha}})2^{(\alpha-1)j}\|D^{\gamma}w\|_{L^{2}}+2^{j\alpha}\|[\mathbb{J}_{\epsilon},Q]\Delta_{j}D^{\gamma}w\|_{L^{2}}.
\end{align*}
Observe that we cannot apply the commutator estimates for the operator $\mathbb{J}_{\epsilon}$ (see, \cite[Section 1.5.4]{milani}) to the remaining term in the last inequality, because after adding them up with respect to $j$, it would yield upper bounds involving $\|w\|_{H^{m+\alpha}}$. Instead, we deal with this term by a duality argument. Indeed, set $h=\Delta_{j}D^{\gamma}w$ and use the self-adjointness of $\mathbb{J}_{\epsilon}$ on $L^{2}$ to show that
\begin{align*}
\langle[\mathbb{J}_{\epsilon},Q]h,\varphi\rangle_{L^{2}}=-\langle h,[\mathbb{J}_{\epsilon},Q^{\top}]\varphi\rangle_{L^{2}}\quad\mbox{for every}\quad \varphi\in L^{2}.
\end{align*}
Then, by the Cauchy-Schwarz inequality and by \cite[Lemma 1.5.2]{milani}, we have
\begin{align*}
|\langle[\mathbb{J}_{\epsilon},Q]h,\varphi\rangle_{L^{2}}|\leq\|h\|_{H^{-1}}\|[\mathbb{J}_{\epsilon},Q^{\top}]\varphi\|_{H^{1}}\leq C\|\nabla Q\|_{L^{\infty}}\|\varphi\|_{L^{2}}\|h\|_{H^{-1}},
\end{align*}
for some positive constant $C$ independent of $h$, $Q$ and $\varphi$. Hence,
\begin{equation}
\|[\mathbb{J}_{\epsilon},Q]h\|_{L^{2}}=\sup_{\|\varphi\|_{L^{2}}=1}|\langle[\mathbb{J}_{\epsilon},Q]h,\varphi\rangle_{L^{2}}|\leq C\|\nabla Q\|_{L^{\infty}}\|h\|_{H^{-1}}.\label{eq:dualityargument}
\end{equation}
We find
\begin{align*}
2^{j\alpha}\|\mathcal{L}_{2}^{j}\|_{L^{2}}&\leq C(\|Q\|_{\dot{H}^{\alpha}}+\|Q\|_{\dot{H}^{s+\alpha}})\left(2^{(\alpha-1)j}\|D^{\gamma}w\|_{L^{2}}+2^{j\alpha}\|\Delta_{j}D^{\gamma}w\|_{H^{-1}}\right)\\
&\leq C(\|Q\|_{\dot{H}^{\alpha}}+\|Q\|_{\dot{H}^{s+\alpha}})\left(2^{(\alpha-1)j}\|w\|_{H^{m}}+2^{jm}\|\Delta_{j}w\|_{L^{2}}\right).
\end{align*}
Therefore,
\begin{align*}
\left(\sum_{j\geq -1}2^{2j\alpha}\|\mathcal{L}_{2}^{j}\|_{L^{2}}^{2}\right)^{1/2}\leq C(\|Q\|_{\dot{H}^{\alpha}}+\|Q\|_{\dot{H}^{s+\alpha}})\|w\|_{H^{m}}.
\end{align*}
If $\gamma=0$ we use \eqref{eq:maincommutatoridea} with $T_{1}=\Delta_{j}$, $T_{2}=\mathbb{J}_{\epsilon}$ and proceed as previously to obtain the same type of bounds.

After collecting all the estimates for $\mathcal{L}_{1}^{j}$, $\mathcal{L}_{2}^{j}$, $\mathcal{L}_{3}^{j}$ and using \eqref{eq:usefulidentityHm} we conclude the result. 
\end{proof}

\section{Basic assumptions for linearized fractional systems}
\label{matrixcond}
In this section we establish the basic structural assumptions for the following linearized system of equations, 
\begin{equation}
\label{eq:linearizedsystem1}
A^{0}(\widetilde{U})\partial_{t}U+A^{j}(\widetilde{U})\partial_{j}U+D(\widetilde{U})U+B(\widetilde{U})\Lambda^{2\alpha}U+H(\widetilde{U})[\Lambda^{2\alpha},G(\widetilde{U})]\mathcal{L}U=F,
\end{equation}
The initial data is prescribed at $t=0$,
\begin{align}
\label{eq:linearizedinitialdata}
U\rvert_{t=0}=U_{0}.
\end{align}
Here, $U$ denotes the unknown variable and is such that $U=U(x,t)\in\mathbb{R}^{N}$ for all $(x,t)\in\mathbb{R}^{d}\times[0,T]$, for some $T>0$. We make the following assumptions for the linear system \eqref{eq:linearizedsystem1}.
\begin{cond}
\label{cond1}
Let $\alpha\in(0,1)$ be fixed but arbitrary.
\begin{itemize}
\item [(C1)] The functions $A^{0}(\widetilde{U})$, $A^{j}(\widetilde{U})$, $D(\widetilde{U})$, $B(\widetilde{U})$, $H(\widetilde{U})$ and $G(\widetilde{U})$ are smooth in $\widetilde{U}\in\mathcal{O}$.
\item [(C2)] $A^{0}(\widetilde{U})$ is real symmetric and positive definite for $\widetilde{U}\in\mathcal{O}$. 
\item [(C3)] $A^{j}(\widetilde{U})$ are real symmetric for each $\widetilde{U}\in\mathcal{O}$. 
\item [(C4)] There is a partition in the vector functions $\widetilde{U}\in\mathcal{O}$ and $U\in\mathbb{R}^{N}$ so that, 
\begin{align*}
\widetilde{U}(x,t)=(\widetilde{u},\widetilde{v})^{\top}(x,t),\quad\mbox{and}\quad U(x,t)=(u,v)^{\top}(x,t),
\end{align*}
with $\widetilde{u}(x,t),u(x,t)\in\mathbb{R}^{k}$, $\widetilde{v}(x,t),v(x,t)\in\mathbb{R}^{m}$ and $N=k+m$, for any $(x,t)\in\mathbb{R}^{d}\times[0,T]$.
\item [(C5)] For each $\widetilde{U}\in\mathcal{O}$, $B(\widetilde{U})\geq 0$. Moreover, $B(\widetilde{U})$ is of the form,
\begin{equation}
\label{eq:BS1}
B(\widetilde{U})=\left(\begin{array}{cc}
	\mathbb{O}_{k\times k}&\mathbb{O}_{k\times m}\\
	\mathbb{O}_{m\times k}& B_{0}(\widetilde{U})
\end{array}\right),
\end{equation}
where $B_{0}(\widetilde{U})$ is a real symmetric positive definite matrix of order $m\times m$ for all $\widetilde{U}\in\mathcal{O}$ and there is a positive constant $\mu_{0}>0$ such that,
\begin{align*}
\mbox{for all}\quad v\in H^{\alpha},\quad\langle B_{0}(\widetilde{U})\Lambda^{2\alpha}v,v\rangle_{L^{2}}\geq\mu_{0}\|\Lambda^{\alpha}v\|_{L^{2}}^{2}.
\end{align*} 
\item [(C6)] There are matrices $H_{0}$ and $G_{0}$ of order $m\times m$ that are smooth functions only on $\widetilde{v}\in\mathbb{R}^{m}$ such that for any $\widetilde{U}\in\mathcal{O}$, 
\begin{equation}
\label{eq:BS2}
H(\widetilde{U})=\left(\begin{array}{cc}
	\mathbb{O}_{k\times k}&\mathbb{O}_{k\times m}\\
	\mathbb{O}_{m\times k}&H_{0}(\widetilde{v})
\end{array}\right)\quad\mbox{and}\quad G(\widetilde{U})=\left(\begin{array}{cc}
	\mathbb{O}_{k\times k}&\mathbb{O}_{k\times m}\\
	\mathbb{O}_{m\times k}&G_{0}(\widetilde{v})
\end{array}\right)
\end{equation}
\end{itemize}
\end{cond}
\begin{cond}
\label{cond2}
Let $s>\frac{d}{2}+1$ be an integer and $\alpha\in(0,1)$ be arbitrary. For  $(\widetilde{u},\widetilde{v})=(\widetilde{u},\widetilde{v})(x,t)$, given functions on $Q_{T}$, we assume the following conditions.
\begin{itemize}
\item [(A1)] $(\widetilde{u},\widetilde{v})\in\mathcal{C}([0,T];H^{s})$, $\widetilde{v}\in L^{2}([0,T];H^{s+\alpha})$.
\item [(A2)] $\partial_{t}\widetilde{u}\in\mathcal{C}([0,T];H^{s-1})$. If $\alpha\in[\frac{1}{2},1]$, we assume that $\partial_{t}\widetilde{v}\in\mathcal{C}([0,T];H^{s-2\alpha})\cap L^{2}([0,T];H^{s-1})$ and, if on the other hand, $\alpha\in(0,\frac{1}{2})$, we assume that $\partial_{t}\widetilde{v}\in\mathcal{C}([0,T]; H^{s-1})\cap L^{2}([0,T]; H^{s-1})$.
\item [(A3)] $(\widetilde{u},\widetilde{v})(x,t)\in\mathcal{O}_{1}$ for any $(x,t)\in Q_{T}$, where $\mathcal{O}_{1}$ is an open bounded convex set in $\mathbb{R}^{N}$ satisfying $\overline{\mathcal{O}_{1}}\subset\mathcal{O}$. 
\end{itemize}
\end{cond}

\begin{rem}
\label{remark1}
Observe that, for  $\widetilde{U}=(\widetilde{u},\widetilde{v})^{\top}$ satisfying condition \ref{cond2}, there is a constant $M>0$ such that
\begin{align}
\|\widetilde{U}(t)\|_{H^{s}}\leq M\quad\mbox{for all}\quad t\in[0,T].\label{eq:M0}
\end{align}
Furthermore, by conditions (A2) and (C3) there are positive constants $a_{0}=a_{0}(\mathcal{O}_{1})$ and $a_{1}=a_{1}(\mathcal{O}_{1})$ such that, 
\begin{align*}
\mbox{for all}\quad t\in[0,T],\quad a_{0}|V|^{2}\leq(A^{0}(\widetilde{U}(t))V,V)_{\mathbb{R}^{N}}\leq a_{1}|V|^{2}.
\end{align*}
\end{rem}

In the following sections we assume that the constant matrix $\mathcal{L}$ is either the identity matrix $\mathbb{I}_{N}$ or the matrix $\mathcal{L}$ defined in \eqref{eq:diffusionmatrix}. In the second case, for every  $U=(u,v)^{\top}\in\mathbb{R}^{N}$ with $u\in\mathbb{R}^{k}$ and $v\in\mathbb{R}^{m}$, we have
\begin{equation}
\label{eq:diffusionvector}
\mathcal{L}_{N}U=\left(\begin{array}{c}
\mathbb{O}_{k\times 1}\\
\mathcal{L}_{0}u
\end{array}\right)\in\mathbb{R}^{N}.
\end{equation} 
In addition to condition \ref{cond1}, one of the following hypotheses is assumed:
\begin{itemize}
\item [(E1)] $\mathcal{L}=\mathbb{I}_{N}$ and $s>\frac{d}{2}+1$.
\item [(E2)] $\mathcal{L}=\mathcal{L}_{N}$ and $s>\frac{d}{2}+s_{\alpha}$ with $s_{\alpha}=\max\{1,2\alpha\}$. Furthermore, if $\alpha\in(\tfrac{1}{2},1)$, we assume $A^{0}(\widetilde{U})$ has the block structure
\begin{align*}
A^{0}(\widetilde{U})=\left(\begin{array}{cc}
A^{0}_{1}(\widetilde{U})&\mathbb{O}_{k\times m}\\
\mathbb{O}_{m\times k}&A^{0}_{2}(\widetilde{v})
\end{array}\right),
\end{align*}
where $A^{0}_{1}(\widetilde{U})$ and $A^{0}_{2}(\widetilde{v})$ are real symmetric, positive definite matrices of order $k\times k$ and $m\times m$ for all $\widetilde{U}=(\widetilde{u},\widetilde{v})\in\mathcal{O}$. 
\end{itemize}

\subsection{Preliminary estimates}
For convenience, we collect the following estimates, which will be used repeatedly throughout the remainder of the paper. Let $A=A(\widetilde{U})$ denote any of the matrix coefficients in \eqref{eq:linearizedsystem1} and assume that $\widetilde{U}$ satisfies condition \ref{cond2} and the uniform bound \eqref{eq:M0}. 

First note that, since $s-1>\frac{d}{2}$, the Sobolev embedding theorem combined with \eqref{eq:M0}, yield $\|\widetilde{U}(t)\|_{L^{\infty}}\leq\kappa_{d}M$ for all $t\in[0,T]$. Moreover, by the smoothness of $A$, $A(\widetilde{U})\in L^{\infty}(\mathbb{R}^{d}\times[0,T])$ and  since $\widetilde{U}\in\mathcal{O}_{1}$, there is a positive constant $c_{A}(\mathcal{O}_{1})$, such that $\|A(\widetilde{U})\|_{L^{\infty}}\leq c_{A}(\mathcal{O}_{1})$
\begin{align}
\|A(\widetilde{U})\|_{L^{\infty}}\leq c_{A}(\mathcal{O}_{1}).\label{eq:firstuseful}
\end{align}
As a consequence, the chain rule estimate \eqref{eq:crest} implies the existence of a positive constant $C$, for which
\begin{align*}
\|\nabla^{x}A(\widetilde{U})\|_{H^{m-1}}\leq C\|A\|_{C^{m}(\overline{\mathcal{O}_{1}})}(1+M)^{m-1}\|\nabla\widetilde{U}\|_{H^{m-1}}.
\end{align*}
Hence, there is a positive constant $C_{A}=C_{A}(\mathcal{O}_{1},M)$, such that
\begin{align*}
\|A(\widetilde{U}(t))\|_{\widehat{H}^{m}}\leq c_{A}(\mathcal{O}_{1})+C_{A}(\mathcal{O}_{1},M)M\quad\mbox{for all}\quad t\in[0,T].
\end{align*}
Define 
\begin{align*}
C(\mathcal{O}_{1},M):=\sup_{A} c_{A}(\mathcal{O}_{1})+C_{A}(\mathcal{O}_{1},M)M,
\end{align*}
where the supremum is taken over the set of all matrix coefficients $A$ in \eqref{eq:linearizedsystem1}. Similarly, for every $\sigma\in[0,s]$ we use the fractional chain rule estimates in \eqref{eq:fractionalchain1} to ensure the existence of a constant $C=C(\mathcal{O}_{1})$ such that 
\begin{align}
\|A(\widetilde{U}(t))\|_{\dot{H}^{\sigma}}\leq C(\mathcal{O}_{1})\|\widetilde{U}(t)\|_{H^{s}}\quad\mbox{for all}\quad t\in[0,T].\label{eq:goodestimate}
\end{align}
Let $G=G(\widetilde{U})$ be as in condition \ref{cond1} and let $\widetilde{U}=(\widetilde{u},\widetilde{v})^{\top}$ satisfy condition \ref{cond2}. By the chain rule estimates \eqref{eq:fractionalchain1} together with condition (C6), we have
\begin{equation}
	\label{eq:Gchainestimates}
	\begin{aligned}
		\|G(\widetilde{U})\|_{\dot{H}^{\alpha}}+\|G(\widetilde{U})\|_{\dot{H}^{s+\alpha}}&\leq \|G_{0}(\widetilde{v})\|_{\dot{H}^{\alpha}}+\|G_{0}(\widetilde{v})\|_{\dot{H}^{s+\alpha}}\\
		&\leq C(\mathcal{O}_{1})\left(\|\widetilde{U}\|_{H^{s}}+\|\widetilde{v}\|_{H^{s+\alpha}}\right).
	\end{aligned}
\end{equation}
These constants may vary from line to line throughout the proofs. Whenever convenient, they will be denoted by $C_{i}(\mathcal{O}_{1},M)$.
\section{Approximate linearized system}
\label{approxlinearsystem}
In this section we construct approximations to the solutions of  \eqref{eq:linearizedsystem1}-\eqref{eq:linearizedinitialdata} by means of the Friedrichs' mollifier, namely
\begin{align*}
\mathbb{J}_{\epsilon}u=\mathcal{F}^{-1}\left(\widehat{\eta_{\epsilon}}\widehat{u}\right)=\eta_{\epsilon}\ast u.
\end{align*}
Let $\{\epsilon_{n}\}_{n\rightarrow\infty}$ be a sequence of positive real numbers such that, $\epsilon_{n}\rightarrow 0$ as $n\rightarrow\infty$. 
We consider the following approximation of system \eqref{eq:linearizedsystem1}, 
\begin{equation}
\label{eq:LSn}
\begin{aligned}
\partial_{t}U_{n}+(A^{0}(\widetilde{U}))^{-1}\left[\mathbb{J}_{\epsilon_{n}}\left(A^{j}(\widetilde{U})\partial_{j}\mathbb{J}_{\epsilon_{n}}U_{n}\right)\right.&+\mathbb{J}_{\epsilon_{n}}\left(D(\widetilde{U})\mathbb{J}_{\epsilon_{n}}U_{n}\right)+\mathbb{J}_{\epsilon_{n}}\left(B(\widetilde{U})\Lambda^{2\alpha}\mathbb{J}_{\epsilon_{n}}U_{n}\right)\\
&+\left.\mathbb{J}_{\epsilon_{n}}\left(H(\widetilde{U})\mathbb{J}_{\epsilon_{n}}[\Lambda^{2\alpha},G(\widetilde{U})]\mathcal{L}\mathbb{J}_{\epsilon_{n}}U_{n}\right)\right]=(A^{0}(\widetilde{U}))^{-1}\mathbb{J}_{\epsilon_{n}}\left(F\right),
\end{aligned}
\tag{LSn}
\end{equation}
for the unknown $U_{n}=(u_{n},v_{n})^{\top}$, with $u_{n}(x,t)\in\mathbb{R}^{k}$ and $v_{n}(x,t)\in\mathbb{R}^{m}$ for all $(x,t)\in\mathbb{R}^{d}\times[0,T]$ and initial data, 
\begin{align}
\label{eq:LSninitialdata}
U_{n}\rvert_{t=0}=\mathbb{J}_{\epsilon_{n}}U_{0}.
\end{align}
In order to prove the existence and uniqueness of solutions for \eqref{eq:LSn} we apply the following lemma. For a proof, see \cite[Proposition 2.17]{novotny}.
\begin{lema}
\label{odelemma}
Let $L=L(t)$, $t\in[0,T]$ be a family of linear bounded operators in a Banach space $X$ and let $L\in\mathcal{C}\left([0,T];\mathcal{B}(X)\right)$. Then, for any $v_{0}\in X$ and $f\in\mathcal{C}([0,T];X)$ the problem 
\begin{equation}
\label{eq:LSnode}
\begin{aligned}
v^{\prime}(t)&=L(t)v(t)+f(t),\quad t\in(0,T),\\
v(0)&=v_{0},
\end{aligned}
\end{equation}
has a unique solution $v\in\mathcal{C}((0,T);X)\cap C^{1}\left([0,T];X\right)$. This means in particular that the derivative $v^{\prime}(t)$ exists with respect to the norm of $X$ in $(0,T)$ and the equation in \eqref{eq:LSnode} is satisfied pointwise. 
\end{lema}

\begin{theo}
\label{existenceUniqLSn}
Assume that conditions \ref{cond1} and \ref{cond2} are satisfied and that $F\in\mathcal{C}([0,T];H^{\sigma})$ for some $\sigma\in[0,s]$. Then, there is a unique solution $U_{n}$ of the Cauchy problem \eqref{eq:LSn} with initial data \eqref{eq:LSninitialdata} that belongs to the space $\mathcal{C}([0,T];H^{s})\cap\mathcal{C}^{1}((0,T);H^{s})$ and satisfies the equation pointwise. 
\end{theo}
\begin{proof}
We begin by defining, for each fixed $t\in[0,T]$, the mapping $L(t): H^{s}\rightarrow H^{s}$ as
\begin{align*}
V~\xmapsto{L} ~-(A^{0}(\widetilde{U}))^{-1}\mathbb{J}_{\epsilon_{n}}\left(A^{j}(\widetilde{U})\partial_{j}\mathbb{J}_{\epsilon_{n}}V\right)&-(A^{0}(\widetilde{U}))^{-1}\mathbb{J}_{\epsilon_{n}}\left(D(\widetilde{U})\mathbb{J}_{\epsilon_{n}}V\right)-(A^{0}(\widetilde{U}))^{-1}\mathbb{J}_{\epsilon_{n}}\left(B(\widetilde{U})\Lambda^{2\alpha}\mathbb{J}_{\epsilon_{n}}V\right)\\
&-(A^{0}(\widetilde{U}))^{-1}\mathbb{J}_{\epsilon_{n}}\left(H(\widetilde{U})\mathbb{J}_{\epsilon_{n}}[\Lambda^{2\alpha},G(\widetilde{U})]\mathcal{L}\mathbb{J}_{\epsilon_{n}}V\right).
\end{align*}	
Notice that the mollifier between $H(\widetilde{U})$ and $[\Lambda^{2\alpha},G(\widetilde{U})]\mathcal{L}\mathbb{J}_{\epsilon_{n}}V$ allows one to obtain the bound
\begin{equation}
	\label{eq:LSncontinuity3}
	\begin{aligned}
		&\left\lVert(A^{0}(\widetilde{U}))^{-1}\mathbb{J}_{\epsilon_{n}}\left(H(\widetilde{U})\mathbb{J}_{\epsilon_{n}}[\Lambda^{2\alpha},G(\widetilde{U})]\mathbb{J}_{\epsilon_{n}}\mathcal{L}V\right)\right\rVert_{H^{s}}\\
		&\leq C(\mathcal{O}_{1},M)\left(\|\mathbb{J}_{\epsilon_{n}}\Lambda^{2\alpha}(G(\widetilde{U})\mathbb{J}_{\epsilon_{n}}\mathcal{L}V)\|_{H^{s}}+\|\mathbb{J}_{\epsilon_{n}}(G(\widetilde{U})\Lambda^{2\alpha}\mathbb{J}_{\epsilon_{n}}\mathcal{L}V)\|_{H^{s}}\right)\\
		&\leq C(\mathcal{O}_{1},M)\frac{1}{\epsilon^{2\alpha}_{n}}\|V\|_{H^{s}}.
	\end{aligned}
\end{equation}
Similar bounds are obtained for the rest of the terms and thus, it holds
\begin{align}
	\label{eq:boundedregulop}
	\sup_{t\in[0,T]}\sup_{\substack{V\in H^{s}\\\|V\|_{H^{s}}=1}}\|L(t)V\|_{H^{s}}\leq C(\mathcal{O}_{1},M)\left(\frac{1}{\epsilon_{n}}+\frac{1}{\epsilon_{n}^{2\alpha}}\right).
\end{align}
This shows that $L\in L^{\infty}([0,T];\mathcal{B}(H^{s}))$. Let $t,t_{0}\in[0,T]$ be  such that $t\rightarrow t_{0}$. For any $V\in H^{s}$ with $\|V\|_{H^{s}}=1$, the same estimates that led us to \eqref{eq:LSncontinuity3} yield 
\begin{align*}
	&\|A^{0}(\widetilde{U}(t))L(t)V-A^{0}(\widetilde{U}(t_{0}))L(t_{0})V\|_{H^{s}}\\
	&\leq\sum_{j=1}^{d}\|A^{j}(\widetilde{U}(t))-A^{j}(\widetilde{U}(t_{0}))\|_{\widehat{H}^{s}}\frac{1}{\epsilon_{n}}+\|D(\widetilde{U}(t))-D(\widetilde{U}(t_{0}))\|_{\widehat{H}^{s}}+\|B(\widetilde{U}(t))-B(\widetilde{U}(t_{0}))\|_{\widehat{H}^{s}}\frac{1}{\epsilon_{n}^{2\alpha}}\\
	&+\|H(\widetilde{U}(t))-H(\widetilde{U}(t_{0}))\|_{\widehat{H}^{s}}C(\mathcal{O}_{1},M)\frac{1}{\epsilon_{n}^{2\alpha}}+\|G(\widetilde{U}(t))-G(\widetilde{U}(t_{0}))\|_{\widehat{H}^{s}}C(\mathcal{O}_{1},M)\frac{1}{\epsilon_{n}^{2\alpha}}.
\end{align*}
By the smoothness of the coefficients, condition \ref{cond2} and the Sobolev product estimates, we infer that
\begin{align}
	\forall~V\in H^{s},~\mbox{with}~\|V\|_{H^{s}}=1,\quad\|L(t)V-L(t_{0})V\|_{H^{s}}\rightarrow 0\quad\mbox{as}\quad t\rightarrow t_{0},\label{eq:boundedregulop2}
\end{align}
(see,\cite[Lemma 7.2]{felipecontraction}). Therefore, \eqref{eq:boundedregulop} and \eqref{eq:boundedregulop2} yield that $L\in\mathcal{C}([0,T];\mathcal{B}(H^{s}))$. In the same way, it follows that $(A^{0}(\widetilde{U}))^{-1}\mathbb{J}_{\epsilon_{n}}\left(F\right)\in\mathcal{C}([0,T];H^{s})$.	

As a consequence, we can view \eqref{eq:LSn} as the Cauchy problem \eqref{eq:LSnode} and thus, by lemma \ref{odelemma}, the result follows.
\end{proof}

\section{Energy estimates for fractional systems}
\label{approxsolutions}

In this section we prove the well-posedness of the linear Cauchy problem \eqref{eq:linearizedsystem1}-\eqref{eq:linearizedinitialdata}. For any $\sigma\in[0,s]$ and any $\alpha\in(0,1)$, we introduce the vector space
\begin{align*}
	X_{\sigma,\alpha,T}^{\infty}:=\left\lbrace U=(u,v)\in L^{\infty}([0,T];H^{\sigma})~\rvert~v\in L^{2}([0,	T];H^{\sigma+\alpha})\right\rbrace,
\end{align*}
and for any $U\in X_{\sigma,\alpha,T}^{\infty}$, we define
\begin{align*}
	E^{2}_{\sigma,\alpha}(U(t)):=\|U(t)\|_{H^{\sigma}}^{2}+\int_{0}^{t}\|\Lambda^{\alpha}v(\tau)\|_{H^{\sigma}}^{2}d\tau\quad\mbox{for}\quad t\in[0,T]. 
\end{align*}
Moreover, we define $X_{\sigma,\alpha,T}:=\mathcal{C}([0,T];H^{\sigma})\cap X_{\sigma,\alpha,T}^{\infty}$. Observe that, under the norm, $E_{\sigma,\alpha,T}(U):=\sup_{t\in[0,T]}E_{\sigma,\alpha}(U(t))$, $X_{\sigma,\alpha,T}^{\infty}$ and $X_{\sigma,\alpha,T}$ are Banach spaces.

\begin{defi}
Assume condition \ref{cond1} is satisfied under one of the hypotheses (E1)-(E2). Let $\widetilde{U}$ satisfy condition \ref{cond2}. For a fixed $\alpha\in(0,1)$, we introduce the family of operators $\{\mathcal{A}_{\alpha,\widetilde{U}}^{\gamma}\}_{\gamma\in\mathbb{N}_{0}^{d}}$, defined as follows. Let $\gamma\in\mathbb{N}_{0}^{d}$ and let $U=(u,v)^{\top}$ be such that the expressions below are well defined. For each $t\in[0,T]$, we set
\begin{align*}
(\mathcal{A}_{\alpha,\widetilde{U}}^{\gamma}U)(t)&:=B(\widetilde{U}(t))\Lambda^{2\alpha}D^{\gamma}U(t)+H(\widetilde{U}(t))D^{\gamma}\left\lbrace [\Lambda^{2\alpha},G(\widetilde{U}(t))]\mathcal{L}U(t)\right\rbrace\\
&+[D^{\gamma},H(\widetilde{U}(t))][\Lambda^{2\alpha},G(\widetilde{U}(t))]\mathcal{L}U(t).
\end{align*}
In particular, when $\gamma=0$ we have
\begin{align*}
(\mathcal{A}_{\alpha,\widetilde{U}}U)(t):=B(\widetilde{U}(t))\Lambda^{2\alpha}U(t)+H(\widetilde{U}(t))[\Lambda^{2\alpha},G(\widetilde{U}(t))]\mathcal{L}U(t).
\end{align*}
For the approximated system \eqref{eq:LSn}, we similarly introduce the family of operators $\{\mathcal{A}_{\alpha,\widetilde{U},}^{\gamma,n}\}_{\gamma\in\mathbb{N}_{0}^{d}}$ defined by
\begin{align*}
(\mathcal{A}_{\alpha,\widetilde{U},}^{\gamma,n}U)(t)&:=B(\widetilde{U}(t))\Lambda^{2\alpha}D^{\gamma}\mathbb{J}_{\epsilon_{n}}U(t)+H(\widetilde{U}(t))D^{\gamma}\left\lbrace \mathbb{J}_{\epsilon_{n}}[\Lambda^{2\alpha},G(\widetilde{U}(t))]\mathbb{J}_{\epsilon_{n}}\mathcal{L}U(t)\right\rbrace\\
&+\mathbb{J}_{\epsilon_{n}}\left([D^{\gamma},H(\widetilde{U}(t))]\left\lbrace\mathbb{J}_{\epsilon_{n}}[\Lambda^{2\alpha},G(\widetilde{U}(t))]\mathbb{J}_{\epsilon_{n}}\mathcal{L}U(t)\right\rbrace\right).
\end{align*}
\end{defi}

\begin{theo}[Coercivity]
\label{fractionalellipticity}
Assume condition \ref{cond1} is satisfied with one of the hypotheses (E1)-(E2). Then, there exists  a nonnegative function $\omega\in L^{1}([0,T])$, depending only on $\widetilde{U}$, $M$ and $\mathcal{O}_{1}$, such that for any integer $m\in[1,s]$, the following statements hold.
\begin{itemize}
\item [(i)] If $U\in X_{s+\alpha,\alpha,T}^{\infty}$, 
\begin{equation}
\label{eq:fractionalstrongelliptic}
\begin{aligned}
\sum_{|\gamma|\leq m}\left\langle(\mathcal{A}_{\alpha,\widetilde{U}}^{\gamma,}U)(t),D^{\gamma}U(t)\right\rangle_{L^{2}}\geq \frac{\mu_{0}}{2}\|\Lambda^{\alpha}v(t)\|_{H^{m}}^{2}-\omega(t)\|U(t)\|_{H^{m}}^{2}\quad\mbox{for all}\quad t\in[0,T].
\end{aligned}
\end{equation}
\item [(ii)] If $U\in X_{s,\alpha,T}^{\infty}$, 
\begin{align}
\left\langle\mathcal{A}_{\alpha,\widetilde{U}}U(t),U(t)\right\rangle_{L^{2}}\geq\frac{\mu_{0}}{2}\|\Lambda^{\alpha}v(t)\|_{L^{2}}^{2}-\omega(t)\|U(t)\|_{L^{2}}^{2}.\label{eq:lowcoercivity}
\end{align}
\item [(iii)] If $U\in X_{s,\alpha,T}^{\infty}$, 
\begin{equation}
\label{eq:fractionalstrongellipt}
\begin{aligned}
\sum_{|\gamma|\leq m}\left\langle(\mathcal{A}_{\alpha,\widetilde{U}}^{\gamma,n}\mathbb{J}_{\epsilon_{n}}U)(t),D^{\gamma}\mathbb{J}_{\epsilon_{n}}U(t)\right\rangle_{L^{2}}\geq \frac{\mu_{0}}{2}\|\Lambda^{\alpha}\mathbb{J}_{\epsilon_{n}}v(t)\|_{H^{m}}^{2}-\omega(t)\|\mathbb{J}_{\epsilon_{n}}U(t)\|_{H^{m}}^{2}\quad\mbox{for all}\quad t\in[0,T].
\end{aligned}
\end{equation}
\end{itemize}
\end{theo}
\begin{proof}
We first prove (i). Fix an integer $m\in[1,s]$ and let $U\in X_{s+\alpha,\alpha,T}^{\infty}$. Condition (C5) yields the lower bound,
\begin{align*}
\sum_{|\gamma|\leq m}\langle B_{0}\Lambda^{2\alpha}D^{\gamma}v,D^{\gamma}v\rangle_{L^{2}}\geq\mu_{0}\|\Lambda^{\alpha}v\|_{H^{m}}^{2}.
\end{align*}
Therefore, by taking the inner product in $L^{2}$ between $\mathcal{A}_{\alpha,\widetilde{U}}^{\gamma}U$ and $D^{\gamma}U$, we obtain	
\begin{align*}
\sum_{|\gamma|\leq m}\left\langle(\mathcal{A}_{\alpha,\widetilde{U}}^{\gamma}U)(t),D^{\gamma}U(t)\right\rangle_{L^{2}}&\geq \mu_{0}\|\Lambda^{\alpha}v\|_{H^{m}}^{2}-\sum_{|\gamma|\leq m}|\mathcal{C}_{\gamma}|-\sum_{|\gamma|\leq m}|\mathcal{D}_{\gamma}|,
\end{align*}	
where
\begin{align*}
\mathcal{C}_{\gamma}:=\left\langle HD^{\gamma}\left\lbrace[\Lambda^{2\alpha},G]\mathcal{L}U\right\rbrace,D^{\gamma}U\right\rangle_{L^{2}}\quad\mbox{and}\quad \mathcal{D}_{\gamma}:=\left\langle\left([D^{\gamma},H]\left\lbrace[\Lambda^{2\alpha},G]\mathcal{L}U\right\rbrace\right),D^{\gamma}U\right\rangle_{L^{2}}.
\end{align*}
The Cauchy-Schwarz inequality together with the commutator estimates for $D^{\gamma}$, give the estimate
\begin{align*}
|\mathcal{D}_{\gamma}|&\leq\|\nabla H\|_{H^{s-1}}\|[\Lambda^{2\alpha},G]U\|_{H^{m-1}}\|D^{\gamma}U\|_{L^{2}}.
\end{align*}
Combining \eqref{eq:fraccommnon} and \eqref{eq:Gchainestimates} with \eqref{eq:M0}, shows that
\begin{align}
|\mathcal{D}_{\gamma}|&\leq C(\mathcal{O}_{1},M)\left(M+\|\widetilde{v}\|_{H^{s+\alpha
}}\right)\|U\|_{H^{m}}^{2}.\label{eq:dngama}
\end{align}
Now, let $\mathcal{L}U=(y,w)\in\mathbb{R}^{k}\times \mathbb{R}^{m}$. As a consequence of condition (C6), we have
\begin{align*}
\mathcal{C}_{\gamma}=\left\langle D^{\gamma}\{[\Lambda^{2\alpha},G_{0}]w\},H_{0}^{\top}D^{\gamma}v\right\rangle_{L^{2}}.
\end{align*}

Assume condition (E1) is satisfied. Then, $w=v$. Consequently, by applying estimate \eqref{eq:secondFeliCommEst} in Theorem \ref{fractionalcommutatorfeli} with $Q=G_{0}$, $w=v$ and $h=H_{0}^{\top}D^{\gamma}v$, we obtain
\begin{align*}
|\mathcal{C}_{\gamma}|&\leq C\left(\|G_{0}\|_{\dot{H}^{\alpha}}+\|G_{0}\|_{\dot{H}^{s+\alpha}}\right)\left(\|v\|_{H^{m}}\|H_{0}^{\top}D^{\gamma}v_{n}\|_{H^{\alpha}}+\|\Lambda^{\alpha} v\|_{H^{m}}\|H_{0}^{\top}D^{\gamma}v\|_{L^{2}}\right).
\end{align*}
Using the Kato-Ponce-Li inequality \eqref{eq:dongliC2}, together with the estimate \eqref{eq:infinitycontrol} with $\sigma=\alpha$ and \eqref{eq:goodestimate}, leads to the following inequalities
\begin{equation}
\label{eq:h0halpha}
\begin{aligned}
\|H_{0}^{\top}D^{\gamma}v\|_{H^{\alpha}}&\leq\|H_{0}^{\top}D^{\gamma}v\|_{L^{2}}+\|\Lambda^{\alpha}\left(H_{0}^{\top}D^{\gamma}v\right)\|_{L^{2}}\\
&\leq\|H_{0}\|_{L^{\infty}}\|D^{\gamma}v\|_{L^{2}}+\|[\Lambda^{\alpha},H_{0}^{\top}]D^{\gamma}v\|_{L^{2}}+\|H_{0}^{\top}\Lambda^{\alpha}D^{\gamma}v\|_{L^{2}}\\
&\leq C(\mathcal{O}_{1},M)\left(\|D^{\gamma}v\|_{L^{2}}+\|\Lambda^{\alpha}D^{\gamma}v\|_{L^{2}}\right).
\end{aligned}
\end{equation}
Applying the previous estimate together with \eqref{eq:Gchainestimates}, yields
\begin{align}
|\mathcal{C}_{\gamma}|\leq C(\mathcal{O}_{1},M)\left(M+\|\widetilde{v}\|_{H^{s+\alpha}}\right)\left(\|U\|_{H^{m}}^{2}+\|U\|_{H^{m}}\|\Lambda^{\alpha}v\|_{H^{m}}\right).\label{eq:cngamma}
\end{align}
By Young's inequality, for any $\epsilon_{0}>0$ there is a constant $C(\mathcal{O}_{1},M,\epsilon_{0})$ such that
\begin{align*}
|\mathcal{C}_{\gamma}|+|D_{\gamma}|\leq C(\mathcal{O}_{1},M)(M+\|\widetilde{v}\|_{H^{s+\alpha}})\|U\|_{H^{m}}^{2}+C(\mathcal{O}_{1},M,\epsilon_{0})(M^{2}+\|\widetilde{v}\|_{H^{s+\alpha}}^{2})\|U\|_{H^{m}}^{2}+\frac{\epsilon_{0}}{2}\|\Lambda^{\alpha}v\|_{H^{m}}^{2}.
\end{align*}
Therefore, choosing $\epsilon_{0}>0$ sufficiently small, yields \eqref{eq:fractionalstrongelliptic}.
Here, the function $\omega:[0,T]\rightarrow\mathbb{R}_{+}$ is defined by
\begin{equation}
\label{eq:firstomega}
\omega(t):=C_{1}(\mathcal{O}_{1},M)\left(M+M^{2}+\|\widetilde{v}(t)\|_{H^{s+\alpha}}+\|\widetilde{v}(t)\|_{H^{s+\alpha}}^{2}\right),
\end{equation}
for some constant $C_{1}(\mathcal{O}_{1},M)$. Moreover $\omega\in L^{1}([0,T])$ as a consequence of condition \ref{cond2}. This concludes the proof under the assumption (E1).

From \eqref{eq:diffusionvector}, under assumption (E2), we have $w=\mathcal{L}_{0}u$. Thus,
\begin{align*}
\mathcal{C}_{\gamma}:=\left\langle D^{\gamma}\left\lbrace[\Lambda^{2\alpha},G_{0}]\mathcal{L}_{0}u\right\rbrace,H_{0}^{\top}D^{\gamma}v\right\rangle_{L^{2}}
\end{align*}
where $\mathcal{L}_{0}u\in H^{m}$ for any $1\leq m\leq s$, since $\mathcal{L}_{0}$ is constant. By applying estimate \eqref{eq:nocancellation1} in Theorem \ref{fractionalcommutatorfeli}, we obtain
\begin{align*}
|\mathcal{C}_{\gamma}|\leq C\left(\|G_{0}\|_{\dot{H}^{\alpha}}+\|G_{0}\|_{\dot{H}^{s+\alpha}}\right)\|U\|_{H^{m}}\|H_{0}^{\top}D^{\gamma}v\|_{H^{\alpha}}.
\end{align*}
Proceeding as in the previous case and using \eqref{eq:h0halpha}, we recover estimate \eqref{eq:cngamma}. Therefore, \eqref{eq:fractionalstrongelliptic} holds with $\omega$ given by \eqref{eq:firstomega}. 

Now, we prove $(ii)$. Again, we start with the estimate
\begin{align}
\left\langle\mathcal{A}_{\alpha,\widetilde{U}}U(t),U(t)\right\rangle_{L^{2}}\geq\mu_{0}\|\Lambda^{\alpha}v\|_{L^{2}}^{2}-\left\lvert\left\langle H(\widetilde{U}(t))[\Lambda^{2\alpha},G\widetilde{U}(t)]\mathcal{L}U(t),U(t)\right\rangle_{L^{2}}\right\rvert.\label{eq:L2ellipticity1}
\end{align}
Assume first that $\alpha\in(0,\tfrac{1}{2}]$. Lemma \ref{dongli} together with \eqref{eq:firstuseful} and the embedding $H^{s-\alpha}\hookrightarrow L^{\infty}$, imply
\begin{align}
	|\left\langle H[\Lambda^{2\alpha},G]\mathcal{L}U,U\right\rangle_{L^{2}}|&\leq C\|H\|_{L^{\infty}}\|\Lambda^{2\alpha}G\|_{L^{\infty}}\|U\|_{L^{2}}^{2}\leq C(\mathcal{O}_{1})\|G\|_{\dot{H}^{s+\alpha}}\|U\|_{L^{2}}^{2}.\label{eq:L2elliptic1}
\end{align}
Meanwhile, for $\alpha\in(\tfrac{1}{2},1)$ we use the method developed for the term $I_{2}$ during the proof of Theorem \ref{fractionalcommutatorfeli}. Let $\mathcal{L}U=(y,w)\in\mathbb{R}^{k}\times\mathbb{R}^{m}$. First, by condition (C6) and the self-adjointness of $\Lambda^{2\alpha}$, 
\begin{align*}
	\left\langle H[\Lambda^{2\alpha},G]\mathcal{L}U,U\right\rangle_{L^{2}}=\left\langle H_{0}[\Lambda^{2\alpha},G_{0}]w,v\right\rangle_{L^{2}}=-\left\langle w,[\Lambda^{2\alpha},G_{0}^{\top}](H_{0}^{\top}v)\right\rangle_{L^{2}}.
\end{align*}
Thus, by the Cauchy-Schwarz inequality,
\begin{align*}
	|\left\langle H_{0}[\Lambda^{2\alpha},G_{0}]w,v\right\rangle_{L^{2}}|\leq|\mathcal{L}|\|U\|_{L^{2}}\|[\Lambda^{2\alpha},G^{\top}_{0}](H^{\top}_{0}v)\|_{L^{2}}.
\end{align*}
By combining \eqref{eq:KatoPonceLicommutator}, the triangle inequality  and the commutator estimate \eqref{eq:dongliC1}, we deduce that
\begin{align*}
	\|[\Lambda^{2\alpha},G^{\top}_{0}](H^{\top}_{0}v)\|_{L^{2}}&\leq\|\mathbb{K}_{2\alpha}(G^{\top}_{0},H^{\top}_{0}v)\|_{L^{2}}+\left\lVert\sum_{0<\beta\leq 2\alpha}\frac{1}{\beta!}D^{\beta}G_{0}^{\top}\Lambda^{2\alpha,\beta}(H_{0}^{\top}v)\right\rVert_{L^{2}}\\
	&\leq C\left(\|\Lambda^{2\alpha}G\|_{L^{\infty}}\|H_{0}^{\top}v\|_{L^{2}}+\|\nabla G\|_{L^{\infty}}\|\Lambda^{2\alpha-1}(H_{0}^{\top}v)\|_{L^{2}}\right). 
\end{align*}
We use the embeddings $H^{s-\alpha}\hookrightarrow L^{\infty}$ and $H^{s-1}\hookrightarrow L^{\infty}$ to control the terms $\|\Lambda^{2\alpha}G\|_{L^{\infty}}$ and $\|\nabla G\|_{L^{\infty}}$, respectively, and since $\|\Lambda^{2\alpha-1}(H_{0}^{\top}v)\|_{L^{2}}\leq C\|\Lambda^{\alpha}(H_{0}^{\top}v)\|_{L^{2}}$, the commutator estimate \eqref{eq:dongliC2} implies
\begin{align*}
	\|[\Lambda^{2\alpha},G^{\top}_{0}](H^{\top}_{0}v)\|_{L^{2}}&\leq C(\|G_{0}\|_{\dot{H}^{\alpha}}+\|G_{0}\|_{\dot{H}^{s+\alpha}})\left(\|H_{0}\|_{L^{\infty}}\|v\|_{L^{2}}+\|[\Lambda^{\alpha},H_{0}^{\top}]v\|_{L^{2}}+\|H_{0}\|_{L^{\infty}}\|\Lambda^{\alpha}v\|_{L^{2}}\right)\\
	&\leq C(\mathcal{O}_{1},M)(\|G_{0}\|_{\dot{H}^{\alpha}}+\|G\|_{\dot{H}^{s+\alpha}})\|v\|_{H^{\alpha}}. 
\end{align*}
Therefore, by \eqref{eq:Gchainestimates}, 
\begin{align}
	\left\langle |H[\Lambda^{2\alpha},G]\mathcal{L}U,U\right\rangle_{L^{2}}|\leq C(\mathcal{O}_{1},M)(M+\|\widetilde{v}\|_{H^{s+\alpha}})\|U\|_{L^{2}}\|v\|_{H^{\alpha}}.\label{eq:L2elliptic2}
\end{align}

Using \eqref{eq:L2elliptic1} in the case $\alpha\in(0,\tfrac{1}{2}]$ and \eqref{eq:L2elliptic2} in the case $\alpha\in(\tfrac{1}{2},1)$, and then applying Young's inequality in \eqref{eq:L2ellipticity1}, yields \eqref{eq:lowcoercivity} with the same $\omega$ defined in \eqref{eq:firstomega}.

The proof of (iii) is analogous to case $(i)$. The only difference is that the terms $\mathcal{C}_{\gamma}$ and $\mathcal{D}_{\gamma}$ are replaced by
\begin{align*}
\mathcal{C}_{n,\gamma}=\left\langle HD^{\gamma}\mathbb{J}_{\epsilon_{n}}\left\lbrace[\Lambda^{2\alpha},G]\mathbb{J}_{\epsilon_{n}}\mathcal{L}U\right\rbrace,D^{\gamma}\mathbb{J}_{\epsilon_{n}}U\right\rangle_{L^{2}}\quad\mbox{and}\quad \mathcal{D}_{n,\gamma}=\left\langle\mathbb{J}_{\epsilon_{n}}\left([D^{\gamma},H]\{\mathbb{J}_{\epsilon_{n}}[\Lambda^{2\alpha},G]\mathbb{J}_{\epsilon_{n}}\mathcal{L}U\},D^{\gamma}\mathbb{J}_{\epsilon_{n}}U\right)\right\rangle_{L^{2}}.
\end{align*}

The estimate for $\mathcal{D}_{n,\gamma}$ is obtained as in the treatment of $\mathcal{D}_{\gamma}$, yielding \eqref{eq:dngama}. Concerning $\mathcal{C}_{n,\gamma}$, condition (C6) and the self-adjointness of $\mathbb{J}_{\epsilon_{n}}$ on $L^{2}$ imply
\begin{align*}
\mathcal{C}_{n,\gamma}=\left\langle D^{\gamma}[\Lambda^{2\alpha},G_{0}]\mathbb{J}_{\epsilon_{n}}w,\mathbb{J}_{\epsilon_{n}}(H_{0}^{\top}D^{\gamma}\mathbb{J}_{\epsilon_{n}}v)\right\rangle_{L^{2}}.
\end{align*}
As a result, the same argument used in the estimate for $\mathcal{C}_{\gamma}$ yields \eqref{eq:cngamma}. This completes the proof of the theorem.
\end{proof}

\begin{theo}
\label{energyestimatesF}
Assume condition \ref{cond1} is satisfied with one of the hypotheses (E1)-(E2).
Define $\mu_{1}:[0,T]\rightarrow\mathbb{R}_{+}$ as
\begin{align*}
	\mu_{1}(t):=1+M+M^{2}+\|\widetilde{v}(t)\|_{H^{s+\alpha}}+\|\widetilde{v}(t)\|_{H^{s+\alpha}}^{2}+\|\partial_{t}\widetilde{U}(t)\|_{H^{s-1}}
\end{align*}
which, by Condition \ref{cond2}, belongs to $L^{1}([0,T])$. Let $F=(f_{1},f_{2})^{\top}:\mathbb{R}^{d}\times[0,T]\rightarrow\mathbb{R}^{k}\times\mathbb{R}^{m}$ satisfy $f_{1}\in\mathcal{C}([0,T];H^{m-1})\cap L^{2}([0,T];H^{m})$ and $f_{2}\in\mathcal{C}([0,T];H^{m-2\alpha})\cap L^{2}([0,T];H^{m-\alpha})$. Assume moreover that $U_{0}\in H^{m}$. Let $U_{n}=(u_{n},v_{n})\in\mathcal{C}([0,T];H^{s})\cap\mathcal{C}^{1}((0,T);H^{s})$ be the unique solution of the Cauchy problem \eqref{eq:LSn}-\eqref{eq:LSninitialdata}. Then, for every $t\in[0,T]$ the following estimate holds
\begin{equation}
\label{eq:parabolicenergyestimate}
\|U_{n}(t)\|_{H^{m}}^{2}+\int_{0}^{t}\|\Lambda^{\alpha}\mathbb{J}_{\epsilon_{n}}v_{n}(\tau)\|_{H^{m}}^{2}d\tau\leq K(t)\Phi_{\alpha}^{m}(U_{0},F,t),
\end{equation}
where
\begin{align*}
K(t):=C_{0}(\mathcal{O}_{1})e^{C_{2}(\mathcal{O}_{1},M)\int_{0}^{t}\mu_{1}(s)ds},
\end{align*}
for some positive constants $C_{0}=C_{0}(\mathcal{O}_{1})$, $C_{2}(\mathcal{O}_{1},M)$ and 
\begin{align*}
\Phi_{\alpha}^{m}(U_{0},F,t):=\|U_{0}\|_{H^{m}}^{2}+\int_{0}^{t}\left(\|f_{1}(\tau)\|_{H^{m}}^{2}+\|f_{2}(\tau)\|_{H^{m-\alpha}}^{2}\right)d\tau.
\end{align*}
In particular, for every integer $m\in[1,s]$, every $n\in\mathbb{N}$, and every $t\in[0,T]$, the estimate
\begin{align}
\label{eq:fractionalenergyT}
E^{2}_{m,\alpha,T}(\mathbb{J}_{\epsilon_{n}}U_{n}(t))\leq K(t)\Phi_{\alpha}^{m}(U_{0},F,t)
\end{align}
holds.
\end{theo}
\begin{proof}
Let $U_{n}$ be the unique solution to the Cauchy problem \eqref{eq:LSn}. For simplicity, we omit the dependence of the coefficients on $\widetilde{U}$. Let $m\in[1,s]$ be an integer and $\gamma\in\mathbb{N}_{0}^{d}$ satisfy $|\gamma|\leq m$. We apply the operator $D^{\gamma}$ to the linear system \eqref{eq:LSn} to obtain
\begin{equation}
\label{eq:DLSn}
\begin{aligned}
\partial_{t}D^{\gamma}U_{n}&+(A^{0})^{-1}\mathbb{J}_{\epsilon_{n}}(A^{j}\partial_{j}D^{\gamma}\mathbb{J}_{\epsilon_{n}}U_{n})+(A^{0})^{-1}\mathbb{J}_{\epsilon_{n}}(DD^{\gamma}\mathbb{J}_{\epsilon_{n}}U_{n})+(A^{0})^{-1}\mathbb{J}_{\epsilon_{n}}(\mathcal{A}_{\alpha,\widetilde{U}}^{\gamma}\mathbb{J}_{\epsilon_{n}}U_{n})\\
&+(A^{0})^{-1}R_{\gamma,n}^{1}+R_{\gamma,n}^{2}=(A^{0})^{-1}D^{\gamma}\mathbb{J}_{\epsilon_{n}}F+[D^{\gamma},(A^{0})^{-1}]\mathbb{J}_{\epsilon_{n}}F.
\end{aligned}
\end{equation}
where, 
\begin{align*}
(A^{0})^{-1}R_{\gamma,n}^{1}&:=(A^{0})^{-1}\mathbb{J}_{\epsilon_{n}}\left([D^{\gamma},A^{j}]\partial_{j}\mathbb{J}_{\epsilon_{n}}U_{n}\right)+(A^{0})^{-1}\mathbb{J}_{\epsilon_{n}}\left([D^{\gamma},B]\Lambda^{2\alpha}\mathbb{J}_{\epsilon_{n}}U_{n}\right)\\
&+(A^{0})^{-1}\mathbb{J}_{\epsilon_{n}}\left([D^{\gamma},D]\mathbb{J}_{\epsilon_{n}}U_{n}\right),
\end{align*}
\begin{align*}
R_{\gamma,n}^{2}&:=[D^{\gamma},(A^{0})^{-1}]\left\lbrace\mathbb{J}_{\epsilon_{n}}(A^{j}\partial_{j}\mathbb{J}_{\epsilon_{n}}U_{n})\right\rbrace+[D^{\gamma},(A^{0})^{-1}]\left\lbrace\mathbb{J}_{\epsilon_{n}}D\mathbb{J}_{\epsilon_{n}}U_{n}\right\rbrace\\
&+[D^{\gamma},(A^{0})^{-1}]\left\lbrace\mathbb{J}_{\epsilon_{n}}B\Lambda^{2\alpha}\mathbb{J}_{\epsilon_{n}}U_{n}\right\rbrace
\end{align*}
and 
\begin{align*}
R_{\gamma,n}^{3}:=[D^{\gamma},(A^{0})^{-1}]\left\lbrace\mathbb{J}_{\epsilon_{n}}\left(H\mathbb{J}_{\epsilon_{n}}[\Lambda^{2\alpha},G]\mathbb{J}_{\epsilon_{n}}\mathcal{L}U_{n}\right)\right\rbrace.
\end{align*}
We multiply equation \eqref{eq:DLSn} by $A^{0}$ and take $L^{2}$-inner product with $D^{\gamma}U_{n}$. Since $\mathbb{J}_{\epsilon_{n}}$ is self-adjoint on $L^{2}$ and each matrix $A^{j}$ is symmetric, we find
\begin{align*}
\langle\mathbb{J}_{\epsilon_{n}}(A^{j}\partial_{j}D^{\gamma}\mathbb{J}_{\epsilon_{n}}U_{n}),D^{\gamma}U_{n}\rangle_{L^{2}}&=\langle A^{j}\partial_{j}D^{\gamma}\mathbb{J}_{\epsilon_{n}}U_{n},D^{\gamma}\mathbb{J}_{\epsilon_{n}}U_{n}\rangle_{L^{2}}=-\frac{1}{2}\langle(\partial_{j}A^{j})D^{\gamma}\mathbb{J}_{\epsilon_{n}}U_{n},D^{\gamma}\mathbb{J}_{\epsilon_{n}}U_{n}\rangle_{L^{2}},\\
\langle\mathbb{J}_{\epsilon_{n}}(B\Lambda^{2\alpha}D^{\gamma}\mathbb{J}_{\epsilon_{n}}U_{n}),D^{\gamma}U_{n}\rangle_{L^{2}}&=\langle B\Lambda^{2\alpha}D^{\gamma}\mathbb{J}_{\epsilon_{n}}U_{n},D^{\gamma}\mathbb{J}_{\epsilon_{n}}U_{n}\rangle_{L^{2}},\\
\langle\mathbb{J}_{\epsilon_{n}}(D D^{\gamma}\mathbb{J}_{\epsilon_{n}}U_{n}),D^{\gamma}\mathbb{J}_{\epsilon_{n}}U_{n}\rangle_{L^{2}}&=\langle D D^{\gamma}\mathbb{J}_{\epsilon_{n}}U_{n},D^{\gamma}\mathbb{J}_{\epsilon_{n}}U_{n}\rangle_{L^{2}}.
\end{align*}
Together with the definition of  $(\mathcal{A}_{\alpha,\widetilde{U}}^{\gamma}\mathbb{J}_{\epsilon_{n}}U_{n})$ (see Theorem \ref{fractionalellipticity}), this leads to the following identity
\begin{align}
\label{eq:energyidentityfrac1}
\frac{1}{2}\frac{d}{dt}\left\langle A^{0}D^{\gamma}U_{n},D^{\gamma}U_{n}\right\rangle_{L^{2}}+\langle (\mathcal{A}_{\alpha,\widetilde{U}}^{\gamma}\mathbb{J}_{\epsilon_{n}}U_{n}),D^{\gamma}\mathbb{J}_{\epsilon_{n}}v_{n}\rangle_{L^{2}}=\mathcal{A}_{n,\gamma}+\mathcal{B}_{n,\gamma}-\mathcal{I}_{n,\gamma}-\mathcal{J}_{n,\gamma}-\mathcal{M}_{n,\gamma},
\end{align}
where
\begin{align*}
\mathcal{A}_{n,\gamma}&=\left\langle D^{\gamma}F,D^{\gamma}\mathbb{J}_{\epsilon_{n}}U_{n}\right\rangle_{L^{2}}+\left\langle A^{0}[D^{\gamma},(A^{0})^{-1}]\mathbb{J}_{\epsilon_{n}}F,D^{\gamma}U_{n}\right\rangle_{L^{2}},\\
\mathcal{B}_{n,\gamma}&=\frac{1}{2}\langle(\partial_{j}A^{j})D^{\gamma}\mathbb{J}_{\epsilon_{n}}U_{n},D^{\gamma}\mathbb{J}_{\epsilon_{n}}U_{n}\rangle_{L^{2}}-\langle D D^{\gamma}\mathbb{J}_{\epsilon_{n}}U_{n},D^{\gamma}\mathbb{J}_{\epsilon_{n}}U_{n}\rangle_{L^{2}}+\frac{1}{2}\left\langle(\partial_{t}A^{0})D^{\gamma}U_{n},D^{\gamma}U_{n}\right\rangle_{L^{2}},\\
\mathcal{I}_{n,\gamma}&=\left\langle R^{1}_{\gamma,n},D^{\gamma}U_{n}\right\rangle_{L^{2}}+\left\langle A^{0}R^{2}_{\gamma,n},D^{\gamma}U_{n}\right\rangle_{L^{2}},\\
\mathcal{M}_{n,\gamma}&=\left\langle A^{0}R^{3}_{\gamma,n},D^{\gamma}U_{n}\right\rangle_{L^{2}}.
\end{align*}	
Let us estimate each term separately. By Cauchy-Schwarz inequality, the embedding $H^{s-1}\hookrightarrow L^{\infty}$ and \eqref{eq:goodestimate}, there is a positive constant $C(\mathcal{O}_{1})$ such that
\begin{align}
\label{eq:energyEstimate1}
|\mathcal{B}_{n,\gamma}|\leq C(\mathcal{O}_{1})\left(1+\|\partial_{t}\widetilde{U}\|_{H^{s-1}}+M\right)\|D^{\gamma}U_{n}\|_{L^{2}}^{2}.
\end{align}
To estimate $\mathcal{I}_{n,\gamma}$ we apply the commutator estimates for $D^{\gamma}$ to obtain
\begin{align}
\label{eq:energyestimatec}
|\mathcal{I}_{n,\gamma}|\leq C(\mathcal{O}_{1},M)\left(\|U_{n}\|_{H^{m}}^{2}+\|U_{n}\|_{H^{m}}\|\Lambda^{\alpha}\mathbb{J}_{\epsilon_{n}}v_{n}\|_{H^{m}}\right).
\end{align}
Next, by the symmetry of $A^{0}$ in condition (C2), it follows that
\begin{align*}
\mathcal{M}_{n,\gamma}=\left\langle[D^{\gamma},(A^{0})^{-1}]\left\lbrace\mathbb{J}_{\epsilon_{n}}\left(H\mathbb{J}_{\epsilon_{n}}[\Lambda^{2\alpha},G]\mathbb{J}_{\epsilon_{n}}\mathcal{L}	U_{n}\right)\right\rbrace,(A^{0})^{-1}D^{\gamma}\mathbb{J}_{\epsilon_{n}}U_{n}\right\rangle_{L^{2}}.
\end{align*}
This term has the same structure as $\mathcal{D}_{n,\gamma}$ in the proof of Theorem \ref{fractionalellipticity}. Therefore, proceeding exactly as in the estimate for $\mathcal{D}_{n,\gamma}$ (see, \eqref{eq:dngama}), leads to the bound
\begin{align*}
|\mathcal{M}_{n,\gamma}|&\leq C(\mathcal{O}_{1},M)\left(M+\|\widetilde{v}\|_{H^{s+\alpha
}}\right)\|U_{n}\|_{H^{m}}^{2}.
\end{align*}
Applying the coercivity inequality \eqref{eq:fractionalstrongellipt}, together with the estimates for $\mathcal{B}_{n,\gamma}$, $\mathcal{I}_{n,\gamma}$ and $\mathcal{M}_{n,\gamma}$ in \eqref{eq:energyidentityfrac1} and summing over $|\gamma|\in[0,m]$, we obtain the following estimate
\begin{equation}
\label{eq:energyEstimate5}
\begin{aligned}
&\frac{1}{2}\frac{d}{dt}\sum_{|\gamma|=0}^{m}\left\langle A^{0}D^{\gamma}U_{n},D^{\gamma}U_{n}\right\rangle_{L^{2}}+\frac{\mu_{0}}{2}\|\Lambda^{\alpha}\mathbb{J}_{\epsilon_{n}}v_{n}\|_{H^{m}}^{2}-\omega(t)\|U_{n}\|_{H^{m}}^{2}\\
&\leq\sum_{|\gamma|=0}^{m}|\mathcal{A}_{n,\gamma}|+C(\mathcal{O}_{1},M)(1+M+\|\widetilde{v}\|_{H^{s+\alpha}}+\|\partial_{t}\widetilde{U}_{n}\|_{H^{s-1}})\left(\|U_{n}\|_{H^{m}}^{2}+\|U_{n}\|_{H^{m}}\|\Lambda^{\alpha}\mathbb{J}_{\epsilon_{n}}v_{n}\|_{H^{m}}\right)
\end{aligned}
\end{equation}
Since $F=(f_{1},f_{2})^{\top}$, 
\begin{equation}
\label{eq:sourcehyper1}
\sum_{|\gamma|=0}^{m}|\mathcal{A}_{n,\gamma}|\leq C(\mathcal{O}_{1},M) \left(\|f_{1}\|_{H^{m}}\|U_{n}\|_{H^{m}}+\|f_{2}\|_{H^{m-\alpha}}[\|U_{n}\|_{H^{m}}+\|\Lambda^{\alpha}\mathbb{J}_{\epsilon_{n}}v_{n}\|_{H^{m}}]\right).
\end{equation}
Then, using the definition of $\omega$ given in \eqref{eq:firstomega} and applying Young's inequality in the right hand side of \eqref{eq:energyEstimate5} and \eqref{eq:sourcehyper1}, shows that
\begin{equation}
\label{eq:energyEstimate10}
\begin{aligned}
\frac{d}{dt}\sum_{|\gamma|=0}^{m}\left\langle A^{0}D^{\gamma}U_{n},D^{\gamma}U_{n}\right\rangle_{L^{2}}&+\frac{\mu_{0}}{2}\|\Lambda^{\alpha}\mathbb{J}_{\epsilon_{n}}v_{n}\|_{H^{m}}^{2}\\
&\leq C_{1}(\mathcal{O}_{1},M)\left[\|f_{1}\|_{H^{m}}^{2}+\|f_{2}\|_{H^{m-\alpha}}^{2}+\mu_{1}(t)\|U_{n}\|_{H^{m}}^{2}\right]
\end{aligned}
\end{equation}
for some positive constant $C_{1}=C_{1}(\mathcal{O}_{1},M)$. By integrating this inequality in the interval $[0,t]$ with $t\in(0,T]$ and use Remark \ref{remark1}, we can assure the existence of positive constants $a_{0}=a_{0}(\mathcal{O}_{1})$ and $a_{1}=a_{1}(\mathcal{O}_{1})$ such that, 
\begin{align*}
a_{0}\|U_{n}(t)\|_{H^{m}}^{2}&+\frac{\mu_{0}}{2}\int_{0}^{t}\|\Lambda^{\alpha}\mathbb{J}_{\epsilon_{n}}v_{n}(\tau)\|_{H^{m}}^{2}d\tau\\
&\leq a_{1}\|U_{0}\|_{H^{m}}^{2}+C_{1}(\mathcal{O}_{1},M)\int_{0}^{t}\left[\|f_{1}(\tau)\|_{H^{m}}^{2}+\|f_{2}(\tau)\|_{H^{m-\alpha}}^{2}+\mu_{1}(\tau)\|U_{n}(\tau)\|_{H^{m}}^{2}\right]d\tau
\end{align*}
Then, a standard application of Gronwall's inequality yields estimate \eqref{eq:parabolicenergyestimate} with 
\begin{equation}
\label{eq:importantconstant}
C_{0}(\mathcal{O}_{1}):=\frac{a_{1}(\mathcal{O}_{1})}{\min\left\lbrace a_{0}(\mathcal{O}_{1}),\frac{\mu_{0}}{2}\right\rbrace}.
\end{equation}
Finally, estimate \eqref{eq:fractionalenergyT} follows from the standard inequality, $\|\mathbb{J}_{\epsilon}U_{n}(t)\|_{H^{m}}\leq\|U_{n}(t)\|_{H^{m}}$, valid for every $t\in[0,T]$ and $\epsilon>0$.
\end{proof}
The next result establishes the convergence of $\{U_{n}\}$ to a strong solution of \eqref{eq:linearizedsystem1} under one of the assumptions (E1)-(E2). Its proof is standard, for completeness we left it in Appendix C, section \ref{AppendixC}.
\begin{theo}
\label{superapprox}
Assume that $U_{0}\in H^{s}$, $f_{1}\in\mathcal{C}([0,T];H^{s-1})\cap L^{2}([0,T];H^{s})$ and $f_{2}\in\mathcal{C}([0,T];H^{s-2\alpha})\cap L^{2}([0,T];H^{s-\alpha})$. Under one of the assumptions (E1)-(E2), there exists a solution $U=(u,v)$ of \eqref{eq:linearizedsystem1}-\eqref{eq:linearizedinitialdata} in the space $X_{s,\alpha,T}^{\infty}\cap\mathcal{C}([0,T];H^{\sigma})$ with $\partial_{t}u\in\mathcal{C}([0,T];H^{\sigma-1})$ and $\partial_{t}v\in\mathcal{C}([0,T];H^{\sigma-2\alpha})$ for any $\sigma\in[0,s)$. Furthermore, $U$ satisfies the estimate
\begin{align}
E_{m,\alpha}^{2}(U(t))\leq K(t)\Phi_{\alpha}^{m}(U_{0},F,t)\label{eq:estimateforapprox}
\end{align}
for all integers $m\in[1,s]$ and $t\in[0,T]$. If, in addition, $\alpha\in(0,\tfrac{1}{2}]$, then  $\partial_{t}v\in\mathcal{C}([0,T];H^{\sigma-1})$ for every $\sigma\in[0,s)$.
\end{theo}

\begin{theo}
	\label{mollifiedapprox}
	Assume conditions \ref{cond1} is satisfied with one of the hypotheses (E1)-(E2). Let $U\in X_{s,\alpha,T}^{\infty}$ be a solution of the Cauchy problem \eqref{eq:linearizedsystem1}-\eqref{eq:linearizedinitialdata} corresponding to the data
	\begin{align*}
		U_{0}\in H^{s},\quad f_{1}\in \mathcal{C}([0,T];H^{s-1})\cap L^{2}([0,T];H^{s}),\quad f_{2}\in\mathcal{C}([0,T];H^{s-2\alpha})\cap L^{2}([0,T];H^{s-\alpha}).
	\end{align*}
	Then, for any $t\in[0,T]$,
	\begin{align}
		E^{2}_{0,\alpha}(U(t))\leq K(t)\left[\|U_{0}\|_{L^{2}}^{2}+\int_{0}^{t}\|F(\tau)\|_{L^{2}}^{2}d\tau\right]\label{eq:L2energy}
	\end{align}
	and
	\begin{align}
		E_{s,\alpha}^{2}(U(t))\leq K(t)\Phi_{\alpha}^{s}(U_{0},F,t).\label{eq:Highenergy}
	\end{align}
	Furthermore, if $U\in\mathcal{C}([0,T];H^{\sigma})$ for some $\sigma\in[0,s)$, then $U\in\mathcal{C}([0,T];H^{s})$.
\end{theo}		
\begin{proof}
First, observe that estimate \eqref{eq:L2energy} is a direct consequence of the low order coercivity \eqref{eq:lowcoercivity} and the energy method. To avoid repetition, we omit the details.

We now turn to the proof of \eqref{eq:Highenergy}. By applying the operator $\mathbb{J}_{\epsilon}$ to \eqref{eq:linearizedsystem1} we obtain that $\mathbb{J}_{\epsilon}U$ satisfies the following Cauchy problem
	\begin{equation}
		\label{eq:mollifiedlinearsystem}
		\begin{aligned}
			A^{0}\partial_{t}\mathbb{J}_{\epsilon}U+A^{j}\partial_{j}\mathbb{J}_{\epsilon}U+D\mathbb{J}_{\epsilon}U&+B\Lambda^{2\alpha}\mathbb{J}_{\epsilon}U+H[\Lambda^{2\alpha},G]\mathcal{L}\mathbb{J}_{\epsilon}U=A^{0}F_{\epsilon}+A^{0}K_{\epsilon}\\
			\mathbb{J}_{\epsilon}U\rvert_{t=0}&=\mathbb{J}_{\epsilon}U_{0},
		\end{aligned}
	\end{equation}
	with
	\begin{align*}
		F_{\epsilon}&:=(A^{0})^{-1}\mathbb{J}_{\epsilon}F+[\mathbb{J}_{\epsilon},(A^{0})^{-1}]F-[\mathbb{J}_{\epsilon},\overline{A}^{j}]\partial_{j}U-[\mathbb{J}_{\epsilon},\overline{D}]U-[\mathbb{J}_{\epsilon},\overline{B}]\Lambda^{2\alpha}U,\\
		K_{\epsilon}&:=-[\mathbb{J}_{\epsilon},\overline{H}]\Lambda^{2\alpha}(G\mathcal{L}U)-\overline{H}\Lambda^{2\alpha}([\mathbb{J}_{\epsilon},G]\mathcal{L}U)+[\mathbb{J}_{\epsilon},\overline{H}G]\Lambda^{2\alpha}\mathcal{L}U,
	\end{align*}
	and where we have set, $\overline{A}^{j}=(A^{0})^{-1}A^{j}$, $\overline{B}=(A^{0})^{-1}B$, $\overline{D}=(A^{0})^{-1}D$ and $\overline{H}=(A^{0})^{-1}H$. 
	
	Let us show that $K_{\epsilon}\in L^{2}([0,T];H^{s-\alpha})$ for any $\epsilon>0$. Since
\begin{equation}
\label{eq:Kepstimate}
\begin{aligned}
\|K_{\epsilon}\|_{H^{s-\alpha}}&\leq\|[\mathbb{J}_{\epsilon},\overline{H}]\Lambda^{2\alpha}(G\mathcal{L}U)\|_{H^{s-\alpha}}+\|\overline{H}\Lambda^{2\alpha}([\mathbb{J}_{\epsilon},G]\mathcal{L}U)\|_{H^{s-\alpha}}+\|[J_{\epsilon},\overline{H}G]\Lambda^{2\alpha}\mathcal{L}U\|_{H^{s-\alpha}}\\
&=:K_{\epsilon}^{1}+K_{\epsilon}^{2}+K_{\epsilon}^{3},
\end{aligned}
\end{equation}
we treat each term separately. Observe that, by the duality argument \eqref{eq:dualityargument}, 
\begin{equation}
\label{eq:Keps1}
\begin{aligned}
K_{\epsilon}^{1}\leq C\left(\|\overline{H}\|_{\widehat{H}^{s}}\|\Lambda^{2\alpha}(G\mathcal{L}U)\|_{H^{-1}}+\|[\mathbb{J}_{\epsilon},\overline{H}]\Lambda^{2\alpha}(G\mathcal{L}U)\|_{\dot{H}^{s-\alpha}}\right).
\end{aligned}
\end{equation}
	For $K_{\epsilon}^{2}$ we use the fractional Leibniz rule \eqref{eq:fractionalproduct2}, 
	\begin{align*}
		K_{\epsilon}^{2}\leq C\left(\|\overline{H}\|_{L^{\infty}}\|\Lambda^{\alpha}[J_{\epsilon},G]\mathcal{L}U\|_{H^{s}}+\|\overline{H}\|_{\dot{H}^{s-\alpha}}\|\Lambda^{2\alpha}[J_{\epsilon},G]\mathcal{L}U\|_{L^{\infty}}\right).
	\end{align*}
	Then, we apply Theorem \ref{secondFcommest} with $m=s$ for the first term in the right hand side of the last inequality,
	\begin{equation}
		\label{eq:middletermestimate}
		\begin{aligned}
			K_{\epsilon}^{2}\leq C(\mathcal{O}_{1},M)\left[(\|G\|_{\dot{H}^{\alpha}}+\|G\|_{\dot{H}^{s+\alpha}})\|U\|_{H^{s}}+\|\Lambda^{2\alpha}[\mathbb{J}_{\epsilon},G]\mathcal{L}U\|_{L^{\infty}}\right]
		\end{aligned}
	\end{equation}
	and let
	\begin{equation}
		\label{eq:firstK4}
		K_{\epsilon}^{4}:=\|\Lambda^{2\alpha}[\mathbb{J}_{\epsilon},G]\mathcal{L}U\|_{L^{\infty}}.
	\end{equation}

Assume first that hypothesis (E1) holds. Then,
\begin{align*}
G\mathcal{L}U=GU=\left(\begin{array}{c}
\mathbb{O}_{k\times 1}\\
G_{0}v
\end{array}\right)\in L^{2}([0,T];H^{s+\alpha}).
\end{align*}
Therefore, by applying statement $(i)$ in lemma \ref{epsilonlimitest} to the second term in the right hand side of \eqref{eq:Keps1} we obtain 
\begin{align*}
K_{\epsilon}^{1}\leq C\|\overline{H}\|_{\widehat{H}^{s}}\|\Lambda^{2\alpha}(G_{0}v)\|_{H^{s-1}}
	\end{align*}
and by the fractional Leibniz rule,
\begin{equation}
\label{eq:Keps1E1}
K_{\epsilon}^{1}\leq C(\mathcal{O}_{1},M)\left(\|v\|_{H^{s+\alpha}}+\|\widetilde{v}\|_{H^{s+\alpha}}\|U\|_{H^{s}}\right).
\end{equation}
	Furthermore, by the block structure in conditions (C5)-(C6) and the triangle inequality
	\begin{align*}
		K_{\epsilon}^{4}&\leq\|\Lambda^{2\alpha}(G_{0}v)\|_{L^{\infty}}+\|G_{0}\Lambda^{2\alpha}v\|_{L^{\infty}}.
	\end{align*}
	Since $v, G_{0}v\in L^{2}([0,T];H^{s+\alpha})$, it holds that $\Lambda^{2\alpha}v, \Lambda^{2\alpha}(G_{0}v)\in L^{2}([0,T];H^{s-\alpha})$. Therefore, the embedding $H^{s-\alpha}\hookrightarrow L^{\infty}$ and the fractional Leibniz rule imply 
\begin{align*}
K_{\epsilon}^{4}&\leq C(\|G_{0}\|_{L^{\infty}}\|v\|_{H^{s+\alpha}}+\|G_{0}\|_{\dot{H}^{s+\alpha}}\|v\|_{L^{\infty}})\leq C(\mathcal{O}_{1},M)(\|v\|_{H^{s+\alpha}}+\|\widetilde{v}\|_{H^{s+\alpha}}\|U\|_{H^{s}})
\end{align*}
	By using this estimate in \eqref{eq:middletermestimate} we obtain 
	\begin{equation}
		\label{eq:Keps21}
		K_{\epsilon}^{2}\leq C(\mathcal{O}_{1},M)\left[(M+\|\widetilde{v}\|_{H^{s+\alpha}})\|U\|_{H^{s}}+\|v\|_{H^{s+\alpha}}\right].
	\end{equation}
	If, on the other hand, hypothesis (E2) holds, we have
	\begin{align*}
		G\mathcal{L}U=G\mathcal{L}_{N}U=\left(\begin{array}{c}
			\mathbb{O}_{k\times 1}\\
			G_{0}\mathcal{L}_{0}u
		\end{array}\right)\in \mathcal{C}([0,T];H^{s}).
	\end{align*}
	Then, $\Lambda^{2\alpha}G\mathcal{L}_{N}U\in\mathcal{C}([0,T];H^{s-s_{\alpha}})$ and we can apply statement (ii) of lemma \ref{epsilonlimitest} to the second term in the right hand side of \eqref{eq:Keps1} to obtain that
	\begin{equation}
		\label{eq:Keps12}	
		K_{\epsilon}^{1}\leq C\|\overline{H}\|_{\widehat{H}^{s}}\|G_{0}\mathcal{L}_{0}u\|_{H^{s}}\leq C(\mathcal{O}_{1},M)\|U\|_{H^{s}}.
	\end{equation}
	For the term $K_{\epsilon}^{4}$, the embedding $H^{s-s_{\alpha}}\hookrightarrow L^{\infty}$ yields the estimate
\begin{align*}
K_{\epsilon}^{4}&\leq\|\Lambda^{2\alpha}(G_{0}\mathcal{L}_{0}u)\|_{L^{\infty}}+\|G_{0}\Lambda^{2\alpha}\mathcal{L}_{0}u\|_{L^{\infty}}\leq C(\mathcal{O}_{1},M)\|U\|_{H^{s}}. 
\end{align*}
	Hence, by \eqref{eq:middletermestimate}, estimate \eqref{eq:Keps21} remains true. The estimates for $K_{\epsilon}^{3}$ are similar to those for $K_{\epsilon}^{1}$. Therefore, by combining estimates \eqref{eq:Keps1E1}, \eqref{eq:Keps21} and \eqref{eq:Keps12} we conclude that, under any of the assumptions (E1)-(E2), the following estimate holds
	\begin{align}
		K_{\epsilon}\leq C(\mathcal{O}_{1},M)\left[(M+\|\widetilde{v}\|_{H^{s+\alpha}}+1)\|U\|_{H^{s}}+\|v\|_{H^{s+\alpha}}\right].\label{eq:KepsDominated}
	\end{align}
	Since $\widetilde{U}$ satisfies condition \ref{cond2} and $U\in X_{s,\alpha,T}^{\infty}$, it follows that $K_{\epsilon}\in L^{2}([0,T];H^{s-\alpha})$.

	Now, let us show that 
	\begin{align}
		\label{eq:hminusa4}
		\int_{0}^{T}\|K_{\epsilon}(t)\|_{H^{s-\alpha}}^{2}dt\rightarrow 0\quad\mbox{as}\quad\epsilon\rightarrow 0. 
	\end{align}
	Notice that, $K_{\epsilon}^{1}(t)$ and $K_{\epsilon}^{3}(t)$ converge to zero for almost all $t\in[0,T]$. In fact, this is an immediate consequence of lemma \ref{epsilonlimitest}: statement $(i)$ yields the convergence under hypothesis (E1), while statement $(ii)$ yields the same conclusion under hypothesis (E2).
	
Now, as we previously showed,
\begin{align*}
K_{\epsilon}^{2}&\leq C(\mathcal{O}_{1},M)\left(K_{\epsilon}^{4}+K_{\epsilon}^{5}\right)
\end{align*}
with $K_{\epsilon}^{5}:=\|\Lambda^{\alpha}[\mathbb{J}_{\epsilon},G]\mathcal{L}U\|_{H^{s}}$. Since 
\begin{align*}
K_{\epsilon}^{5}\leq\|[\mathbb{J}_{\epsilon},G]\mathcal{L}U\|_{H^{1}}+\|\Lambda^{s+\alpha}[\mathbb{J}_{\epsilon},G]\mathcal{L}U\|_{L^{2}}
\end{align*}
	and $\mathcal{L}U\in\mathcal{C}([0,T];H^{s})$, statement (iii) in lemma \ref{epsilonlimitest} together with \cite[Lemma 1.5.2]{milani} imply that $K_{\epsilon}^{5}(t)\rightarrow 0$ for almost all $t\in[0,T]$.
	
	Finally, consider the term $K_{\epsilon}^{4}$. If hypothesis (E1) is satisfied we use conditions (C5)-(C6) together with the embedding $H^{s-\alpha}\hookrightarrow L^{\infty}$ to obtain 
	\begin{align*}
		K_{\epsilon}^{4}&\leq C\|\Lambda^{2\alpha}[\mathbb{J}_{\epsilon},G_{0}]v\|_{H^{s-\alpha}}\leq C\left(\|[\mathbb{J}_{\epsilon},G_{0}]v\|_{H^{2}}+\|\Lambda^{s+\alpha}[\mathbb{J}_{\epsilon},G_{0}]v\|_{L^{2}}\right).
	\end{align*}
	Therefore, as a consequence of \cite[Theorem 1.5.6]{milani} and statement $(iii)$ in lemma \ref{epsilonlimitest} we conclude that $K_{\epsilon}(t)\rightarrow 0$ for almost all $t\in[0,T]$. The same conclusion is obtained under hypothesis (E2). Indeed, in this case we use the embedding $H^{s-s_{\alpha}}\hookrightarrow L^{\infty}$ in \eqref{eq:firstK4}, to obtain
	\begin{align*}
		K_{\epsilon}^{4}\leq C\|[\mathbb{J}_{\epsilon},G]\mathcal{L}_{N}U\|_{H^{s}}.
	\end{align*} 
	Then, the conclusion follows by \cite[Theorem 1.5.6]{milani}. Therefore, $K_{\epsilon}(t)\rightarrow 0$ for almost all $t\in[0,T]$ and thus, by \eqref{eq:KepsDominated}, the dominated convergence Theorem yields \eqref{eq:hminusa4}.
	
	Applying the coercivity estimate \eqref{eq:fractionalstrongelliptic} with $U$ replaced by $\mathbb{J}_{\epsilon}U$, and combining it with the energy argument used in the proof of Theorem \ref{energyestimatesF}, we obtain
	\begin{equation}
		\label{eq:epsilonparestimate}
		E_{s,\alpha}^{2}(\mathbb{J}_{\epsilon}U(t))\leq K(t)\Phi_{\alpha}^{s}(\mathbb{J}_{\epsilon}U_{0},A^{0}F_{\epsilon}+A^{0}K_{\epsilon},t)
	\end{equation}
	for all $t\in[0,T]$. Since the fractional Leibniz rule and the embedding $H^{s-\alpha}\hookrightarrow L^{\infty}$ imply that
	\begin{align*}
		\|A^{0}K_{\epsilon}\|_{H^{s-\alpha}}\leq C(\mathcal{O}_{1},M)\|K_{\epsilon}\|_{H^{s-\alpha}},
	\end{align*} 
	it follows from \eqref{eq:hminusa4} that $A^{0}K_{\epsilon}\rightarrow 0$ in $L^{2}([0,T];H^{s-\alpha})$.
	
	To establish the convergence of $F_{\epsilon}$, write $A^{0}F_{\epsilon}=(f_{1}^{\epsilon},f_{2}^{\epsilon})\in\mathbb{R}^{k}\times\mathbb{R}^{m}$. Observe that, as a consequence of conditions (C5)-(C6), the nonzero components of $A^{0}[\mathbb{J}_{\epsilon},\overline{B}]\Lambda^{2\alpha}U$ are contained in the second block $f_{2}^{\epsilon}$. Therefore, by the convergence properties of $\mathbb{J}_{\epsilon}$ (see, \cite[Theorem 1.5.1 and Lemma 1.5.2]{milani}) it follows that $f_{1}^{\epsilon}\rightarrow f_{1}$ in $L^{2}([0,T];H^{s})$. For the second component, statement $(i)$ in lemma \ref{epsilonlimitest} yields $A^{0}[\mathbb{J}_{\epsilon},\overline{B}]\Lambda^{2\alpha}U\rightarrow 0$ as $\epsilon\rightarrow 0$ in $L^{2}([0,T];H^{s-\alpha})$ and thus, $f_{2}^{\epsilon}\rightarrow f_{2}$ in $L^{2}([0,T];H^{s-\alpha})$ as $\epsilon\rightarrow 0$. 
	Hence, for every $t\in[0,T]$,
	\begin{align*}
		\Phi_{\alpha}^{s}(\mathbb{J}_{\epsilon}U_{0},A^{0}F_{\epsilon}+A^{0}K_{\epsilon},t)\rightarrow\Phi_{\alpha}^{s}(U_{0},F,t)\quad\mbox{as}\quad\epsilon\rightarrow 0.
	\end{align*}
	Therefore, passing to the limit  in \eqref{eq:epsilonparestimate} as $\epsilon\rightarrow 0$ yields estimate \eqref{eq:Highenergy}, which proves the first assertion of the theorem.

Finally, if $U\in\mathcal{C}([0,T];H^{\sigma})$ for some $\sigma\in[0,s)$ it follows that $\mathbb{J}_{\epsilon}U\in\mathcal{C}([0,T];H^{s})$ for any $\epsilon>0$. Let $\epsilon_{1}>0$ and $\epsilon_{2}>0$ and observe that $\mathbb{J}_{\epsilon_{1}}U-\mathbb{J}_{\epsilon_{2}}U$ satisfies \eqref{eq:mollifiedlinearsystem} with source term $A^{0}(F_{\epsilon_{1}}-F_{\epsilon_{2}})+A^{0}(K_{\epsilon_{1}}-K_{\epsilon_{2}})$ and initial data $\mathbb{J}_{\epsilon_{1}}U_{0}-\mathbb{J}_{\epsilon_{2}}U_{0}$. Then, the energy method combined with \eqref{eq:hminusa4} yields
	\begin{align*}
		\sup_{t\in[0,T]} E^{2}_{s,\alpha}(\mathbb{J}_{\epsilon_{1}}U(t)-\mathbb{J}_{\epsilon_{2}}U(t))\rightarrow 0\quad\mbox{as}\quad\epsilon_{1},\epsilon_{2}\rightarrow 0.
	\end{align*}
	Consequently, $\{\mathbb{J}_{\epsilon}U\}_{\epsilon>0}$ is a Cauchy sequence in $\mathcal{C}([0,T];H^{s})$ and thus, $U\in\mathcal{C}([0,T];H^{s})$. 
\end{proof}

\begin{theo}
\label{finallinearwellposed}
Let $1\leq m\leq s$ be an integer and assume that $U_{0}\in H^{m}$, $f_{1}\in\mathcal{C}([0,T];H^{m-1})\cap L^{2}([0,T];H^{m})$ and $f_{2}\in\mathcal{C}([0,T];H^{m-2\alpha})\cap L^{2}([0,T];H^{m-\alpha})$. Assume condition \ref{cond1} is satisfied with one of the hypotheses (E1)-(E2). The Cauchy problem \eqref{eq:linearizedsystem1}-\eqref{eq:linearizedinitialdata} admits a unique solution $U=(u,v)$ in the space $X_{m,\alpha,T}$. Moreover, $\partial_{t}u\in\mathcal{C}([0,T];H^{m-1})$, $\partial_{t}v\in\mathcal{C}([0,T];H^{m-2\alpha})\cap L^{2}([0,T];H^{m-\alpha})$ and $U$ satisfies the estimate
\begin{align*}
E_{m,\alpha}^{2}(U(t))\leq K(t)\Phi_{\alpha}^{m}(U_{0},F,t)\quad\mbox{for all}\quad t\in[0,T].
\end{align*}
If, in addition, $\alpha\in(0,\tfrac{1}{2}]$, then  $\partial_{t}v\in\mathcal{C}([0,T];H^{m-1})$.
\end{theo}
\begin{proof}
We consider the sequence $\{U_{n}^{\prime}\}$ of solutions of 
\begin{equation}
\label{eq:finalapproxlinearsystem}
\begin{aligned}
A^{0}(\widetilde{U})\partial_{t}U_{n}^{\prime}+A^{j}(\widetilde{U})\partial_{j}U_{n}^{\prime}+D(\widetilde{U})U_{n}^{\prime}&+B(\widetilde{U})\Lambda^{2\alpha}U_{n}^{\prime}+H(\widetilde{U})[\Lambda^{2\alpha},G(\widetilde{U})]\mathcal{L}U_{n}^{\prime}=\mathbb{J}_{\epsilon_{n}}F,\\
	U_{n}^{\prime}\rvert_{t=0}&=\mathbb{J}_{\epsilon_{n}}U_{0}.
\end{aligned}
\end{equation}
Thanks to Theorems \ref{superapprox} and \ref{mollifiedapprox}, $U^{\prime}_{n}=(u^{\prime}_{n},v^{\prime}_{n})$ is a well-defined element of the space $X_{s,\alpha,T}$. Consequently, the difference $V^{\prime}_{n,p}:=U_{n+p}^{\prime}-U_{n}^{\prime}$ satisfies the energy estimate
\begin{align}
	E_{m,\alpha}^{2}(V_{n,p}^{\prime}(t))\leq K(t)\Phi_{\alpha}^{m}\left((\mathbb{J}_{\epsilon_{n+p}}-\mathbb{J}_{\epsilon_{n}})U_{0},(\mathbb{J}_{\epsilon_{n+p}}-\mathbb{J}_{\epsilon_{n}})F,t\right)\quad\mbox{for all}\quad t\in[0,T],\label{eq:parcauchy}
\end{align}
for any $1\leq m\leq s$. Passing to the limit in \eqref{eq:finalapproxlinearsystem} and \eqref{eq:parcauchy} shows the existence of $U\in X_{m,\alpha,T}$ that solves \eqref{eq:linearizedsystem1}-\eqref{eq:linearizedinitialdata}. The uniqueness is a consequence of the energy estimates and the regularity properties of $\partial_{t}u$ and $\partial_{t}v$ follow directly from equation \eqref{eq:linearizedsystem1}, the regularity of $U$, the embedding $H^{m-2\alpha}\hookrightarrow H^{m-1}$ for $\alpha\in(0,\tfrac{1}{2}]$ and the assumptions on $F$.
\end{proof}

\section{Invariant sets under iterations and boundedness on the high norm}
\label{nonlinear}
We now introduce the following definition.

\begin{defi}
Let $s>\frac{d}{2}+1$ be an integer and $\alpha\in(0,1)$ be given. We denote by $X_{s,\alpha,T}(\mathcal{O}_{1},M,M_{1})$ the set of functions $\widetilde{U}=\widetilde{U}(x,t)\in\mathcal{O}$ satisfying condition \ref{cond2} and such that
\begin{align}
E_{s,\alpha}^{2}(\widetilde{U}(t))\leq M^{2}\quad\mbox{for all}\quad t\in[0,T]\label{eq:boundenergy1}
\end{align}
and
\begin{align}
\int_{0}^{t}\|\partial_{t}\widetilde{U}(\tau)\|_{H^{s-1}}^{2}d\tau\leq M_{1}^{2}\quad\mbox{for all}\quad t\in[0,T],\label{eq:timeuniformderivative}
\end{align}
where $M$ and $M_{1}$ are constants. 
\end{defi}
The first objective of this section is to determine $\mathcal{O}_{1}$, $M$, $M_{1}$ and $T_{0}>0$ ($\leq T$) such that for $\widetilde{U}\in X_{s,\alpha,T_{0}}(\mathcal{O}_{1},M,M_{1})$, the Cauchy problem, 
\begin{equation}
\label{eq:invariantsystem1}
\begin{aligned}
A^{0}(\widetilde{U})\partial_{t}U+A^{j}(\widetilde{U})\partial_{j}U+D(\widetilde{U})U&+B(\widetilde{U})\Lambda^{2\alpha}U+H(\widetilde{U})[\Lambda^{2\alpha},G(\widetilde{U})]\mathcal{L}U=0,\\
U\rvert_{t=0}&=U_{0},
\end{aligned}
\end{equation}
has a unique solution in $X_{s,\alpha,T_{0}}(\mathcal{O}_{1},M,M_{1})$. That is, $X_{s,\alpha,T_{0}}(\mathcal{O}_{1},M,M_{1})$ is invariant under the solution map, $\widetilde{U}\mapsto U$. To this end, we consider the following condition.
\begin{cond}
\label{initialcondition}
$U_{0}\in H^{s}$ and $U_{0}\in\mathcal{O}_{0}$, where $\mathcal{O}_{0}$ is a bounded open convex set in $\mathbb{R}^{N}$ such that $\overline{\mathcal{O}_{0}}\subset\mathcal{O}$. 
\end{cond}
\begin{prop}
\label{timeestimateL2}
Assume that condition \ref{cond1} is satisfied with one of the hypotheses (E1)-(E2). Moreover, assume that $U_{0}\in H^{s}$ satisfies condition \ref{initialcondition}. Let $\widetilde{U}\in X_{s,\alpha,T}(\mathcal{O}_{1},M,M_{1})$ and let $U$ be a solution of \eqref{eq:invariantsystem1} satisfying \eqref{eq:boundenergy1} with $M$ replaced by $\overline{M}$. Then, there is a positive constant $C_{3}(\mathcal{O}_{1},M)>0$ for which
\begin{equation}
\label{eq:timeuniformestimate}
\begin{aligned}
\int_{0}^{t}\|\partial_{t}U(\tau)\|_{H^{s-1}}^{2}d\tau\leq C_{3}(\mathcal{O}_{1},M)\overline{M}^{2}\left[t+1+M^{2}\right].
\end{aligned}
\end{equation}
\end{prop}
\begin{proof}
Since $\widetilde{U}\in X_{s,\alpha,T}(\mathcal{O}_{1},M,M_{1})$ and $U$ is a solution of \eqref{eq:invariantsystem1}, Theorem \ref{finallinearwellposed} implies that we can take the norm of $\partial_{t}U$ in the space $H^{s-1}$. Then, by the Sobolev product estimates, 
\begin{align}
\|\partial_{t}U\|_{H^{s-1}}\leq C(\mathcal{O}_{1},M)\left(\|U\|_{H^{s}}+\|\Lambda^{2\alpha}v\|_{H^{s-1}}+\|[\Lambda^{2\alpha},G]\mathcal{L}U\|_{H^{s-1}}\right).\label{eq:timederivativeinvariant1}
\end{align}
Under any of the hypotheses (E1)-(E2), $\mathcal{L}U\in\mathcal{C}([0,T];H^{s})$. Therefore, we apply estimate \eqref{eq:fraccommnon} in \eqref{eq:timederivativeinvariant1} and use \eqref{eq:Gchainestimates} to obtain
\begin{align*}
\|\partial_{t}U\|_{H^{s-1}}\leq C_{3}(\mathcal{O}_{1},M)\left[\|U\|_{H^{s}}+\|\Lambda^{\alpha}v\|_{H^{s}}+\|\widetilde{v}\|_{H^{s+\alpha}}\|U\|_{H^{s}}\right].
\end{align*}
By integrating and using the uniform bounds for $U$ and $\widetilde{U}$, \eqref{eq:timeuniformestimate} follows.
\end{proof}

Now, fix a constant $\delta_{1}>0$ so that $0<\delta_{1}<\delta_{0}:=d(\mathcal{O}_{0},\partial\mathcal{O})$ and take
\begin{align}
\mathcal{O}_{1}&=\delta_{1}-\mbox{neighborhood of}\quad\mathcal{O}_{0},\label{eq:O1}\\
M&=\sqrt{2C_{0}(\mathcal{O}_{1})}\|U_{0}\|_{H^{s}}.\label{eq:Mconstant}
\end{align}
\begin{theo}
\label{invariantsetsTheorem}
Assume condition \ref{initialcondition} holds true. There is a positive constant $T_{0}$ depending on $\delta_{0}$, $\delta_{1}$ and $\|U_{0}\|_{H^{s}}$ such that, if $\widetilde{U}\in X_{s,\alpha,T}(\mathcal{O}_{1},M, M_{1})$ with $\mathcal{O}_{1}$ and $M$ defined by \eqref{eq:O1}-\eqref{eq:Mconstant}, the Cauchy problem \eqref{eq:invariantsystem1} has a unique solution $U=(u,v)$ in the same set $X_{s,\alpha,T}(\mathcal{O}_{1},M,M_{1})$. 
\end{theo}
\begin{proof}
The existence of a unique solution to the Cauchy problem \eqref{eq:invariantsystem1} follows from \ref{finallinearwellposed} with $m=s$. By using the energy estimates we have
\begin{align}
E^{2}_{s,\alpha}(U(t))\leq C_{0}(\mathcal{O}_{1})e^{C_{2}(\mathcal{O}_{1},M)\int_{0}^{t}\mu_{1}(\tau)d\tau}\|U_{0}\|_{H^{s}}^{2},\label{eq:energyinvariant1}
\end{align}
where, $\mu_{1}$ was defined in Theorem \ref{energyestimatesF} and, by H\"older's inequality
\begin{align*}
\int_{0}^{t}\mu_{1}(\tau)d\tau&\leq\left[(1+M+M^{2})t+t^{1/2}M\right]+t^{1/2}M_{1}+\int_{0}^{t}\|\widetilde{v}(\tau)\|_{H^{s+\alpha}}^{2}d\tau
\end{align*}
In consequence, since $\widetilde{v}\in L^{2}([0,T];H^{s+\alpha})$ there is $T_{0}>0$ such that 
\begin{align*}
e^{C_{1}(\mathcal{O}_{1},M)\int_{0}^{T_{0}}\mu_{1}(\tau)d\tau}\leq 2
\end{align*}
and thus, by \eqref{eq:energyinvariant1}
\begin{align*}
\sup_{t\in[0,T_{0}]}E^{2}_{s,\alpha}(U(t))\leq 2C_{0}(\mathcal{O}_{1})\|U_{0}\|_{H^{s}}^{2}:=M^{2}.
\end{align*}
Then, by applying proposition \ref{timeestimateL2} with $\overline{M}=M$ it holds that
\begin{align*}
\int_{0}^{T_{0}}\|\partial_{t}U(\tau)\|_{H^{s-1}}^{2}d\tau\leq C_{3}(\mathcal{O}_{1},M)M^{2}\left[T_{0}+1+M^{2}\right].
\end{align*}
Therefore, if we take $0<T_{0}<1$ and define $M_{1}$ as 
\begin{align*}
M_{1}^{2}:=C_{3}(\mathcal{O}_{1},M)M^{2}[2+M^{2}],
\end{align*}
condition \eqref{eq:timeuniformderivative} is satisfied. This estimate gives
\begin{align*}
|U(x,t)-U(x,0)|\leq\int_{0}^{t}\|\partial_{t}U(\tau)\|_{L^{\infty}}d\tau\leq\kappa_{s-1}\int_{0}^{t}\|\partial_{j}U(\tau)\|_{H^{s-1}}\leq\kappa_{s-1}T_{0}^{1/2}M_{1}.
\end{align*}
Hence, by taking $T_{0}>0$ such that  $\kappa_{s-1}T_{0}^{1/2}M_{1}\leq\delta_{1}$, we conclude that $U(x,t)\in\mathcal{O}_{1}$ for all $(x,t)\in Q_{T_{0}}$. 
\end{proof}
\section{H\"older continuity in the low norm}
\label{holdernorm}
According with Theorem \ref{invariantsetsTheorem}, for every $\widetilde{U}\in X_{s,\alpha,T_{0}}(\mathcal{O}_{1},M,M_{1})$ there is a unique $U\in X_{s,\alpha,T_{0}}(\mathcal{O}_{1},M,M_{1})$ that satisfies the Cauchy problem \eqref{eq:invariantsystem1}. Hence, the solution map
\begin{equation}
\label{eq:solutionmapping}
\begin{aligned}
\mathcal{T}_{\alpha}:~X_{s,\alpha,T_{0}}(\mathcal{O}_{1},M,M_{1})&\rightarrow~X_{s,\alpha,T_{0}}(\mathcal{O}_{1},M,M_{1}),\\
\widetilde{U}&\mapsto~U,
\end{aligned}
\end{equation}
is well-defined. Observe that $X_{s,\alpha,T_{0}}(\mathcal{O}_{1},M,M_{1})\subset Y_{s-1,\alpha,T_{0}}$, where
\begin{align}
Y_{s-1,\alpha,T_{0}}:=\{U=(u,v)^{\top}\in L^{2}([0,T_{0}];H^{s-1})~|~\Lambda^{\alpha}v\in L^{2}([0,T_{0}];H^{s-1})\}.\label{eq:Yspace}
\end{align}
The purpose of this section is to prove that the solution map $\mathcal{T}_{\alpha}$ is continuous when $X_{s,\alpha,T_{0}}(\mathcal{O}_{1},M,M_{1})$ is endowed with the topology induced by $Y_{s-1,\alpha,T_{0}}$.

Consider $\widetilde{U}^{1},\widetilde{U}^{2}\in X_{s,\alpha,T_{0}}(\mathcal{O}_{1},M)$ and the corresponding $\mathcal{T}_{\alpha}(\widetilde{U}_{1}):=U_{1}$ and $\mathcal{T}_{\alpha}(\widetilde{U}_{2}):=U_{2}$ also in $X_{s,\alpha,T_{0}}(\mathcal{O}_{1},M)$. Then, the difference $U^{1}-U^{2}=:U$ satisfies
\begin{equation}
\label{eq:contractionsystem}
A^{0}(\widetilde{U}^{1})\partial_{t}U+A^{j}(\widetilde{U}^{1})\partial_{j}U+D(\widetilde{U}^{1})U+B(\widetilde{U}^{1})\Lambda^{2\alpha}U+H(\widetilde{U}^{1})[\Lambda^{2\alpha},G(\widetilde{U}^{1})]\mathcal{L}U=F^{1,2}+Q^{1,2}
\end{equation}
with,
\begin{equation}
\label{eq:contractionsource}
\begin{aligned}
(A^{0}(\widetilde{U}^{1}))^{-1}F^{1,2}&:=-\left[\overline{A}^{j}(\widetilde{U}^{1})-\overline{A}^{j}(\widetilde{U}^{2})\right]\partial_{j}U^{2}-\left[\overline{D}(\widetilde{U}^{1})-\overline{D}(\widetilde{U}^{2})\right]U^{2}-\left[\overline{B}(\widetilde{U}^{1})-\overline{B}(\widetilde{U}^{2})\right]\Lambda^{2\alpha}U^{2},\\
(A^{0}(\widetilde{U}^{1}))^{-1}Q^{1,2}&:=-\left[\overline{H}(\widetilde{U}^{1})-\overline{H}(\widetilde{U}^{2})\right][\Lambda^{2\alpha},G(\widetilde{U}^{1})]\mathcal{L}U^{2}-\overline{H}(\widetilde{U}^{2})[\Lambda^{2\alpha},G(\widetilde{U}^{1})-G(\widetilde{U}^{2})]\mathcal{L}U^{2}
\end{aligned}
\end{equation}
and initial data, $	U\rvert_{t=0}=0$. Here, for every matrix coefficient $A$, we have used the notation, $\overline{A}(\widetilde{U}^{\ell})=(A^{0}(\widetilde{U}^{\ell}))^{-1}A(\widetilde{U}^{\ell})$ for $\ell=1,2$.

At first sight, one might attempt to establish a contraction estimate for the solution operator $\mathcal{T}_{\alpha}$ in the norm of the space $X_{s-1,\alpha,T_{0}}^{\infty}$ by applying the energy method to the difference $U^{1}-U^{2}$. The main obstacle to this approach is the non-local contribution
\begin{align*}
	\mathcal{N}_{1,2}:=\overline{H}(\widetilde{U}^{2})[\Lambda^{2\alpha},G(\widetilde{U}^{1})-G(\widetilde{U}^{2})]\mathcal{L}U^{2},
\end{align*}
which appears in the source term $Q^{1,2}$ of \eqref{eq:contractionsource}. In order to obtain a contraction estimate, one would need a bound of the form 
\begin{align}
\int_{0}^{T_{0}}\|\mathcal{N}_{1,2}(\tau)\|_{H^{s-1-\alpha}}^{2}d\tau\leq C(\mathcal{O}_{1},M)f(T_{0})\|\widetilde{U}^{1}-\widetilde{U}^{2}\|_{X_{s-1,\alpha,T_{0}}^{\infty}}^{2},\label{eq:failedcontraction}
\end{align}
where $f(T_{0})\rightarrow 0$ as $T_{0}\rightarrow 0$. Let us assume that $H$ and $A^{0}$ are constant matrices (as in \cite{chenalign}). Since 
\begin{align*}
\Lambda^{s-1-\alpha}[\Lambda^{2\alpha},G(\widetilde{U}^{1})-G(\widetilde{U}^{2})]\mathcal{L}U^{2}&=[\Lambda^{s-1+\alpha},G(\widetilde{U}^{1})-G(\widetilde{U}^{2})]\mathcal{L}U^{2}\\
&-[\Lambda^{s-1-\alpha},G(\widetilde{U}^{1})-G(\widetilde{U}^{2})]\Lambda^{2\alpha}\mathcal{L}U^{2},
\end{align*}
a natural idea is to invoke standard commutator estimates. Indeed, by the fractional Leibniz rule with $p_{1}=p_{2}=4$, $p_{3}=2$, $p_{4}=\infty$ and the embedding $H^{1}\hookrightarrow L^{4}$, we have
\begin{align}
\|\mathcal{N}_{1,2}\|_{\dot{H}^{s-1-\alpha}}\leq C\left(\mathcal{N}_{1,2}^{c}+\mathcal{N}_{1,2}^{nc}\right)\label{eq:stability}
\end{align}
with 
\begin{align*}
&\mathcal{N}_{1,2}^{c}:=\|G_{0}(\widetilde{v}^{1})-G_{0}(\widetilde{v}^{2})\|_{H^{2}}\|U^{2}\|_{H^{s}}+\|\Lambda^{s-1-\alpha}(G_{0}(\widetilde{v}^{1})-G_{0}(\widetilde{v}^{2}))\|_{L^{2}}\|\Lambda^{2\alpha}v^{2}\|_{L^{\infty}},\\
&\mathcal{N}^{nc}_{1,2}:=\|\Lambda^{s-1+\alpha}(G_{0}(\widetilde{v}^{1})-G_{0}(\widetilde{v}^{2}))\|_{L^{2}}\|U^{2}\|_{L^{\infty}}.
\end{align*}
Under assumptions $(E1)-(E2)$, it is easily seen that 
\begin{align*}
\int_{0}^{T_{0}}\mathcal{N}_{1,2}^{c}(\tau)^{2}d\tau\leq C(\mathcal{O}_{1},M)\|\widetilde{U}^{1}-\widetilde{U}^{2}\|_{X_{s-1,\alpha,T_{0}}^{\infty}}^{2}\left(M^{2}T_{0}+\int_{0}^{T_{0}}\|\Lambda^{\alpha}v^{2}(\tau)\|_{H^{s}}^{2}d\tau\right).
\end{align*}
Hence, the contribution $\mathcal{N}_{1,2}^{c}$ is contractive in the sense \eqref{eq:failedcontraction}. However, the best available estimate for the term $\mathcal{N}_{1,2}^{nc}$ is 
\begin{align*}
\int_{0}^{T^{0}}\mathcal{N}_{1,2}^{nc}(\tau)^{2}d\tau\leq C(\mathcal{O}_{1},M)M^{2}\int_{0}^{T_{0}}\|\Lambda^{\alpha}(\widetilde{v}_{1}-\widetilde{v}_{2})(\tau)\|_{H^{s-1}}^{2}d\tau\leq C(\mathcal{O}_{1},M)M^{2}\|\widetilde{U}^{1}-\widetilde{U}^{2}\|_{X_{s-1,\alpha,T_{0}}^{\infty}}^{2}
\end{align*}
Therefore, $\mathcal{N}_{1,2}^{nc}$ is of Lipschitz type with respect to the norm of $X_{s-1,\alpha,T_{0}}^{\infty}$, but there is no small factor $f(T_{0})$.

An alternative approach is to break down the commutator and use de triangle inequality to estimate each term separately. Following this path also leads to the estimate for $\mathcal{N}_{1,2}^{nc}$.

This difficulty stems from the fact that the nonlinear term $\mathcal{N}_{1,2}$ involves derivatives of order $2\alpha$ on the difference $G_{0}(\widetilde{v}^{1})-G_{0}(\widetilde{v}^{2})$, which is the same order of the dissipative term. Consequently, the estimate of $\mathcal{N}_{1,2}^{nc}$ exhausts all the regularity available for $\widetilde{v}^{1}-\widetilde{v}^{2}$, leaving no small factor $f(T_{0})\to 0$ as $T_{0}\to 0$. This is in contrast with the standard models for the dynamics of a compressible viscous, heat conducting fluid in space (see, \cite{felipecontraction}, \cite{kawa} and \cite{serre2}, for example), where the fully nonlinear terms are of differential order one, whereas the  dominant dissipative terms are of differential order two. Therefore, in this sense, the commutator non-linearity cannot be considered a low order perturbation of the dissipative term.

The next theorem establishes a H\"older continuity estimate of order $1/2$ for the solution map with respect to the norm $Y_{s-1,\alpha,T_{0}}$. Although a Lipschitz continuity estimate can be obtained in $X_{s-1,\alpha,T_{0}}^{\infty}$, the H\"older estimate is more naturally adapted to the compactness argument developed below.
\begin{theo}
\label{holdercontinuity}
Let $s>\frac{d}{2}+1$ be an integer and $\alpha\in(0,1)$. Under one of the hypotheses (E1)-(E2), there is a constant $C_{3}(\mathcal{O}_{1},M,T_{0})>0$ such that for any $\widetilde{U}^{1},\widetilde{U}^{2}\in X_{s,\alpha,T_{0}}(\mathcal{O}_{1},M)$ and any $t\in[0,T_{0}]$, 
\begin{equation}
\label{eq:contractioninequality1h}
\begin{aligned}
&\|\mathcal{T}_{\alpha}(\widetilde{U}^{2}(t))-\mathcal{T}_{\alpha}(\widetilde{U}^{1}(t))\|_{H^{s-1}}^{2}+\int_{0}^{t}\|\Lambda^{\alpha}(v_{1}(\tau)-v_{2}(\tau))\|_{H^{s-1}}^{2}d\tau\\
&\leq C_{3}(\mathcal{O}_{1},M,T_{0})\left(\int_{0}^{t}\|\widetilde{U}^{1}(\tau)-\widetilde{U}^{2}(\tau)\|_{H^{s-1}}^{2}d\tau+\int_{0}^{t}\|\Lambda^{\alpha}(\widetilde{v}^{1}(\tau)-\widetilde{v}^{2}(\tau))\|_{H^{s-1}}^{2}d\tau\right)^{1/2}.
\end{aligned}
\end{equation}
In particular, the mapping $\mathcal{T}_{\alpha}:X_{s,\alpha,T_{0}}(\mathcal{O}_{1},M)\rightarrow X_{s,\alpha,T_{0}}(\mathcal{O}_{1},M)$ is H\"older continuous of order $1/2$ in the topology of the space $Y_{s-1,\alpha,T_{0}}$.
\end{theo}
\begin{proof}
For any $\widetilde{U}^{1},\widetilde{U}^{2}\subset X_{s,\alpha,T_{0}}(\mathcal{O}_{1},M)$ set $U^{1}=\mathcal{T}_{\alpha}(\widetilde{U}^{1})$ and $U^{2}=\mathcal{T}_{\alpha}(U^{2})$. Since $U=U^{1}-U^{2}$, is a solution of \eqref{eq:contractionsystem} with \eqref{eq:contractionsource}, we can apply the energy method described in Theorem \ref{energyestimatesF} to any value of $\alpha\in(0,1)$ (see estimates \eqref{eq:energyEstimate10}) in order to obtain that, 
\begin{align}
\frac{d}{dt}\|U(t)\|_{H^{s-1}}^{2}+\|\Lambda^{\alpha}v(t)\|_{H^{s-1}}^{2}\leq C_{2}(\mathcal{O}_{1},M)\left[\sum_{|\gamma|=0}^{s-1}|\mathcal{J}_{\gamma}(t)|+|\mathcal{K}_{\gamma}(t)|+\mu_{1}(t)\|U(t)\|_{H^{s-1}}^{2}\right],\label{eq:continuityEstimate1}
\end{align}
where, 
\begin{align*}
\mathcal{J}_{\gamma}(t):=\langle D^{\gamma}F^{1,2}(t),D^{\gamma}U(t)\rangle_{L^{2}}\quad\mbox{and}\quad\mathcal{K}_{\gamma}(t):=\langle D^{\gamma}Q^{1,2}(t),D^{\gamma}U(t)\rangle_{L^{2}}.
\end{align*}
After integrating in the interval $[0,T_{0}]$ and using Gronwall's inequality, we arrive at the estimate
\begin{align}
\|U(t)\|_{H^{s-1}}^{2}+\int_{0}^{t}\|\Lambda^{\alpha}v(\tau)\|_{H^{s-1}}^{2}d\tau\leq 2C_{2}(\mathcal{O}_{1},M)\int_{0}^{t}\left(\sum_{|\gamma|=0}^{s-1}|\mathcal{J}_{\gamma}(t)|+|\mathcal{K}_{\gamma}(t)|\right)dt\label{eq:continuityEstimate2}
\end{align}
for all $t\in[0,T_{0}]$. The treatment of $\mathcal{J}_{\gamma}$ is relatively straightforward. By the Cauchy-Schwarz inequality and the Sobolev product estimates, we have
\begin{align*}
\sum_{|\gamma|=0}^{s-1}|\mathcal{J}_{\gamma}(t)|\leq C_{1}(\mathcal{O}_{1},M)\|\widetilde{U}^{1}-\widetilde{U}^{2}\|_{H^{s-1}}(\|U^{2}\|_{H^{s}}+\|\Lambda^{2\alpha}v^{2}\|_{H^{s-1}})\|U\|_{H^{s-1}}.
\end{align*}
Since $U^{j}\in X_{s,\alpha,T_{0}}(\mathcal{O}_{1},M)$, we obtain
\begin{align}
\sum_{|\gamma|=0}^{s-1}|\mathcal{J}_{\gamma}(t)|\leq 2	MC_{1}(\mathcal{O}_{1},M)\left(M+\|\Lambda^{2\alpha}v^{2}\|_{H^{s-1}}\right)\|\widetilde{U}^{1}-\widetilde{U}^{2}\|_{H^{s-1}}.\label{eq:continuityEstimate3}
\end{align}
Thus, the terms $\mathcal{J}_{\gamma}$ can be controlled directly in terms of the norm $\|\widetilde{U}^{1}-\widetilde{U}^{2}\|_{Y_{s-1,\alpha,T_{0}}}$.

The situation is different for the commutator contribution contained in $\mathcal{K}_{\gamma}$. A direct estimate would require controlling $G_{0}(\widetilde{v}^{1})-G_{0}(\widetilde{v}^{2})$ at a regularity level that is not available in the topology of $Y_{s-1,\alpha,T_{0}}$ (defined in \eqref{eq:Yspace}). Instead, we integrate by parts in the higher order contributions, thereby reducing the regularity required on $\widetilde{v}^{1}-\widetilde{v}^{2}$. Let $\mathcal{L}U^{2}=(y^{2},w^{2})\in\mathbb{R}^{k}\times\mathbb{R}^{m}$. As a consequence of condition (C6),
\begin{align*}
\sum_{|\gamma|=0}^{s-1}|\mathcal{K}_{\gamma}(t)|&\leq\sum_{|\gamma|=0}^{s-1}|\left\langle D^{\gamma}\left\lbrace\left[H_{0}(\widetilde{v}^{1})-H_{0}(\widetilde{v}^{2})\right][\Lambda^{2\alpha},G_{0}(\widetilde{v}^{1})]w^{2}\right\rbrace,D^{\gamma}v\right\rangle_{L^{2}}|\\
&+\sum_{|\gamma|=0}^{s-1}|\left\langle D^{\gamma}\left\lbrace H_{0}(\widetilde{v}^{2})[\Lambda^{2\alpha},G_{0}(\widetilde{v}^{1})-G_{0}(\widetilde{v}^{2})]w^{2}\right\rbrace,D^{\gamma}v\right\rangle_{L^{2}}|=:\mathcal{K}_{s-1}^{1}+\mathcal{K}_{s-2}^{2}.
\end{align*}
For estimating $\mathcal{K}_{s-1}^{1}$ we first apply the Cauchy-Schwarz inequality and then estimate \eqref{eq:fraccommnon} to obtain
\begin{align*}
\mathcal{K}_{s-1}^{1}&\leq C_{1}(\mathcal{O}_{1},M)\|\widetilde{U}^{1}-\widetilde{U}^{2}\|_{H^{s-1}}\left(\|G_{0}(\widetilde{v}^{1})\|_{\dot{H}^{\alpha}}+\|G_{0}(\widetilde{v}^{1})\|_{\dot{H}^{s+\alpha}}\right)\|U^{2}\|_{H^{s}}\|U\|_{H^{s-1}}.
\end{align*}
Since $\|\widetilde{U}^{j}\|_{H^{s}},\|U^{j}\|_{H^{s}}\leq M$,
\begin{align}
\label{eq:continuityEstimate4}
\mathcal{K}_{s-1}^{1}\leq 2M^{2}C_{1}(\mathcal{O}_{1},M)\|\widetilde{U}^{1}-\widetilde{U}^{2}\|_{H^{s-1}}\left(M+\|\Lambda^{2\alpha}\widetilde{v}^{1}\|_{H^{s-1}}\right).
\end{align}
In order to estimate the term $\mathcal{K}_{s-1}^{2}$ we first observe that, 
\begin{align*}
\mathcal{K}_{s-1}^{2}\leq\mathcal{K}_{s-1}^{3}+\mathcal{K}_{s-1}^{4},
\end{align*}
with 
\begin{align*}
\mathcal{K}_{s-1}^{3}:=\sum_{|\gamma|=0}^{s-1}\left\lvert\left\langle[D^{\gamma},H_{0}(\widetilde{v}^{2})][\Lambda^{2\alpha},G_{0}(\widetilde{v}^{1})-G_{0}(\widetilde{v}^{2})]w^{2},D^{\gamma}v\right\rangle_{L^{2}}\right\rvert
\end{align*}
and 
\begin{align*}
\mathcal{K}_{s-1}^{4}:=\sum_{|\gamma|=0}^{s-1}\left\lvert\left\langle D^{\gamma}[\Lambda^{2\alpha},G_{0}(\widetilde{v}^{1})-G_{0}(\widetilde{v}^{2})]w^{2},H_{0}(\widetilde{v}^{2})^{\top}D^{\gamma}v\right\rangle_{L^{2}}\right\rvert.
\end{align*}
For the term $\mathcal{K}_{s-1}^{3}$, the commutator estimates for $D^{\gamma}$ imply 
\begin{align*}
\mathcal{K}_{s-1}^{3}&\leq C(\mathcal{O}_{1},M)\|[\Lambda^{2\alpha},G_{0}(\widetilde{v}^{1})-G_{0}(\widetilde{v}^{2})]w^{2}\|_{H^{s-2}}\|U\|_{H^{s-1}}.
\end{align*}
The term $\mathcal{K}_{s-1}^{4}$ contains the highest-order derivatives on the commutator. To estimate $\mathcal{K}_{s-1}^{4}$ we write $\gamma=\eta+\beta$ with $|\eta|=1$ for every $\gamma\in\mathbb{N}^{d}$ with $|\gamma|\geq 1$ and use integration by parts,
\begin{align*}
\mathcal{K}_{s-1}^{4}&:=\left\lvert\left\langle[\Lambda^{2\alpha},G_{0}(\widetilde{v}^{1})-G_{0}(\widetilde{v}^{2})]w^{2},H_{0}(\widetilde{v}^{2})^{\top}v\right\rangle_{L^{2}}\right\rvert\\
&+\sum_{|\beta|=0}^{s-2}\left\lvert\left\langle D^{\beta}[\Lambda^{2\alpha},G_{0}(\widetilde{v}^{1})-G_{0}(\widetilde{v}^{2})]w^{2},D^{\eta}\left(H_{0}(\widetilde{v}^{2})^{\top}D^{\beta+\eta}v\right)\right\rangle_{L^{2}}\right\rvert\\
&\leq C(\mathcal{O}_{1},M)\|[\Lambda^{2\alpha},G_{0}(\widetilde{v}^{1})-G_{0}(\widetilde{v}^{2})]w^{2}\|_{H^{s-2}}\|U\|_{H^{s}}.
\end{align*}
Then, by the triangle inequality, the Sobolev product estimates and the fractional Leibniz rule it follows that
\begin{align*}
\|[\Lambda^{2\alpha},G_{0}(\widetilde{v}^{1})-G_{0}(\widetilde{v}^{2})]w^{2}\|_{H^{s-2}}
&\leq C\left(\|(G_{0}(\widetilde{v}^{1})-G_{0}(\widetilde{v}^{2}))w^{2}\|_{H^{2}}+\|(G(\widetilde{v}^{1})-G(\widetilde{v}^{2}))w^{2}\|_{\dot{H}^{s-2+2\alpha}}\right)\\
&+\|G_{0}(\widetilde{v}^{1})-G_{0}(\widetilde{v}^{2})\|_{\widehat{H}^{s-1}}\|\Lambda^{2\alpha}w^{2}\|_{H^{s-2}}\\
&\leq C\left(\|G_{0}(\widetilde{v}^{1})-G_{0}(\widetilde{v}^{2})\|_{\widehat{H}^{s-1}}\|U^{2}\|_{H^{s}}+\|G_{0}(\widetilde{v}^{1})-G_{0}(\widetilde{v}^{2})\|_{\dot{H}^{s-2+2\alpha}}\|U^{2}\|_{L^{\infty}}\right.\\
&+\left.\|G_{0}(\widetilde{v}^{1})-G_{0}(\widetilde{v}^{2})\|_{L^{\infty}}\|U^{2}\|_{\dot{H}^{s-2+2\alpha}}\right)
\end{align*}
and thus, by \eqref{eq:Lipschitzcontinuityfinal} we  obtain
\begin{align*}
\|[\Lambda^{2\alpha},G_{0}(\widetilde{v}^{1})-G_{0}(\widetilde{v}^{2})]w^{2}\|_{H^{s-2}}\leq C(\mathcal{O}_{1},M)M\left(\|\widetilde{U}^{1}-\widetilde{U}^{2}\|_{H^{s-1}}+\|\widetilde{v}^{1}-\widetilde{v}^{2}\|_{H^{s-2+2\alpha}}\right).
\end{align*}
In consequence, by combining the estimates for $\mathcal{K}_{s-1}^{3}$ and $\mathcal{K}_{s-1}^{4}$ we conclude
\begin{equation}
\label{eq:continuityEstimate5}
\begin{aligned}
\mathcal{K}_{s-1}^{2}&\leq C_{1}(\mathcal{O}_{1},M)2M^{2}\left(\|\widetilde{U}^{1}-\widetilde{U}^{2}\|_{H^{s-1}}+\|\widetilde{v}^{1}-\widetilde{v}^{2}\|_{H^{s-2+2\alpha}}\right).
\end{aligned}
\end{equation}
Therefore, by\eqref{eq:continuityEstimate4} and \eqref{eq:continuityEstimate5}, we have
\begin{equation}
\label{eq:continuityEstimate6}
\begin{aligned}
\sum_{|\gamma|=0}^{s-1}|\mathcal{K}_{\gamma}(t)|\leq&2M^{2}C_{1}(\mathcal{O}_{1},M)\left[\left(1+M+\|\Lambda^{2\alpha}\widetilde{v}^{1}\|_{H^{s-1}}\right)\|\widetilde{U}^{1}-\widetilde{U}^{2}\|_{H^{s-1}}+\|\widetilde{v}^{1}-\widetilde{v}^{2}\|_{H^{s-2+2\alpha}}\right].
\end{aligned}
\end{equation}
Using \eqref{eq:continuityEstimate3} and \eqref{eq:continuityEstimate6} in the right hand side of \eqref{eq:continuityEstimate2}, applying H\"older's inequality and using that $\widetilde{U}^{1},U^{2}\in X_{s,\alpha,T_{0}}(\mathcal{O}_{1},M,M_{1})$, gives
\begin{align*}
&\|U(t)\|_{H^{s-1}}^{2}+\int_{0}^{t}\|\Lambda^{\alpha}v(\tau)\|_{H^{s-1}}^{2}d\tau\\
&\leq C_{2}(\mathcal{O}_{1},M)\left(\int_{0}^{t}(M^{2}+\|\Lambda^{\alpha}v^{2}(\tau)\|_{H^{s}}^{2}+\|\Lambda^{\alpha}\widetilde{v}^{1}\|_{H^{s}}^{2})d\tau\right)^{1/2}\left(\int_{0}^{t}\|\widetilde{U}^{1}(\tau)-\widetilde{U}^{2}(\tau)\|_{H^{s-1}}^{2}d\tau\right)^{1/2}\\
&+ C_{2}(\mathcal{O}_{1},M)t^{1/2}\left(\int_{0}^{t}\|\Lambda^{\alpha}(\widetilde{v}^{1}(\tau)-\widetilde{v}^{2}(\tau))\|_{H^{s-1}}^{2}d\tau\right)^{1/2}\\
&\leq C_{2}(\mathcal{O}_{1},M)(Mt^{1/2}+M)\left(\int_{0}^{t}\|\widetilde{U}^{1}(\tau)-\widetilde{U}^{2}(\tau)\|_{H^{s-1}}^{2}d\tau\right)^{1/2}\\
&+ C_{2}(\mathcal{O}_{1},M)t^{1/2}\left(\int_{0}^{t}\|\Lambda^{\alpha}(\widetilde{v}^{1}(\tau)-\widetilde{v}^{2}(\tau))\|_{H^{s-1}}^{2}d\tau\right)^{1/2}
\end{align*}
for all $t\in[0,T_{0}]$. Hence, estimate \eqref{eq:contractioninequality1h} follows. 
\end{proof}
\section{Compactness: tightness of the $L^{2}$ norm}
\label{tight}
In the following result we establish the relative compactness of the solution map $\mathcal{T}_{\alpha}: X_{s,\alpha,T_{0}}(\mathcal{O}_{1},M,M_{1})\rightarrow X_{s,\alpha,T_{0}}(\mathcal{O}_{1},M,M_{1})$ with respect to the topology of $Y_{s-1,\alpha,T_{0}}$. The proof relies on three ingredients: The Aubin-Lions lemma, a uniform control of the $L^{2}$-tails of the family $U_{n}=\mathcal{T}_{\alpha}(\widetilde{U}_{n})$, and the Kato-Ponce-Li commutator expansion. More precisely, if $\omega$ is a scalar valued function and $V$ is a vector field, then
\begin{align}
\omega\Lambda^{2\alpha}V=\Lambda^{2\alpha}(\omega V)-\mathbb{K}_{2\alpha}(\omega,V)-\sum_{0<|\beta|\leq 2\alpha}\frac{1}{\beta!}D^{\beta}\omega\Lambda^{2\alpha,\beta}V,\label{eq:usefuldecomp}
\end{align}
where $\mathbb{K}_{\sigma}(\omega,V)$ was defined in lemma \ref{dongli}. 

\begin{lema}
\label{compactness}
The set $\mathcal{T}_{\alpha}\left(X_{s,\alpha,T_{0}}(\mathcal{O}_{1},M)\right)\subset X_{s,\alpha,T_{0}}(\mathcal{O}_{1},M)$ is relatively compact in $Y_{s-1,\alpha,T_{0}}$.
\end{lema}
\begin{proof}
Fix $\alpha\in(0,1)$ and consider the sequences $\{\widetilde{U}_{n}\},\{U_{n}\}\subset X_{s,\alpha,T_{0}}(\mathcal{O}_{1},M)$ where $U_{n}=\mathcal{T}_{\alpha}(\widetilde{U}_{n})$ for all $n\in\mathbb{N}$. By the definition of $X_{s,\alpha,T_{0}}(\mathcal{O}_{1},M)$, the weak convergence properties of Hilbert spaces and the Banach-Alaoglu Theorem, we can extract a sub-sequence $\{U_{n(j)}\}\subset\{U_{n}\}$ and some $U\in X_{s,\alpha,T_{0}}$ such that 
\begin{equation}
\label{eq:weakconvergence}
\begin{aligned}
U_{n(j)}\overset{\ast}{\rightharpoonup}U&\quad\mbox{in}\quad L^{\infty}([0,T_{0}];H^{s}),\\
\Lambda^{\alpha}v_{n(j)}\rightharpoonup \Lambda^{\alpha}v&\quad\mbox{in}\quad L^{2}([0,T_{0}];H^{s}),\\
\partial_{t}U_{n(j)}\rightharpoonup\partial_{t}U&\quad\mbox{in}\quad L^{2}([0,T_{0}];H^{s-1}).
\end{aligned}
\end{equation}
In consequence, by the lower semicontinuity of the norms,
\begin{equation}
\label{eq:weakcomp1}
\begin{aligned}
\|U\|_{L^{\infty}([0,T_{0}];H^{s})}^{2}&+\int_{0}^{T_{0}}\|\Lambda^{\alpha}v(\tau)\|_{H^{s}}^{2}d\tau\\
&\leq\liminf_{j\rightarrow\infty}\left(\|U_{n(j)}\|_{L^{\infty}([0,T_{0}];H^{s})}^{2}+\int_{0}^{T_{0}}\|\Lambda^{\alpha}v_{n(j)}(\tau)\|_{H^{s}}^{2}d\tau\right)\leq M^{2}
\end{aligned}
\end{equation}
and
\begin{align}
\int_{0}^{T_{0}}\|\partial_{t}U\|_{H^{s-1}}^{2}d\tau\leq\liminf_{j\rightarrow\infty}\int_{0}^{T_{0}}\|\partial_{t}U_{n(j)}(\tau)\|_{H^{s-1}}^{2}d\tau\leq M_{1}^{2}.\label{eq:weakcomp2} 
\end{align}
Moreover, the weak$^{\ast}$ convergence in \eqref{eq:weakconvergence} implies that
\begin{align}
U_{n(j)}\rightharpoonup U\quad\mbox{in}\quad L^{2}([0,T_{0}];H^{s}). \label{eq:weakconvergence2}
\end{align}
Now, let $B_{N}$ be a ball of radius $N\in\mathbb{N}$ centered at zero and consider the compact embedding $H^{s}(B_{N})\hookrightarrow \hookrightarrow H^{s-1}(B_{N})$. By the Aubin-Lions lemma combined with a diagonal argument and the uniqueness of the weak limit \eqref{eq:weakconvergence2}, there exists a sub-sequence $\{U_{n(k)}\}\subset\{U_{n(j)}\}$ such that, 
\begin{align}
\mbox{for any}\quad N\in\mathbb{N},\quad U_{n(k)}\rightarrow U\quad\mbox{in}\quad L^{2}([0,T_{0}];H^{s-1}(B_{N}))\quad\mbox{as}\quad k\rightarrow\infty. \label{eq:aubinlions}
\end{align}
To recover the compactness in the whole space $\mathbb{R}^{d}$ we show that the $L^{2}$-norm of $\{U_{n(k)}\}$ is tight. Let $\phi\in\mathcal{D}(\mathbb{R}^{d})$ be equal to 1 on the unit ball $B_{1}(0)$ and, for $N\in\mathbb{N}$, let $\phi_{N}(x)=\phi(\tfrac{x}{N})$ and $U_{k,N}=(1-\phi_{N})U_{n(k)}$. Since $U_{n(k)}=\mathcal{T}_{\alpha}(\widetilde{U}_{n(k)})$, for any $k\in\mathbb{N}$ it holds that
\begin{align*}
A^{0}(\widetilde{U}_{n(k)})\partial_{t}U_{n(k)}+A^{j}(\widetilde{U}_{n(k)})\partial_{j}U_{n(k)}&+D(\widetilde{U}_{n(k)})U_{n(k)}\\
&+B(\widetilde{U}_{n(k)})\Lambda^{2\alpha}U_{n(k)}+H(\widetilde{U}_{n(k)})[\Lambda^{2\alpha},G(\widetilde{U}_{n(k)})]\mathcal{L}U_{n(k)}=0,\\
&U_{n(k)}\rvert_{t=0}=U_{0}. 
\end{align*}
For any matrix coefficient $A$, we introduce the notation, $A_{k}=A(\widetilde{U}_{n(k)})$ and $A_{k,N}=(1-\phi_{N})A_{k}$. Multiplying the above equation by $1-\phi_{N}\geq 0$ and using the decomposition $U_{n(k)}=U_{k,N}+\phi_{N}U_{n(k)}$ in every term except $\Lambda^{2\alpha}U_{n(k)}$, gives
\begin{equation}
\label{eq:tailequation0}
\begin{aligned}
A^{0}_{k}\partial_{t}U_{k,N}+A^{j}_{k,N}\partial_{j}U_{k,N}&+D_{k,N}U_{k,N}+B_{k}(1-\phi_{N})\Lambda^{2\alpha}U_{n(k)}+H_{k}(1-\phi_{N})[\Lambda^{2\alpha},G_{k}]\mathcal{L}U_{k,N}\\
&=-(1-\phi_{N})A^{j}_{k}\partial_{j}(\phi_{N}U_{n(k)})-(1-\phi_{N})D_{k}(\phi_{N}U_{n(k)})\\
&-(1-\phi_{N})H_{k}[\Lambda^{2\alpha},G_{k}]\mathcal{L}(\phi_{N}U_{n(k)}).
\end{aligned}
\end{equation}
Before applying the energy method to this equation, observe that the coefficient $(1-\phi_{N})B_{k}$ no longer satisfies the structural assumptions of condition \ref{cond1}. To recover the required form of the diffusion term, we apply \eqref{eq:usefuldecomp} with $\omega=1-\phi_{N}$ and $V=U_{n(k)}$, so that
\begin{align*}
(1-\phi_{N})\Lambda^{2\alpha}U_{n(k)}=\Lambda^{2\alpha}U_{k,N}-\mathbb{K}_{2\alpha}(1-\phi_{N},U_{n(k)})-\sum_{0<|\beta|\leq 2\alpha}\frac{1}{\beta!}D^{\beta}(1-\phi_{N})\Lambda^{2\alpha,\beta}U_{n(k)}.
\end{align*}
Similarly, 
\begin{align*}
(1-\phi_{N})[\Lambda^{2\alpha},G_{k}]\mathcal{L}U_{k,N}&=[\Lambda^{2\alpha},(1-\phi_{N})G_{k}]\mathcal{L}U_{k,N}+\mathbb{K}_{2\alpha}(1-\phi_{N},G_{k}\mathcal{L}U_{k,N})\\
&-\sum_{0<\beta\leq 2\alpha}\frac{1}{\beta!}D^{\beta}(1-\phi_{N})\Lambda^{2\alpha,\beta}(G_{k}\mathcal{L}U_{k,N}).
\end{align*}
By using these identities in \eqref{eq:tailequation0} we conclude that $U_{k,N}$ satisfies the following Cauchy problem
\begin{equation}
\label{eq:tailequation1}
\begin{aligned}
A^{0}_{k}\partial_{t}U_{k,N}+A^{j}_{k,N}\partial_{j}U_{k,N}+D_{k,N}U_{k,N}&+B_{k}\Lambda^{2\alpha}U_{k,N}+H_{k}[\Lambda^{2\alpha},G_{k,N}]\mathcal{L}U_{k,N}=F_{k,N}+W_{k,N},\\
U_{k,N}\rvert_{t=0}&=(1-\phi_{N})U_{0},
\end{aligned}
\end{equation}
with, 
\begin{align*}
F_{k,N}&:=-(1-\phi_{N})A^{j}(\widetilde{U}_{n(k)})\partial_{j}(\phi_{N}U_{n(k)})-(1-\phi_{N})D(\widetilde{U}_{n(k)})(\phi_{N}U_{n(k)})\\
&-(1-\phi_{N})H(\widetilde{U}_{n(k)})[\Lambda^{2\alpha},G(\widetilde{U}_{n(k)})]\mathcal{L}(\phi_{N}U_{n(k)})
\end{align*}
and
\begin{align*}
W_{k,N}&:=B_{k}\mathbb{K}_{2\alpha}(1-\phi_{N},U_{n(k)})+H_{k}\mathbb{K}_{2\alpha}(1-\phi_{N},G_{k}\mathcal{L}U_{k,N})\\
&+B_{k}\sum_{0<|\beta|\leq 2\alpha}\frac{1}{\beta!}D^{\beta}(1-\phi_{N})\Lambda^{2\alpha,\beta}U_{n(k)}+H_{k}\sum_{0<|\beta|\leq 2\alpha}\frac{1}{\beta!}D^{\beta}(1-\phi_{N})\Lambda^{2\alpha,\beta}(G_{k}\mathcal{L}U_{k,N}).
\end{align*}
Then, by the $L^{2}$ energy estimate \eqref{eq:L2energy}, we have
\begin{equation}
\label{eq:energytail}
\begin{aligned}
\|U_{k,N}(t)\|_{L^{2}}^{2}&+\int_{0}^{T_{0}}\|\Lambda^{\alpha}v_{k,N}(\tau)\|_{L^{2}}^{2}d\tau\\
&\leq C(\mathcal{O}_{1},M)\left[\|(1-\phi_{N})U_{0}\|_{L^{2}}^{2}+\int_{0}^{T_{0}}\|F_{k,N}(\tau)\|_{L^{2}}^{2}+\|W_{k,N}(\tau)\|_{L^{2}}^{2}d\tau\right].
\end{aligned}
\end{equation}
Next, we estimate the source terms in \eqref{eq:energytail}. By the triangle inequality, 
\begin{align*}
&\int_{0}^{T_{0}}\|F_{k,N}(\tau)\|_{L^{2}}^{2}d\tau\\
&\leq 2\int_{0}^{T_{0}}\|(1-\phi_{N})A^{j}(\widetilde{U}_{n(k)})\partial_{j}(\phi_{N}U_{n(k)}-\phi_{N}U)(\tau)\|_{L^{2}}^{2}d\tau+2\int_{0}^{T_{0}}\|(1-\phi_{N})A^{j}(\widetilde{U}_{n(k)})\partial_{j}(\phi_{N}U(\tau))\|_{L^{2}}^{2}d\tau\\
&+ 2\int_{0}^{T_{0}}\|(1-\phi_{N})D(\widetilde{U}_{n(k)})(\phi_{N}U_{n(k)}-\phi_{N}U)(\tau)\|_{L^{2}}^{2}d\tau+2\int_{0}^{T_{0}}\|(1-\phi_{N})D(\widetilde{U}_{n(k)})\phi_{N}U(\tau)\|_{L^{2}}^{2}d\tau\\
&+ 2\int_{0}^{T_{0}}\|(1-\phi_{N})H(\widetilde{U}_{n(k)})[\Lambda^{2\alpha},G(\widetilde{U}_{n(k)})]\mathcal{L}(\phi_{N}U_{n(k)}-\phi_{N}U)(\tau)\|_{L^{2}}^{2}d\tau\\
&+2\int_{0}^{T_{0}}\|(1-\phi_{N})H(\widetilde{U}_{n(k)})[\Lambda^{2\alpha},G(\widetilde{U}_{n(k)})]\mathcal{L}(\phi_{N}U(\tau))\|_{L^{2}}^{2}d\tau
\end{align*}
and since $\widetilde{U}_{n(k)}\in X_{s,\alpha,T_{0}}(\mathcal{O}_{1},M,M_{1})$, for all $k\in\mathbb{N}$, the Sobolev product estimates imply 
\begin{equation}
\label{eq:energytail2}
\begin{aligned}
&\int_{0}^{T_{0}}\|F_{k,N}(\tau)\|_{L^{2}}^{2}d\tau\leq  C(\mathcal{O}_{1},M)\left[\int_{0}^{T_{0}}\|(1-\phi_{N})\nabla(\phi_{N}U_{n(k)}-\phi_{N}U)(\tau)\|_{L^{2}}^{2}d\tau+\int_{0}^{T_{0}}\|(1-\phi_{N})\nabla(\phi_{N}U(\tau))\|_{L^{2}}^{2}d\tau\right.\\
&+ \int_{0}^{T_{0}}\|(1-\phi_{N})(\phi_{N}U_{n(k)}-\phi_{N}U)(\tau)\|_{L^{2}}^{2}d\tau+\int_{0}^{T_{0}}\|(1-\phi_{N})\phi_{N}U(\tau)\|_{L^{2}}^{2}d\tau\\
&+ \int_{0}^{T_{0}}\|(1-\phi_{N})[\Lambda^{2\alpha},G_{k}]\mathcal{L}(\phi_{N}U_{n(k)}-\phi_{N}U)(\tau)\|_{L^{2}}^{2}d\tau\\
&+\left.\int_{0}^{T_{0}}\|(1-\phi_{N})[\Lambda^{2\alpha},G_{k}]\mathcal{L}(\phi_{N}U(\tau))\|_{L^{2}}^{2}d\tau\right]. 
\end{aligned}
\end{equation}
The fractional Leibniz rule in lemmas \ref{fractionalprodest} and \ref{dongli} together with the Sobolev embedding theorem yield
\begin{align}
&\|(1-\phi_{N})[\Lambda^{2\alpha},G(\widetilde{U}_{n(k)})]\mathcal{L}(\phi_{N}U_{n(k)}-\phi_{N}U)(\tau)\|_{L^{2}}\nonumber\\
&\leq C\left\lbrace\begin{array}{cc}
\|\Lambda^{2\alpha}G(\widetilde{U}_{n(k)})\|_{L^{\infty}}\|\phi_{N}(U_{n(k)}-U)\|_{L^{2}}&0<\alpha\leq\tfrac{1}{2},\nonumber\\
\|\Lambda^{2\alpha}G(\widetilde{U}_{n(k)})\|_{L^{2}}\|\phi_{N}(U_{n(k)}-U)\|_{L^{\infty}}+\|\nabla G(\widetilde{U}_{n(k)})\|_{L^{\infty}}\|\Lambda^{2\alpha-1}(\phi_{N}(U_{n(k)}-U))\|_{L^{2}}&\tfrac{1}{2}<\alpha<1
\end{array}\right.\nonumber\\
&\leq C(\mathcal{O}_{1},M)\left\lbrace\begin{array}{cc}
\|\phi_{N}(U_{n(k)}-U)\|_{L^{2}}&0<\alpha\leq\tfrac{1}{2},\\
\|\phi_{N}(U_{n(k)}-U)\|_{H^{s-1}}+\|\Lambda^{2\alpha-1}(\phi_{N}(U_{n(k)}-U))\|_{L^{2}}&\tfrac{1}{2}<\alpha<1.
\end{array}\right.\label{eq:commutatortail}
\end{align}
In order to estimate the last term in \eqref{eq:energytail2} we first apply the triangle inequality,
\begin{align*}
\|(1-\phi_{N})[\Lambda^{2\alpha},G_{k}]\mathcal{L}(\phi_{N}U)\|_{L^{2}}&\leq\|(1-\phi_{N})\Lambda^{2\alpha}(G_{k}\mathcal{L}(\phi_{N}U))\|_{L^{2}}+\|(1-\phi_{N})G_{k}\Lambda^{2\alpha}\mathcal{L}(\phi_{N}U)\|_{L^{2}}\\
&\leq\|(1-\phi_{N})\Lambda^{2\alpha}(G_{k}\mathcal{L}(\phi_{N}U))\|_{L^{2}}+C(\mathcal{O}_{1},M)\|(1-\phi_{N})\Lambda^{2\alpha}(\phi_{N}U)\|_{L^{2}}.
\end{align*}
For each term in the right hand side of the last inequality we apply \eqref{eq:usefuldecomp} and then, as a consequence of lemma \ref{dongli} and the Sobolev product estimates, we obtain
\begin{align*}
&\|(1-\phi_{N})\Lambda^{2\alpha}(G_{k}\mathcal{L}(\phi_{N}U))\|_{L^{2}}\leq \left\lVert\Lambda^{2\alpha}\left((1-\phi_{N})G_{k}\mathcal{L}(\phi_{N}U)\right)\right\rVert_{L^{2}}+\left\lVert \mathbb{K}_{2\alpha}(1-\phi_{N},G_{k}\mathcal{L}(\phi_{N}U))\right\rVert_{L^{2}}\\
&+\left\lVert \sum_{0<|\beta|\leq 2\alpha}\frac{1}{\beta!}D^{\beta}(1-\phi_{N})\Lambda^{2\alpha,\beta}(G(\widetilde{U}_{n(k)})\phi_{N}U)\right\rVert_{L^{2}}\\
&\leq C(\mathcal{O}_{1},M)\left(\|(1-\phi_{N})\phi_{N}U\|_{H^{2}}+\|\Lambda^{2\alpha}(1-\phi_{N})\|_{L^{\infty}}\|\phi_{N}U\|_{L^{2}}+\|\nabla(1-\phi_{N})\|_{L^{\infty}}\|\phi_{N}U\|_{H^{1}}\right)
\end{align*}
and 
\begin{align*}
\|(1-\phi_{N})\Lambda^{2\alpha}(\phi_{N}U)\|_{L^{2}}\leq C\left(\|(1-\phi_{N})\phi_{N}U\|_{H^{2}}+\|\Lambda^{2\alpha}(1-\phi_{N})\|_{L^{\infty}}\|\phi_{N}U\|_{L^{2}}+\|\nabla(1-\phi_{N})\|_{L^{\infty}}\|\phi_{N}U\|_{H^{1}}\right).
\end{align*}
Thus, 
\begin{equation}
\label{eq:energytail4}
\begin{aligned}
&\|(1-\phi_{N})[\Lambda^{2\alpha},G_{k}]\mathcal{L}(\phi_{N}U)\|_{L^{2}}\\
&\leq C(\mathcal{O}_{1},M)\left(\|(1-\phi_{N})\phi_{N}U\|_{H^{2}}+\|\Lambda^{2\alpha}(1-\phi_{N})\|_{L^{\infty}}\|\phi_{N}U\|_{L^{2}}+\|\nabla(1-\phi_{N})\|_{L^{\infty}}\|\phi_{N}U\|_{H^{1}}\right).
\end{aligned}
\end{equation}
Now, take the $\limsup_{k\rightarrow\infty}$ in \eqref{eq:energytail2}. In view of \eqref{eq:aubinlions}, \eqref{eq:commutatortail} and \eqref{eq:energytail4},
\begin{align*}
&\limsup_{k\rightarrow\infty}\int_{0}^{T_{0}}\|F_{k,N}(\tau)\|_{L^{2}}^{2}d\tau\leq C(\mathcal{O}_{1},M)\left(\int_{0}^{T_{0}}\|(1-\phi_{N})\nabla(\phi_{N}U(\tau))\|_{L^{2}}^{2}d\tau+\int_{0}^{T_{0}}\|(1-\phi_{N})\phi_{N}U(\tau)\|_{H^{2}}^{2}d\tau\right.\\
&+\left.\|\Lambda^{2\alpha}(1-\phi_{N})\|_{L^{\infty}}^{2}\int_{0}^{T_{0}}\|\phi_{N}U(\tau)\|_{L^{2}}^{2}d\tau+\|\nabla(1-\phi_{N})\|_{L^{\infty}}^{2}\int_{0}^{T_{0}}\|\phi_{N}U(\tau)\|_{H^{1}}^{2}d\tau\right)
\end{align*}
holds for any $N\in\mathbb{N}$. By expanding the derivative $\nabla(\phi_{N}U)$ and using the formulas,
\begin{align}
\|\Lambda^{2\alpha}(1-\phi_{N})\|_{L^{\infty}}=N^{-2\alpha}\|\Lambda^{2\alpha}\phi\|_{L^{\infty}}=\frac{N^{-2\alpha}}{(2\pi)^{d}}\||\cdot|^{2\alpha}\widehat{\phi}\|_{L^{1}}\quad\mbox{and}\quad\|\nabla(1-\phi_{N})\|_{L^{\infty}}=N^{-1}\|\nabla\phi\|_{L^{\infty}},\label{eq:decayratesf}
\end{align}
we find a constant $C_{4}=C_{4}(\mathcal{O}_{1},M,d,\phi,\alpha)>0$ such that
\begin{align*}
&\limsup_{k\rightarrow\infty}\int_{0}^{T_{0}}\|F_{k,N}(\tau)\|_{L^{2}}^{2}d\tau\leq C_{4}\left[\int_{0}^{T_{0}}\|(1-\phi_{N})U(\tau)\|_{L^{2}}^{2}+\|(1-\phi_{N})\nabla U(\tau)\|_{L^{2}}^{2}d\tau\right.\\
&+\int_{0}^{T_{0}}\|\nabla(1-\phi_{N})U\|_{L^{2}}^{2}+\|(1-\phi_{N})D^{2}U(\tau)\|_{L^{2}}^{2}+\|\nabla(1-\phi_{N})\nabla U(\tau)\|_{L^{2}}^{2}+\|D^{2}(1-\phi_{N})U(\tau)\|_{L^{2}}^{2} d\tau\\
&+\left.\left(\frac{1}{N^{4\alpha}}+\frac{1}{N^{2}}\right)\int_{0}^{T_{0}}\|U(\tau)\|_{H^{1}}^{2}d\tau\right]\\
&\leq C_{4}\left[\int_{0}^{T_{0}}\|(1-\phi_{N})U(\tau)\|_{L^{2}}^{2}+\|(1-\phi_{N})\nabla U(\tau)\|_{L^{2}}^{2}+\|(1-\phi_{N})D^{2}U(\tau)\|_{L^{2}}^{2}d\tau\right.\\
&+\left.\left(\frac{1}{N^{4\alpha}}+\frac{1}{N^{2}}+\frac{1}{N^{4}}\right)\int_{0}^{T_{0}}\|U(\tau)\|_{H^{1}}^{2}d\tau\right].
\end{align*}
Therefore, as a consequence of \eqref{eq:weakcomp1}, for any $N\in\mathbb{N}$ it holds that
\begin{equation}
\label{eq:limsupdecaytail}
\begin{aligned}
&\limsup_{k\rightarrow\infty}\int_{0}^{T_{0}}\|F_{k,N}(\tau)\|_{L^{2}}^{2}d\tau\leq C_{4} \left(\frac{1}{N^{4\alpha}}+\frac{1}{N^{2}}+\frac{1}{N^{4}}\right)T_{0}M^{2}\\
&+ C_{4}\left[\int_{0}^{T_{0}}\|(1-\phi_{N})U(\tau)\|_{L^{2}}^{2}+\|(1-\phi_{N})\nabla U(\tau)\|_{L^{2}}^{2}+\|(1-\phi_{N})D^{2}U(\tau)\|_{L^{2}}^{2}d\tau\right].
\end{aligned}
\end{equation}
Similarly, by lemma \ref{dongli} and the Sobolev product estimates it follows that
\begin{align*}
\int_{0}^{T_{0}}\|W_{k,N}(\tau)\|_{L^{2}}^{2}d\tau&\leq C(\mathcal{O}_{1},M)\left(\|\Lambda^{2\alpha}(1-\phi_{N})\|_{L^{\infty}}^{2}\int_{0}^{T_{0}}\|U_{n(k)}(\tau)\|_{L^{2}}^{2}+\|\nabla(1-\phi_{N})\|_{L^{\infty}}^{2}\int_{0}^{T_{0}}\|U_{n(k)}(\tau)\|_{H^{1}}^{2}d\tau\right).
\end{align*}
As a consequence of  \eqref{eq:decayratesf} and $\{U_{n(k)}\}\subset X_{s,\alpha,T_{0}}(\mathcal{O}_{1},M)$ we conclude that, for each $N\in\mathbb{N}$ the upper bound
\begin{equation}
\label{eq:limsupdecaytail2}
\int_{0}^{T_{0}}\|W_{k,N}(\tau)\|_{L^{2}}^{2}d\tau\leq C_{4}M^{2}T_{0}\left(\frac{1}{N^{4\alpha}}+\frac{1}{N^{2}}+\frac{1}{N^{4}}\right),
\end{equation}
holds for all $k\in\mathbb{N}$. Therefore, by using \eqref{eq:limsupdecaytail} and \eqref{eq:limsupdecaytail2}, in estimate \eqref{eq:energytail} we find that
\begin{equation}
\label{eq:finaldecaytail}
\begin{aligned}
\limsup_{k\rightarrow\infty}\int_{0}^{T_{0}}\|U_{k,N}(\tau)\|_{L^{2}}^{2}d\tau&\leq C_{4}\int_{0}^{T_{0}}\left(\|(1-\phi_{N})U(\tau)\|_{L^{2}}^{2}+\|(1-\phi_{N})\nabla U(\tau)\|_{L^{2}}^{2}+\|(1-\phi_{N})D^{2}U(\tau)\|_{L^{2}}^{2}\right)d\tau\\
&+C_{4}\left[\left(\frac{1}{N^{4\alpha}}+\frac{1}{N^{2}}+\frac{1}{N^{4}}\right)T_{0}M^{2}+\|(1-\phi_{N})U_{0}\|_{L^{2}}^{2}\right].
\end{aligned}
\end{equation}
Now, let us show that $U_{n(k)}\rightarrow U$ in $L^{2}([0,T_{0}];L^{2})$ as $k\rightarrow\infty$. For every $k,N\in\mathbb{N}$ we have that, 
\begin{align*}
&\int_{0}^{T_{0}}\|U_{n(k)}(\tau)-U(\tau)\|_{L^{2}}^{2}d\tau\leq 2\int_{0}^{T_{0}}\left(\|\phi_{N}(U_{n(k)}-U)(\tau)\|_{L^{2}}^{2}+\|(1-\phi_{N})(U_{n(k)}-U)(\tau)\|_{L^{2}}^{2}\right)d\tau\\
&\leq 2\left[\int_{0}^{T_{0}}\|\phi_{N}(U_{n(k)}-U)(\tau)\|_{L^{2}}^{2}d\tau+\int_{0}^{T_{0}}\|U_{k,N}(\tau)\|_{L^{2}}^{2}d\tau+\int_{0}^{T_{0}}\|(1-\phi_{N})U(\tau)\|_{L^{2}}^{2}d\tau\right].
\end{align*}
In particular, \eqref{eq:aubinlions} implies that, for any $N\in\mathbb{N}$, $\phi_{N}U_{n(k)}\rightarrow\phi_{N}U$ in $L^{2}([0,T_{0}];L^{2})$ and thus, by \eqref{eq:finaldecaytail}, 
\begin{align*}
&\limsup_{k\rightarrow\infty}\int_{0}^{T_{0}}\|U_{n(k)}(\tau)-U(\tau)\|_{L^{2}}^{2}d\tau\\
&\leq C_{4}\int_{0}^{T_{0}}\left(\|(1-\phi_{N})U(\tau)\|_{L^{2}}^{2}+\|(1-\phi_{N})\nabla U(\tau)\|_{L^{2}}^{2}+\|(1-\phi_{N})D^{2}U(\tau)\|_{L^{2}}^{2}\right)d\tau\\
&+C_{4}\left[\left(\frac{1}{N^{4\alpha}}+\frac{1}{N^{2}}+\frac{1}{N^{4}}\right)T_{0}M^{2}+\|(1-\phi_{N})U_{0}\|_{L^{2}}^{2}\right],
\end{align*}
which is valid for any $N\in\mathbb{N}$. Since $U\in L^{\infty}([0,T_{0}];H^{s})$ and $s\geq 2$, the dominated convergence theorem implies
\begin{align*}
\int_{0}^{T_{0}}\|(1-\phi_{N})D^{\ell}U(\tau)\|_{L^{2}}^{2}\rightarrow 0,\quad\ell=0,1,2\quad\mbox{and}\quad \|(1-\phi_{N})U_{0}\|_{L^{2}}\rightarrow 0,
\end{align*}
as $N\rightarrow\infty$. Therefore, given $\epsilon>0$ we can find $N\in\mathbb{N}$ sufficiently large so that
\begin{align}
\limsup_{k\rightarrow\infty}\int_{0}^{T_{0}}\|U_{n(k)}(\tau)-U(\tau)\|_{L^{2}}^{2}d\tau\leq\epsilon.\label{eq:L2convergence}
\end{align}
Since $\epsilon>0$ is arbitrary, it follows that $U_{n(k)}\rightarrow U\quad\mbox{in}\quad L^{2}([0,T_{0}];L^{2})$.

Next, we upgrade the convergence to  the space $L^{2}([0,T_{0}];H^{s-1})$. Let $\theta=\tfrac{s-1}{s}$. By interpolating between $L^{2}$ and $H^{s}$ we get
\begin{align*}
\int_{0}^{T_{0}}\|(U_{n(k)}-U)(\tau)\|_{H^{s-1}}^{2}d\tau\leq\int_{0}^{T_{0}}\|(U_{n(k)}-U)(\tau)\|_{L^{2}}^{2(1-\theta)}\|(U_{n(k)}-U)(\tau)\|_{H^{s}}^{2\theta}d\tau.
\end{align*}
Applying H\"older's inequality and using \eqref{eq:weakcomp1}, we obtain
\begin{align*}
\int_{0}^{T_{0}}\|(U_{n(k)}-U)(\tau)\|_{H^{s-1}}^{2}d\tau&\leq\left(\int_{0}^{T_{0}}\|(U_{n(k)}-U)(\tau)\|_{L^{2}}^{2}d\tau\right)^{1-\theta}\left(\int_{0}^{T_{0}}\|(U_{n(k)}-U)(\tau)\|_{H^{s}}^{2}d\tau\right)^{\theta}\\
&\leq (4T_{0}M^{2})^{\theta}\left(\int_{0}^{T_{0}}\|(U_{n(k)}-U)(\tau)\|_{L^{2}}^{2}d\tau\right)^{1-\theta}.
\end{align*}
Hence, \eqref{eq:L2convergence} implies that $U_{n(k)}\rightarrow U$ in $L^{2}([0,T_{0}];H^{s-1})$. 
Finally, let us consider the sequence $\{\Lambda^{\alpha}v_{n(k)}\}$, where $U_{n(k)}=(u_{n(k)},v_{n(k)})$. Interpolating between the homogeneous Sobolev spaces $\dot{H}^{-\alpha}$ and $\dot{H}^{s-1}$, with $\theta=\tfrac{\alpha}{s-1+\alpha}$, yields
\begin{align*}
\int_{0}^{T_{0}}\|\Lambda^{\alpha}(v_{n(k)}-v)(\tau)\|_{L^{2}}^{2}d\tau\leq\left(\int_{0}^{T_{0}}\|(v_{n(k)}-v)(\tau)\|_{L^{2}}^{2}d\tau\right)^{1-\theta}\left(\int_{0}^{T_{0}}\|\Lambda^{\alpha}(v_{n(k)}-v)(\tau)\|_{\dot{H}^{s-1}}^{2}d\tau\right)^{\theta}.
\end{align*} 
Using again \eqref{eq:weakcomp1} together with \eqref{eq:L2convergence}, we deduce that $\Lambda^{\alpha}v_{n(k)}\rightarrow\Lambda^{\alpha}v$ in $L^{2}([0,T_{0}];L^{2})$. Finally, interpolation between $L^{2}$ and $H^{s-1+\alpha}$ shows that the convergence actually holds in $L^{2}([0,T_{0}];H^{s-1})$. Indeed, choosing $\theta=\frac{s-1}{s-1+\alpha}$, we obtain
\begin{align*}
\int_{0}^{T_{0}}\|\Lambda^{\alpha}(v_{n(k)}-v)(\tau)\|_{H^{s-1}}^{2}d\tau\leq\left(\int_{0}^{T_{0}}\|\Lambda^{\alpha}(v_{n(k)}-v)(\tau)\|_{L^{2}}^{2}d\tau\right)^{1-\theta}\left(\int_{0}^{T_{0}}\|\Lambda^{\alpha}(v_{n(k)}-v)(\tau)\|_{H^{s}}^{2}d\tau\right)^{\theta}.
\end{align*}
Combining this estimate with \eqref{eq:weakcomp1} and the previous convergence result gives $\Lambda^{\alpha}v_{n(k)}\rightarrow \Lambda^{\alpha}v$ in $L^{2}([0,T_{0}];H^{s-1})$. Therefore, $U_{n(k)}\rightarrow U$ in $Y_{s-1,\alpha,T_{0}}$. Consequently, $\mathcal{T}_{\alpha}(X_{s,\alpha,T_{0}}(\mathcal{O}_{1},M,M_{1}))$ is relatively compact in $Y_{s-1,\alpha,T_{0}}$, which completes the proof. 
\end{proof}
\section{Local existence: The fixed point argument}
\label{fixedlocal}
In this section we use Schauder's fixed point theorem to deduce the local existence of solutions for the Cauchy problem \eqref{eq:quasilinearsystem1}. This theorem cannot be applied directly to $\mathcal{T}_{\alpha}:X_{s-1,\alpha,T_{0}}(\mathcal{O}_{1},M,M_{1})\rightarrow X_{s-1,\alpha,T_{0}}(\mathcal{O}_{1},M,M_{1})$ because $X_{s-1,\alpha,T_{0}}(\mathcal{O}_{1},M,M_{1})$ is not a closed set with respect to the topology of $Y_{s-1,\alpha,T_{0}}$. Therefore, we show that $\mathcal{T}_{\alpha}$ can be extended to the closure of its domain and then apply Schauder's fixed point theorem to this extension. Finally, we recover the required regularity for the solution by using the energy estimate. 

\begin{defi}
Let $s>\frac{d}{2}+1$ be an integer and $\alpha\in(0,1)$ be given. Assume that $\mathcal{O}_{1}$ and $M$ are defined by \eqref{eq:O1} and \eqref{eq:Mconstant}, respectively and that $T_{0}>0$ and $M_{1}$ satisfy the conditions of Theorem \ref{invariantsetsTheorem}. Define $X_{\infty}$ as the set of all functions $V=(w,z)^{\top}$ of $(x,t)\in Q_{T_{0}}$ with $w\in\mathbb{R}^{k}$ and $z\in\mathbb{R}^{m}$ such that 
\begin{equation}
\label{eq:Xinftyconditions}
\begin{aligned}
&V\in L^{\infty}([0,T_{0}];H^{s}),\quad z\in L^{2}([0,T_{0}];H^{s+\alpha}),\\
&\partial_{t}V\in L^{2}([0,T_{0}];H^{s-1}),\quad V(x,t)\in\mathcal{O}_{1}\quad\mbox{for all}\quad(x,t)\in Q_{T_{0}},
\end{aligned}
\end{equation}
and for all $t\in[0,T_{0}]$, 
\begin{align}
E^{2}_{s,\alpha}(V(t))\leq M^{2}\quad\mbox{and}\quad\int_{0}^{t}\|\partial_{t}V(\tau)\|_{H^{s-1}}^{2}d\tau\leq M_{1}^{2}.\label{eq:energyestimatesinfty}
\end{align}
\end{defi}
\begin{prop}
\label{closure}
The set $X_{\infty}$ is the closure of $X_{s,\alpha,T_{0}}(\mathcal{O}_{1},M,M_{1})$ with respect to the topology of $Y_{s-1,\alpha,T_{0}}$. Furthermore, $X_{\infty}$ can be characterized as the set of all vectors $U\in Y_{s-1,\alpha,T_{0}}$ for which there is a sequence $\{U_{k}\}\subset X_{s-1,\alpha,T_{0}}(\mathcal{O}_{1},M,M_{1})$ and $V\in Y_{s-1,\alpha,T_{0}}$ such that 
\begin{align}
U_{k}\rightarrow U\quad\mbox{and}\quad \mathcal{T}_{\alpha}(U_{k})\rightarrow V\quad\mbox{in}\quad Y_{s-1,\alpha,T_{0}}.\label{eq:domainextension}
\end{align}
\end{prop}
\begin{proof}
Let $V\in\overline{X_{s-1,\alpha,T_{0}}(\mathcal{O}_{1},M,M_{1})}$. Then, there is $\{V_{k}\}\subset X_{s,\alpha,T_{0}}(\mathcal{O}_{1},M,M_{1})$ such that $V_{k}\rightarrow V$ in $Y_{s-1,\alpha,T_{0}}$. Since $\{V_{k}\}$ is bounded and satisfies \eqref{eq:boundenergy1} and \eqref{eq:timeuniformderivative}, we can extract a subsequence  $\{V_{k(j)}\}\subset\{V_{k}\}$ that satisfies \eqref{eq:weakconvergence}, \eqref{eq:weakcomp1} and \eqref{eq:weakcomp2} with $V$ replacing $U$. Since $T_{0}$ satisfies Theorem \ref{invariantsetsTheorem}, it follows that $V(x,t)\in\mathcal{O}_{1}$ for all $(x,t)\in Q_{T_{0}}$. Hence, $V\in X_{\infty}$. On the other hand, let $V\in X_{\infty}$. Consider the sequence $V_{k}:=\mathbb{J}_{\epsilon_{k}}V$ with $\epsilon_{k}\rightarrow 0$. Since $\partial_{t}V\in L^{2}([0,T];H^{s-1})$ it follows that $V\in\mathcal{C}([0,T_{0}];H^{s-1})$. Consequently, $V_{k}\in\mathcal{C}([0,T_{0}];H^{s})$ for all $k\in\mathbb{N}$. It is easy to verify that $\{V_{k}\}$ satisfies \eqref{eq:Xinftyconditions} and \eqref{eq:energyestimatesinfty}. Therefore, $\{V_{k}\}\subset X_{s,\alpha,T_{0}}(\mathcal{O}_{1},M,M_{1})$. Moreover, the standard properties of the mollifier $\mathbb{J}_{\epsilon_{k}}V$ combined with the dominated convergence theorem imply that $V_{k}\rightarrow V$ in $Y_{s-1,\alpha,T_{0}}$. Thus, $X_{\infty}=\overline{X_{s-1,\alpha,T_{0}}(\mathcal{O}_{1},M,M_{1})}$. 

In order to prove the second statement, let $X_{0}$ be the set of vectors in $Y_{s-1,\alpha,T_{0}}$ that satisfies \eqref{eq:domainextension}. Then, by the first part of the proposition and the definition of $X_{0}$, we have that $X_{0}\subset X_{\infty}$. On the other hand, let $V\in X_{\infty}$. There is $\{V_{k}\}\subset X_{s,\alpha,T_{0}}(\mathcal{O}_{1},M,M_{1})$ such that $V_{k}\rightarrow V$ in $Y_{s-1,\alpha,T_{0}}$. Then, by the H\"older continuity described in \eqref{eq:contractioninequality1h} and the embedding, $L^{\infty}([0,T_{0}];H^{s-1})\hookrightarrow L^{2}([0,T_{0}];H^{s-1})$, we deduce that $\{\mathcal{T}_{\alpha}(V_{k})\}$ is a Cauchy sequence in $Y_{s-1,\alpha,T_{0}}$. In consequence, there is $V\in Y_{s-1,\alpha,T_{0}}$ such that \eqref{eq:domainextension} holds and thus, $X_{\infty}\subset X_{0}$. This concludes the proof.
\end{proof}
In the next result we establish the existence of an extension of $\mathcal{T}_{\alpha}$ to the closure of its domain.
\begin{prop}
\label{extension}
There is an extension $\mathcal{T}_{\alpha,\infty}$ of $\mathcal{T}_{\alpha}$ to the set $X_{\infty}$ defined as $\mathcal{T}_{\alpha,\infty}(U)=V$, where $U$ and $V$ satisfy \eqref{eq:domainextension} with the following properties:
\begin{itemize}
	\item [(1)] The mapping $U\xmapsto{T_{\alpha,\infty}} V$ is well defined in the sense that it is independent of the Cauchy sequence that satisfies $U_{k}\rightarrow U$. Moreover, $U$ and $V$ satisfy the following Cauchy problem
\begin{equation}
\label{eq:linearextension}
\begin{aligned}
A^{0}(U)\partial_{t}V+A^{j}(U)\partial_{j}V&+D(U)V+B(U)\Lambda^{2\alpha}V+H(U)[\Lambda^{2\alpha},G(U)]\mathcal{L}V=0,\\
V\rvert_{t=0}&=U_{0}.
\end{aligned}
\end{equation}
	\item [(2)] There is a positive constant $C_{5}(\mathcal{O}_{1},M,T_{0})$ such that 
	\begin{align}
\|\mathcal{T}_{\alpha,\infty}(U_{1})-\mathcal{T}_{\alpha,\infty}(U_{2})\|_{Y_{s-1,\alpha,T_{0}}}^{2}\leq C_{5}(\mathcal{O}_{1},M,T_{0})\|U_{1}-U_{2}\|_{Y_{s-1,\alpha,T_{0}}}\quad\mbox{for all}\quad U_{1},U_{2}\in X_{\infty}.\label{eq:holderextension}
	\end{align}
\item [(3)] $\mathcal{T}_{\alpha,\infty}(X_{\infty})\subset X_{\infty}$ and $\mathcal{T}_{\alpha,\infty}(U)=\mathcal{T}_{\alpha}(U)$ for all $U\in X_{s,\alpha,T_{0}}(\mathcal{O}_{1},M,M_{1})$. 
\item [(4)] $\mathcal{T}_{\alpha,\infty}(X_{\infty})$ is relatively compact with respect to the topology of $Y_{s-1,\alpha,T_{0}}$.
\end{itemize}
\end{prop}
\begin{proof}
Let $U\in X_{\infty}$. There is $\{U_{k}\}\subset X_{s,\alpha,T_{0}}$ and $V$ in $Y_{s-1,\alpha,T_{0}}$ such that $U_{k}\rightarrow U$ and $\mathcal{T}_{\alpha}(U_{k})\rightarrow V$ in $Y_{s-1,\alpha,T_{0}}$. By the definition of $\mathcal{T}_{\alpha}$, if we set $V_{k}:=\mathcal{T}_{\alpha}(U_{k})$, we have that
\begin{equation}
\label{eq:approximatelinearinfty}
\begin{aligned}
A^{0}(U_{k})\partial_{t}V_{k}+A^{j}(U_{k})\partial_{j}V_{k}&+D(U_{k})V_{k}+B(U_{k})\Lambda^{2\alpha}V_{k}+H(U_{k})[\Lambda^{2\alpha},G(U_{k})]\mathcal{L}V_{k}=0,\\
V_{k}\rvert_{t=0}&=U_{0},
\end{aligned}
\end{equation}
holds for all $k\in\mathbb{N}$. By taking the limit when $k\rightarrow\infty$ and using the convergence in \eqref{eq:domainextension} we find that $V$ satisfies the Cauchy problem \eqref{eq:linearextension}. If there is another sequence $\{\widetilde{U}_{k}\}\subset X_{s,\alpha,T_{0}}(\mathcal{O}_{1},M,M_{1})$ such that $\widetilde{U}_{k}\rightarrow U$ and $\widetilde{V}_{k}:=\mathcal{T}_{\alpha}(\widetilde{U}_{k})\rightarrow\widetilde{V}$ in $Y$ then, as in the previous lines, $\widetilde{V}$ satisfies \eqref{eq:linearextension}. Since $V,\widetilde{V}\in X_{\infty}$, we can apply the energy method to conclude that $E_{s,\alpha,T_{0}}(V-\widetilde{V})=0$. Therefore, $\widetilde{V}=V$ and $\mathcal{T}_{\alpha,\infty}$ is well-defined. This concludes $(1)$. 

Next, we prove statement $(2)$. Let $U_{1},U_{2}\in X_{\infty}$. Then, there are sequences $\{U_{1,k}\}$ and $\{U_{2,k}\}$ in $X_{s,\alpha,T_{0}}(\mathcal{O}_{1},M,M_{1})$ and vectors $V_{1},V_{2}\in Y_{s-1,\alpha,T_{0}}$ such that $U_{1,k}\rightarrow U_{1}$, $U_{2,k}\rightarrow U_{2}$, $\mathcal{T}_{\alpha}(U_{1,k})\rightarrow V_{1}$ and $\mathcal{T}_{\alpha}(U_{2,k})\rightarrow V_{2}$. As a consequence of \eqref{eq:contractioninequality1h} and the embedding $L^{\infty}([0,T_{0}];H^{s-1})\hookrightarrow L^{2}([0,T_{0}];H^{s-1})$, it holds that
\begin{align*}
\|\mathcal{T}_{\alpha}(U_{1,k})-\mathcal{T}_{\alpha}(U_{2,k})\|_{Y_{s-1,\alpha,T_{0}}}^{2}\leq C_{5}(\mathcal{O}_{1},M,T_{0})\|U_{1,k}-U_{2,k}\|_{Y_{s-1,\alpha,T_{0}}}\quad\mbox{for all}\quad k\in\mathbb{N}. 
\end{align*}
After taking the limit when $k\rightarrow \infty$, \eqref{eq:holderextension} follows.\\
The proof of statement $(3)$ follows immediately from the definition of $\mathcal{T}_{\alpha,\infty}$ and proposition \ref{closure}.

Finally, we prove statement $(4)$. Observe that, $\mathcal{T}_{\alpha,\infty}(X_{\infty})\subset\overline{\mathcal{T}_{\alpha}(X_{s,\alpha,T_{0}}(\mathcal{O}_{1},M,M_{1}))}$ holds as a consequence of \eqref{eq:domainextension}. Therefore, lemma \ref{compactness} implies that the closure of $\mathcal{T}_{\alpha,\infty}(X_{\infty})$ is contained in a compact set in $Y_{s-1,\alpha,T_{0}}$. The result follows. 
\end{proof}

\begin{theo}[Local existence]
\label{localexistence}
Let $\alpha\in(0,1)$ and $s>\tfrac{d}{2}+1$ be an integer. Assume that condition \ref{cond1} is satisfied with one of the hypotheses (E1)-(E2). Assume that $U_{0}\in H^{s}$ satisfies condition \ref{initialcondition}. Then, there is $T_{0}>0$ and solution $U=(u,v)$ to the Cauchy problem \eqref{eq:quasilinearsystem1} in the set $X_{s,\alpha,T_{0}}(\mathcal{O}_{1},M,M_{1})$.
\end{theo}
\begin{proof}
First, define $\mathcal{O}_{1}$ and $M$ as in \eqref{eq:O1} and \eqref{eq:Mconstant}, respectively. By Theorem \ref{invariantsetsTheorem}, there is $T_{0}>0$ such that the solution map $\mathcal{T}_{\alpha}:X_{s,\alpha,T_{0}}(\mathcal{O}_{1},M,M_{1})\rightarrow X_{s,\alpha,T_{0}}(\mathcal{O}_{1},M,M_{1})$ is well defined. Consider the extension $\mathcal{T}_{\alpha,\infty}$ of $\mathcal{T}_{\alpha}$ to the set $X_{\infty}$ defined in proposition \ref{extension}. It is easy to verify that $X_{\infty}$ is a convex set. Consequently, by proposition \ref{extension}, we can apply Schauder's fixed point theorem  to conclude the existence of a fixed point $U_{\infty}\in X_{\infty}$ of $\mathcal{T}_{\alpha,\infty}$ (see \cite{morrisch}, for example). Thus, $\mathcal{T}_{\infty,\alpha}(U_{\infty})=U_{\infty}$ and, as a consequence of statement $(1)$ in proposition \ref{extension}, $U_{\infty}$ satisfies the Cauchy problem
\begin{align*}
A^{0}(U_{\infty})\partial_{t}U_{\infty}+A^{j}(U_{\infty})\partial_{j}U_{\infty}&+D(U_{\infty})U_{\infty}+B(U_{\infty})\Lambda^{2\alpha}U_{\infty}+H(U_{\infty})[\Lambda^{2\alpha},G(U_{\infty})]\mathcal{L}U_{\infty}=0,\\
U_{\infty}\rvert_{t=0}&=U_{0}.
\end{align*}
Since $U_{\infty}\in X_{\infty}$, Theorem \ref{mollifiedapprox} implies that $U_{\infty}\in\mathcal{C}([0,T_{0}];H^{s})$. Hence, $U_{\infty}\in X_{s,\alpha,T_{0}}(\mathcal{O}_{1},M,M_{1})$. 
\end{proof}

\section{Uniqueness}
\label{uniqueness}
Given the lack of a contraction type inequality for the solution map associated with system \eqref{eq:quasilinearsystem1}, we can only show uniqueness of solutions under the assumption of smallness of the initial data.
	
\begin{theo}
\label{localwellposedness}
Let $s>\frac{d}{2}+1$ be an integer and let $\alpha\in(0,1)$ be given. Assume that condition \ref{cond1} is satisfied with a constant matrix  $H=H_{c}$. Let $U_{0}\in H^{s}$ satisfy condition \ref{initialcondition} and define $\mathcal{O}_{1}$, $M$ and $M_{1}$ as in section \ref{nonlinear}. Then, there is $T_{0}>0$, such that the Cauchy problem \eqref{eq:quasilinearsystem1} has a unique solution $U=(u,v)$ in $X_{s,\alpha, T_{1}}(\mathcal{O}_{1},M,M_{1})$ if one of the following conditions holds true: 
\begin{itemize}
\item [(i)] $H_{c}=\mathbb{O}_{N}$.
\item [(ii)] One of the hypotheses (E1)-(E2) is satisfied with a constant block matrix $A^{0}_{2}=A^{0}_{2,c}$. Moreover, let $\kappa_{\sigma,d}$ be de norm of the embedding $H^{\sigma}\hookrightarrow L^{\infty}$ whenever $\sigma>\frac{d}{2}$ and define the constant 
\begin{align*}
\kappa(\alpha,s,d):=\left\lbrace\begin{array}{ccc}
\kappa_{s-\alpha,d}&\mbox{if}&\mbox{(E1) holds},\\
\kappa_{s-s_{\alpha},d}&\mbox{if}&\mbox{(E2) holds}.
\end{array}\right.
\end{align*}
Assume the following condition,
\begin{equation}
\label{eq:constantuniqueness}
0<\omega_{0}:=2C_{0}(\mathcal{O}_{1})|H_{c}|^{2}C_{G}^{2}(\mathcal{O}_{1},M,s,\alpha)\kappa^{2}(\alpha,s,d)|\mathcal{L}|^{2}M^{2}<1,
\end{equation}
where $C_{G}(\mathcal{O}_{1},M,s,\alpha)$ is the constant in \eqref{eq:Lipschitzcontinuityfinal} with $\sigma=s-1-\alpha$. 
\end{itemize}

\end{theo}
\begin{proof}
Define $\mathcal{O}_{1}$, $M$ and $M_{1}$ as in section \ref{nonlinear}. Then, by Theorem \ref{localexistence}, there exists $T_{0}>0$ and a solution $U=(u,v)\in X_{s,\alpha, T_{0}}(\mathcal{O}_{1},M,M_{1})$. Let $U^{1}$ and $U^{2}$ be two different solution of the Cauchy problem \eqref{eq:quasilinearsystem1} in $X_{s,\alpha,T_{0}}(\mathcal{O}_{1},M,M_{1})$. Then, $U:=U^{1}-U^{2}$ satisfies \eqref{eq:contractionsystem} and \eqref{eq:contractionsource} with $\widetilde{U}^{1}$ and $\widetilde{U}^{2}$ replaced by $U^{1}$ and $U^{2}$, respectively. By the block structure in condition $(C5)$, we have that $F^{1,2}=(A^{0}(U^{1}))^{-1}(f^{1,2}_{1},f^{1,2}_{2})$ where $(f^{1,2}_{1},f^{1,2}_{2})\in\mathbb{R}^{k}\times\mathbb{R}^{m}$. Therefore, by the energy estimates with $m=s-1$, 
\begin{equation}
\label{eq:contractionuniqueness}
\begin{aligned}
\sup_{\tau\in[0,t]}\|U(\tau)\|_{H^{s-1}}^{2}&+\int_{0}^{t}\|\Lambda^{\alpha}(v^{1}-v^{2})(\tau)\|_{H^{s-1}}d\tau\\
&\leq 2C_{0}(\mathcal{O}_{1})\int_{0}^{t}\left(\|f_{1}^{1,2}(\tau)\|_{H^{s-1}}^{2}+\|f_{2}^{1,2}(\tau)\|_{H^{s-1-\alpha}}^{2}+\|Q^{1,2}(\tau)\|_{H^{s-1-\alpha}}^{2}\right)d\tau.
\end{aligned}
\end{equation}
Since $U^{1},U^{2}\in X_{s,\alpha,T_{0}}(\mathcal{O}_{1},M,M_{1})$, 
\begin{align}
\|f_{1}^{1,2}(\tau)\|_{H^{s-1}}\leq C(\mathcal{O}_{1},M)M\|U^{1}-U^{2}\|_{H^{s-1}} \label{eq:fcontraction1}
\end{align}
and, by \eqref{eq:fractionalproduct2}, we have
\begin{align*}
\|f^{1,2}_{2}\|_{H^{s-1-\alpha}}
&\leq C(\mathcal{O}_{1},M)\left(M\|U^{1}-U^{2}\|_{H^{s-1}}+\|\overline{B}_{0}(U^{2})-\overline{B}_{0}(U^{1})\|_{\dot{H}^{s-1-\alpha}}\|\Lambda^{2\alpha}v^{2}\|_{L^{\infty}}\right.\\
&\left.+\|\overline{B}_{0}(U^{2})-\overline{B}_{0}(U^{1})\|_{L^{\infty}}\|v^{2}\|_{H^{s}}\right).
\end{align*}
Then, by \eqref{eq:Lipschitzcontinuityfinal} in lemma \ref{continuityHsigma} and the embedding $H^{s-\alpha}\hookrightarrow L^{\infty}$,
\begin{equation}
\label{eq:fcontraction2}
\begin{aligned}
\|f_{2}^{1,2}\|_{H^{s-1-\alpha}}\leq C(\mathcal{O}_{1},M)\|U^{1}-U^{2}\|_{H^{s-1}}\left(M+\|\Lambda^{\alpha}v^{2}\|_{H^{s}}\right).
\end{aligned}
\end{equation}
If we assume $(i)$, we have $Q^{1,2}=0$. Thus, by using \eqref{eq:fcontraction1} and \eqref{eq:fcontraction2} in \eqref{eq:contractionuniqueness}, yields
\begin{align*}
\sup_{\tau\in[0,t]}\|U(\tau)\|_{H^{s-1}}^{2}&+\int_{0}^{t}\|\Lambda^{\alpha}(v^{1}-v^{2})(\tau)\|_{H^{s-1}}d\tau\\
&\leq C(\mathcal{O}_{1},M)\sup_{\tau\in[0,t]}\|U(\tau)\|_{H^{s-1}}^{2}\left[t+\int_{0}^{t}\|\Lambda^{\alpha}v^{2}(\tau)\|_{H^{s}}^{2}\right].
\end{align*}
Since $\lim_{t\rightarrow 0}\int_{0}^{t}\|\Lambda^{\alpha}v^{2}(\tau)\|_{H^{s}}^{2}d\tau=0$, there is $0<T_{1}\leq T_{0}$ such that
\begin{align*}
	0<\sup_{t\in[0,T_{1}]}\left[t+\int_{0}^{t}\|\Lambda^{\alpha}v^{2}(\tau)\|_{H^{s}}^{2}d\tau\right]<1.
\end{align*}
Therefore, it follows that $U^{1}\equiv U^{2}$ on $[0,T_{1}]$. Then, by the uniform continuity of $f(t)=\int_{0}^{t}\|\Lambda^{\alpha}v^{2}(\tau)\|_{H^{s}}d\tau$, we can iterate the argument with initial data in $t=T_{1}, 2T_{1},..$ and thus,  uniqueness follows in $[0,T_{0}]$.

Now assume that $(ii)$ holds and let $\mathcal{L}U^{2}=(y^{2},w^{2})\in\mathbb{R}^{k}\times\mathbb{R}^{m}$. By using the embedding $H^{s-1}\hookrightarrow H^{s-1-\alpha}$, we can estimate the first term in $Q^{1,2}$ via the Sobolev product estimates together with \eqref{eq:fraccommnon},
\begin{align*}
\|Q^{1,2}\|_{H^{s-1-\alpha}}\leq C(\mathcal{O}_{1},M)M\|U^{1}-U^{2}\|_{H^{s-1}}(M+\|\Lambda^{\alpha}v^{1}\|_{H^{s}})+\|\overline{H}(U^{2})[\Lambda^{2\alpha},G(U^{1})-G(U^{2})]\mathcal{L}U^{2}\|_{H^{s-1-\alpha}}.
\end{align*}
Let $Q^{1,2}_{0}$ be the second term in the last inequality. Since the block matrix $A^{0}_{2}=A^{0}_{2,c}$ is constant, the block structure in condition (C6) implies
\begin{align*}
\|Q^{1,2}_{0}\|_{H^{s-1-\alpha}}
&\leq|H_{c}|\left(\|(G_{0}(v^{1})-G_{0}(v^{2}))w^{2}\|_{H^{2}}+\|(G_{0}(v^{1})-G_{0}(v^{2}))\Lambda^{2\alpha}w^{2}\|_{L^{2}}\right.\\
&+\left.\|[\Lambda^{2\alpha},G_{0}(v^{1})-G_{0}(v^{2})]w^{2}\|_{\dot{H}^{s-1-\alpha}}\right)\\
&\leq|H_{c}|\left(C_{G}(\mathcal{O}_{1},M)M\|U^{1}-U^{2}\|_{H^{s-1}}+\|[\Lambda^{2\alpha},G_{0}(v^{1})-G_{0}(v^{2})]w^{2}\|_{\dot{H}^{s-1-\alpha}}\right).
\end{align*}
Here, we use that $s-1\geq 2$ and $U^{2}\in X_{s,\alpha,T_{0}}(\mathcal{O}_{1},M,M_{1})$.
Then, by the identity
\begin{align*}
\Lambda^{s-1-\alpha}[\Lambda^{2\alpha},G_{0}(v^{1})-G_{0}(v^{2})]w^{2}&=[\Lambda^{s-1+\alpha},G_{0}(v^{1})-G_{0}(v^{2})]w^{2}\\
&-[\Lambda^{s-1-\alpha},G_{0}(v^{1})-G_{0}(v^{2})]\Lambda^{2\alpha}w^{2},
\end{align*}
together with \eqref{eq:fractionalcommutator2} and the embeddings $H^{1}\hookrightarrow L^{4}$, $H^{s}\hookrightarrow L^{\infty}$,  we obtain
\begin{align*}
&\|[\Lambda^{2\alpha},G_{0}(v^{1})-G_{0}(v^{2})]w^{2}\|_{\dot{H}^{s-1-\alpha}}\\
&\leq C\left(\|G_{0}(v^{1})-G_{0}(v^{2})\|_{H^{2}}\|\mathcal{L}U^{2}\|_{H^{s}}+\|G_{0}(v^{1})-G_{0}(v^{2})\|_{\dot{H}^{s-1+\alpha}}\|\mathcal{L}U^{2}\|_{H^{s}}\right.\\
&+\left.\|G_{0}(v^{1})-G_{0}(v^{2})\|_{\dot{H}^{s-1-\alpha}}\|\Lambda^{2\alpha}\mathcal{L}U^{2}\|_{L^{\infty}}.\right)
\end{align*}
Therefore, the embedding $H^{s-1}\hookrightarrow H^{2}$  and  estimate \eqref{eq:Lipschitzcontinuityfinal} yield
\begin{align*}
&\|[\Lambda^{2\alpha},G_{0}(v^{1})-G_{0}(v^{2})]\mathcal{L}U^{2}\|_{\dot{H}^{s-1-\alpha}}\\
&\leq C_{G}(\mathcal{O}_{1},M,s,\alpha)\left(\|U^{1}-U^{2}\|_{H^{s-1}}|\mathcal{L}|M+\|\Lambda^{\alpha}(v^{1}-v^{2})\|_{H^{s-1}}|\mathcal{L}|M+\|U^{1}-U^{2}\|_{H^{s-1}}\|\Lambda^{2\alpha}\mathcal{L}U^{2}\|_{L^{\infty}}\right).
\end{align*}
Notice that, under any of the hypotheses (E1)-(E2), the following bound holds 
\begin{align*}
\|\Lambda^{2\alpha}\mathcal{L}U^{2}\|_{L^{\infty}}\leq \kappa(\alpha,d)|\mathcal{L}|\left(\|\Lambda^{\alpha}v^{2}\|_{H^{s}}+\|U^{2}\|_{H^{s}}\right).
\end{align*}
Thus, we have
\begin{align*}
&\|[\Lambda^{2\alpha},G_{0}(v^{1})-G_{0}(v^{2})]w^{2}\|_{\dot{H}^{s-1-\alpha}}\\
&\leq C_{G}(\mathcal{O}_{1},M,s,\alpha) \kappa(\alpha,d)|\mathcal{L}|\left(\left[M+\|\Lambda^{\alpha}v^{2}\|_{H^{s}}\right]\|U^{1}-U^{2}\|_{H^{s-1}}+\|\Lambda^{\alpha}(v^{1}-v^{2})\|_{H^{s-1}}M\right)
\end{align*}
and 
\begin{align*}
\|Q^{1,2}_{0}\|_{H^{s-1-\alpha}}&\leq|H_{c}|C_{G}(\mathcal{O}_{1},M,s,\alpha)\left(M+\kappa(\alpha,s,d)|\mathcal{L}|[M+\|\Lambda^{\alpha}v^{2}\|_{H^{s}}]\right)\|U^{1}-U^{2}\|_{H^{s-1}}\\
&+|H_{c}|C_{G}(\mathcal{O}_{1},M,s,\alpha)\kappa(\alpha,s,d)|\mathcal{L}|M\|\Lambda^{\alpha}(v^{1}-v^{2})\|_{H^{s-1}}.
\end{align*}
Consequently, 
\begin{equation}
\label{eq:fcontraction3}
\begin{aligned}
2C_{0}(\mathcal{O}_{1})\int_{0}^{t}\|Q^{1,2}(\tau)\|_{H^{s-1-\alpha}}^{2}d\tau&\leq C_{1}(\mathcal{O}_{1},M)\sup_{\tau\in[0,t]}\|U(\tau)\|_{H^{s-1}}^{2}\left[t+\int_{0}^{t}\|\Lambda^{\alpha}v^{2}(\tau)\|_{H^{s}}^{2}+\|\Lambda^{\alpha}v^{1}(\tau)\|_{H^{s}}^{2}d\tau\right]\\
&+\omega_{0}\int_{0}^{t}\|\Lambda^{\alpha}(v^{1}-v^{2})(\tau)\|_{H^{s-1}}^{2}d\tau
\end{aligned}
\end{equation}
By using estimates \eqref{eq:fcontraction1}, \eqref{eq:fcontraction2} and \eqref{eq:fcontraction3} in \eqref{eq:contractionuniqueness} together with the assumption \eqref{eq:constantuniqueness},  we conclude that
\begin{align*}
\sup_{\tau\in[0,t]}\|U(\tau)\|_{H^{s-1}}^{2}\leq C(\mathcal{O}_{1},M)\sup_{\tau\in[0,t]}\|U(\tau)\|_{H^{s-1}}^{2}\left[t+\int_{0}^{t}\|\Lambda^{\alpha}v^{2}(\tau)\|_{H^{s}}^{2}+\|\Lambda^{\alpha}v^{1}(\tau)\|_{H^{s}}^{2}d\tau\right].
\end{align*}
Then, as in case $(i)$, uniqueness follows.
\end{proof}
\begin{rem}
Observe that from theorem \ref{localexistence} together with statement $(i)$ in theorem \ref{localwellposedness}, the Cauchy problem for 
\begin{align*}
A^{0}(U)\partial_{t}U+A^{j}(U)\partial_{j}U+D(U)U&+B(U)\Lambda^{2\alpha}U=0,\\
	U\rvert_{t=0}&=U_{0},
\end{align*}
is locally well-posed in $X_{s,\alpha,T_{1}}(\mathcal{O}_{1},M,M_{1})$ for arbitrary initial data $U_{0}$ in $H^{s}$ and for any $\alpha\in(0,1)$. 
\end{rem}

\section{Compressible Euler system with singular velocity alignment}\label{alignment section}
Let us consider system \eqref{eq:Eulervelal} with kernel given by \eqref{eq:fractionalkernel} with $d=3$ and $N=4$. We are interested in writing system \eqref{eq:Eulervelal} in the symmetric form \eqref{eq:quasilinearsystem1} satisfying condition \ref{cond1}. 

We begin by considering the vector state of variables $V=(\rho,v)^{\top}\in\mathcal{O}_{V}\subset\mathbb{R}^{4}$ where $\mathcal{O}_{V}$ is an open convex set defined as
\begin{align*}
\mathcal{O}_{V}=\{(\rho,v)^{\top}\in\mathbb{R}^{4}~|~\rho>0\quad\mbox{and}\quad v\in\mathbb{R}^{3}\}.
\end{align*}
We make the standard assumption:
\begin{condition}[\textbf{T}]
The pressure field $p$ varies within the domain $\mathcal{D}=\{\rho\in\mathbb{R}~|~\rho>0\}$, $p>0$, $p^{\prime}(\rho)>0$ and $p\in\mathcal{C}^{\infty}(\mathcal{D})$.
\end{condition}	
We first write system \eqref{eq:Eulervelal} in the form 
\begin{align}
\widetilde{A}^{0}(V)\partial_{t}V+\widetilde{A}^{j}(V)\partial_{j}V+\widetilde{D}(V)V+\widetilde{B}(V)\Lambda^{2\alpha}V+\widetilde{H}(V)[\Lambda^{2\alpha},\widetilde{G}(V)]\mathcal{L}_{4}V=0.\label{eq:quasilinearEulerv}
\end{align}
By \eqref{eq:fractionalPV}, the non-local velocity alignment term can be reformulated as
\begin{align*}
-\rho(x,t)c_{\alpha}\int_{\mathbb{R}^{d}}\phi(x-y)(v(x,t)-v(y,t))\rho(y)dy=-\rho_{c}\rho(x,t)\Lambda^{2\alpha}v(x,t)-\rho(x,t)[\Lambda^{2\alpha},v](\rho-\rho_{c}).  
\end{align*}
In order to obtain the quasilinear form \eqref{eq:quasilinearEulerv}, we provide the coefficients $\widetilde{A}^{0}(V)$, $\widetilde{B}(V)$, $\widetilde{D}(V)$, $\widetilde{H}(V)$, $\widetilde{G}(V)$ and the symbol, $\widetilde{A}(\xi;V):=\sum_{j=1}^{d}\widetilde{A}^{j}(V)\xi_{j}$for any $\xi=(\xi_{1},..,\xi_{d})\in\mathbb{S}^{d-1}$ and $V\in\mathcal{O}_{V}$. We have that, 
\begin{align}
	\widetilde{A}^{0}(V):=\left(\begin{array}{cc}
		1&\\
		&\rho\mathbb{I}_{3}
	\end{array}\right),
\end{align}
where $\mathbb{I}_{d}$ denotes the identity matrix of order $d\times d$, and all the empty spaces refer to zero block matrices of appropriate sizes, and 
\begin{align}
	\widetilde{A}(\xi;V):=\left(\begin{array}{cccc}
		\xi\cdot v&\rho\xi_{1}&\rho\xi_{2}&\rho\xi_{3}\\
		\xi_{1}p^{\prime}(\rho)&\rho\xi\cdot v&0&0\\
		\xi_{2}p^{\prime}(\rho)&0&\rho\xi\cdot v&0\\
		\xi_{3}p^{\prime}(\rho)&0&0&\rho\xi\cdot v
	\end{array}\right).\label{eq:tildecoefficients2}
\end{align}
The remaining coefficients are the following:
\begin{align*}
	&\widetilde{B}(V):=\left(\begin{array}{cc}
		0&\\
		&\rho_{c}\rho\mathbb{I}_{3}
	\end{array}\right),\quad\widetilde{H}(V):=\left(\begin{array}{cc}
		0&\\
		&\rho\mathbb{I}_{3}
	\end{array}\right),\\
	&\widetilde{G}(V):=\left(\begin{array}{cccc}
		0&0&0&0\\
		0&v_{1}&0&0\\
		0&0&v_{2}&0\\
		0&0&0&v_{3}
	\end{array}\right),\quad\widetilde{D}(V):=\left(\begin{array}{cc}
		0&\\
		&\beta\rho\mathbb{I}_{3}
	\end{array}\right)
\end{align*}
and the constant matrix $\mathcal{L}_{4}$ is given as
\begin{align}
	\mathcal{L}_{4}:=\left(\begin{array}{cccc}
		0&0&0&0\\
		1&0&0&0\\
		1&0&0&0\\
		1&0&0&0
	\end{array}\right). \label{eq:L4}
\end{align}
Let $U_{c}=(\rho_{c},0)\in\mathcal{O}_{V}$ be a constant state. Since we consider initial data satisfying that
\begin{align*}
V(0)-U_{c}\in H^{s},
\end{align*}
we introduce the variable $U=V-U_{c}$. Then, $U=(\eta,v)^{\top}$ varies within the open convex set
\begin{align*}
\mathcal{O}=\{(\eta,v)^{\top}\in\mathbb{R}^{4}~|~\eta>-\rho_{c}\quad\mbox{and}\quad v\in\mathbb{R}^{3}\}.
\end{align*}

Clearly, system \eqref{eq:quasilinearEulerv} does not satisfy the assumption of symmetry in conditions (C2), (C3) and (C5). Moreover, it does not satisfy condition (C6). To correct this problem we propose a Friedrichs' symmetrizer for \eqref{eq:quasilinearEulerv} given by the following diagonal matrix function
\begin{align*}
	\mathcal{S}(U+U_{c})=\left(\begin{array}{cc}
		\frac{p^{\prime}(\eta+\rho_{c})}{(\eta+\rho_{c})^{2}}&\\
		&\frac{1}{(\eta+\rho_{c})}\mathbb{I}_{3}
	\end{array}\right).
\end{align*}
Then, by using $U=V-U_{c}$ in \eqref{eq:quasilinearEulerv} and noticing that $\widetilde{D}(U+U_{c})U_{c}=0$ for all $U\in\mathcal{O}$, we arrive at the system
\begin{align}
A^{0}(U)\partial_{t}U+A^{j}(U)\partial_{j}U+D(U)U+B(U)\Lambda^{2\alpha}U+H(U)[\Lambda^{2\alpha},G(U)]\mathcal{L}_{4}U=0\label{eq:Eulerquasymm}
\end{align}
with coefficients given by 
\begin{equation}
\label{eq:Eulercoeff}
\begin{aligned}
&A^{0}(U):=\mathcal{S}(U+U_{c})\widetilde{A}^{0}(U+U_{c})=\left(\begin{array}{cc}
\frac{p^{\prime}(\eta+\rho_{c})}{(\eta+\rho_{c})^{2}}&\\
&\mathbb{I}_{3}
\end{array}\right),\quad B(U):=\mathcal{S}(U+U_{c})\widetilde{B}(U+U_{c}):=\left(\begin{array}{cc}
	0&\\
	&\rho_{c}\mathbb{I}_{3}
\end{array}\right)\\
&D(U):=\mathcal{S}(U+U_{c})\widetilde{D}(U+U_{c}):=\left(\begin{array}{cc}
	0&\\
	&\beta\mathbb{I}_{3}
\end{array}\right),\quad H(U):=\mathcal{S}(U+U_{c})\widetilde{H}(U+U_{c}):=\left(\begin{array}{cc}
	0&\\
	&\mathbb{I}_{3}
\end{array}\right),\\
&G(U):=\left(\begin{array}{cccc}
	0&0&0&0\\
	0&v_{1}&0&0\\
	0&0&v_{2}&0\\
	0&0&0&v_{3}
\end{array}\right)
\end{aligned}
\end{equation}
and 
\begin{align}
A(\xi;U):=\mathcal{S}(U+U_{c})\widetilde{A}(\xi;U+U_{c})=\left(\begin{array}{cccc}
	\frac{p^{\prime}(\eta+\rho_{c})}{(\eta+\rho_{c})^{2}}\xi\cdot v&\xi_{1}\frac{p^{\prime}(\eta+\rho_{c})}{\eta+\rho_{c}}&\xi_{2}\frac{p^{\prime}(\eta+\rho_{c})}{\eta+\rho_{c}}&\xi_{3}\frac{p^{\prime}(\eta+\rho_{c})}{\eta+\rho_{c}}\\
	\xi_{1}\frac{p^{\prime}(\eta+\rho_{c})}{\eta+\rho_{c}}&\xi\cdot v&0&0\\
\xi_{2}\frac{p^{\prime}(\eta+\rho_{c})}{\eta+\rho_{c}}&0&\xi\cdot v&0\\
	\xi_{3}\frac{p^{\prime}(\eta+\rho_{c})}{\eta+\rho_{c}}&0&0&\xi\cdot v
\end{array}\right).\label{eq:Eulersymbol}
\end{align}
In particular, notice that $B_{0}(U):=\rho_{c}\mathbb{I}_{3}$ and thus, for any $v\in H^{\alpha}$ we have
\begin{align}
\left\langle B_{0}(U)\Lambda^{2\alpha}v,v\right\rangle_{L^{2}}\geq\rho_{c}\|\Lambda^{\alpha}v\|_{L^{2}}^{2}.\label{eq:Eulercoercive}
\end{align}
Also, observe that $A^{0}_{2}=\mathbb{I}_{3}$ is a constant matrix. Therefore, with these coefficients, condition \ref{cond1} and assumption $(ii)$ in theorem \ref{localwellposedness} are satisfied. Hence, we have the following result.
\begin{theo}
Let $d=3$. Assume that $\alpha\in(0,1)$ is given and set $s_{\alpha}:=\max\{1,2\alpha\}$. Let $s\in\mathbb{N}$ be such that $s>\frac{d}{2}+s_{\alpha}$. Let condition (\textbf{T}) hold true. Suppose that the initial data satisfies $U_{0}=(\rho_{0}-\rho_{c},v_{0})\in H^{s}$ and $\inf_{x\in\mathbb{R}^{3}}\rho_{0}(x)>0$. Then, there are  positive constants $T_{0}>0$, $M$ and $M_{1}$, and an open convex bounded set $\mathcal{O}_{1}$ compactly contained in $\mathcal{O}$ such that the Cauchy problem for the quasilinear symmetric form of \eqref{eq:Eulervelal}, that is, \eqref{eq:Eulerquasymm} with coefficients given by \eqref{eq:Eulercoeff} and \eqref{eq:Eulersymbol}, has a solution $U=(\rho-\rho_{c},v)\in\mathcal{C}([0,T_{0}];H^{s})$ with $v\in L^{2}([0,T_{0}];H^{s+\alpha})$ in the set $X_{s,\alpha,T_{0}}(\mathcal{O}_{1},M,M_{1})$ with  $\partial_{t}\rho\in\mathcal{C}([0,T_{0}];H^{s-1})$ and $\partial_{t}v\in\mathcal{C}([0,T_{0}];H^{s-2\alpha})$. In particular, the solution satisfies the energy estimate 
\begin{align*}
\sup_{t\in[0,T_{0}]}\|(\rho-\rho_{c})(t)\|_{H^{s}}^{2}+\int_{0}^{T_{0}}\|\Lambda^{\alpha}v(t)\|_{H^{s}}^{2}dt\leq M^{2}.
\end{align*}
Additionally, if $\alpha\in(0,\tfrac{1}{2}]$, we have that $\partial_{t}v\in\mathcal{C}([0,T_{0}];H^{s-1})$. Furthermore, there is a small parameter $\eta_{0}>0$, such that if 
\begin{align}
\|(\rho_{0}-\rho_{c},v_{0})\|_{H^{s}}\leq\eta_{0},
\end{align}
the solution is unique among all solutions in $X_{s,\alpha,T_{0}}(\mathcal{O}_{1},M,M_{1})$ with the same initial data. 
\end{theo}
\begin{proof}
The existence statement follows directly from Theorem \ref{localexistence}. Thus, it remains to verify the uniqueness criterion given by condition \eqref{eq:constantuniqueness} in statement $(ii)$ of Theorem \ref{localwellposedness}.

Observe that, since $A^{0}=\mathbb{I}_{4}$ we may choose $a_{0}(\mathcal{O}_{1})=a_{1}(\mathcal{O}_{1})=1$ and, by \eqref{eq:Eulercoercive}, $\mu_{0}=\rho_{c}$. Consequently, from \eqref{eq:importantconstant} and \eqref{eq:Mconstant}, we obtain
\begin{align*}
C_{0}(\mathcal{O}_{1})=\frac{1}{\min\left\lbrace1,\frac{\rho_{c}}{2}\right\rbrace}\quad\mbox{and}\quad M:=\sqrt{\frac{2}{\min\left\lbrace1,\frac{\rho_{c}}{2}\right\rbrace}}\|(\rho_{0}-\rho_{c},v_{0})\|_{H^{s}}. 
\end{align*}
Moreover, $|H|=|\mathcal{L}|=\sqrt{3}$. Therefore, in this case, the parameter $\omega_{0}$ in \eqref{eq:constantuniqueness} is given by
\begin{align*}
\omega_{0}=18\left(\frac{1}{\min\left\lbrace1,\frac{\rho_{c}}{2}\right\rbrace}\right)^{3/2}C_{G}^{2}(\mathcal{O}_{1},M)\kappa^{2}(\alpha,s,d)\|(\rho_{0}-\rho_{c},v_{0})\|_{H^{s}}^{2}.
\end{align*}
Next, observe that $\|DG(v)\|_{L^{\infty}}\leq 1$ and $D^{k}G(v)=0$ for all $k\geq 2$. Then, $\|DG\|_{\mathcal{C}^{[s-1-\alpha]}}\leq 1$ and, by the chain rule estimates and Theorem \ref{invariantsetsTheorem}, we obtain
\begin{align*}
\|DG(v)\|_{\dot{H}^{s-1-\alpha}}&\leq C\|DG\|_{\mathcal{C}^{[s-1-\alpha]}}(1+\|v\|_{L^{\infty}}^{s-1-\alpha})\|v\|_{\dot{H}^{s-1-\alpha}}\\
&\leq C(1+(\kappa_{s,d}M)^{s-1-\alpha})M.
\end{align*}
Therefore, as a consequence of \eqref{eq:lipschitzrule}, we can take 
\begin{align*}
C_{G}(\mathcal{O}_{1},M):=\max\{\|DG(v)\|_{L^{\infty}},\|DG(v)\|_{\dot{H}^{s-1-\alpha}}\}.
\end{align*}
We now choose 
\begin{align*}
\eta_{0}:=\left(\frac{2}{\min\left\lbrace1,\frac{\rho_{c}}{2}\right\rbrace}\right)^{-1/2}
\min\left\lbrace\frac{1}{2C},\frac{1}{2\kappa_{\sigma,d}},\frac{1}{18\kappa^{2}(\alpha,s,d)}\left(\frac{1}{\min\left\lbrace1,\frac{\rho_{c}}{2}\right\rbrace}\right)^{-1}\right\rbrace
\end{align*}
so that $C_{G}(\mathcal{O}_{1},M)\leq 1$.  Substituting this bound into the expression for $\omega_{0}$, we obtain $0<\omega_{0}<1$. Therefore, condition \eqref{eq:constantuniqueness} is satisfied, and the uniqueness statement follows from theorem \ref{localwellposedness}. 
\end{proof}
\section{Appendix A}
\begin{lema}
\label{continuityHsigma}
Lets $s>\frac{d}{2}+1$ and $\alpha\in(0,1)$. Assume that condition \ref{cond2} is satisfied. Let $g$ be a smooth function of $\widetilde{U}=(\widetilde{u},\widetilde{v})\in\mathcal{O}$.
\begin{itemize}
\item [(a)] For every $0<\sigma\leq s$, it holds that $g(\widetilde{U})\in\mathcal{C}([0,T];\dot{H}^{\sigma})$.
\item [(b)]  Assume that $s_{0},\sigma>0$ are such that $s_{0}-\sigma>\frac{d}{2}$ and $h\in\dot{H}^{\sigma}\cap\dot{H}^{s_{0}}$. Then, there is a positive constant $C$ depending only on $s_{0}$, $\sigma$ and $d$ such that 
\begin{align}
\|\Lambda^{\sigma}h\|_{L^{\infty}}\leq C\left(\|h\|_{\dot{H}^{\sigma}}+\|h\|_{\dot{H}^{s_{0}}}\right)\label{eq:uniquenessrecovery}
\end{align}
\item [(c)] If $\sigma\in(0,1]$ we have that, $\Lambda^{\sigma}g(\widetilde{U})\in\mathcal{C}([0,T];L^{\infty})$ and there are two positive constants $C_{1}=C_{1}(d)$ and $C_{2}=C_{2}(d,s,\sigma)$ such that
\begin{align}
\label{eq:infinitycontrol}
\sup_{t\in[0,T]}\|\Lambda^{\sigma}g(\widetilde{U}(t))\|_{L^{\infty}}\leq C_{1}\|g(\widetilde{U})\|_{\mathcal{C}([0,T];\dot{H}^{\sigma})}+C_{2}\|g(\widetilde{U})\|_{\mathcal{C}([0,T];\dot{H}^{s})}.
\end{align}
\item [(d)] Let $s_{0}\geq s$ and $\sigma\in[0,s_{0}]$. Let $\widetilde{U}^{1}=(\widetilde{u},\widetilde{v}^{1})^{\top}$ and $\widetilde{U}^{2}=(\widetilde{u}^{2},\widetilde{v}^{2})^{\top}$ satisfy condition \ref{cond2}. Then, there is a positive constant $C_{g}(\mathcal{O}_{1},M,\sigma)$ such that
\begin{align}
\label{eq:Lipschitzcontinuityfinal}
\|g(\widetilde{U}^{1})-g(\widetilde{U}^{2})\|_{\dot{H}^{\sigma}}\leq C_{g}(\mathcal{O}_{1},M,\sigma)\left(\|\widetilde{U}^{1}-\widetilde{U}^{2}\|_{H^{\sigma}}+\|\widetilde{U}^{1}-\widetilde{U}^{2}\|_{L^{\infty}}\right).
\end{align}

\end{itemize}
\end{lema}
\begin{proof}
In order to prove the continuity of the mapping $g(\widetilde{U}):[0,T]\rightarrow \dot{H}^{\sigma}$ we consider the following formula, 
\begin{align*}
g(\widetilde{U}(t))-g(\widetilde{U}(t_{0}))=\int_{0}^{1}Dg(\mathcal{U}(\tau;t,t_{0}))d\tau\left[\widetilde{U}(t)-\widetilde{U}(t_{0})\right],
\end{align*}
where $\mathcal{U}(\tau;t,t_{0}):=\widetilde{U}(t_{0})+\tau(\widetilde{U}(t)-\widetilde{U}(t_{0}))$. Then, by \eqref{eq:fractionalproduct2}, 
\begin{equation}
\label{eq:contcomp1}
\begin{aligned}
\|g(\widetilde{U}(t))-g(\widetilde{U}(t_{0}))\|_{\dot{H}^{\sigma}}&\leq\left\lVert\int_{0}^{1}Dg(\mathcal{U}(\tau;t,t_{0}))d\tau\right\rVert_{L^{\infty}}\|\widetilde{U}(t)-\widetilde{U}(t_{0})\|_{\dot{H}^{\sigma}}\\
&+\left\lVert\int_{0}^{1}Dg(\mathcal{U}(\tau;t,t_{0}))d\tau\right\rVert_{\dot{H}^{\sigma}}\|\widetilde{U}(t)-\widetilde{U}(t_{0})\|_{L^{\infty}}.
\end{aligned}
\end{equation}
Observe that $Dg(\mathcal{U}(\tau;t,t_{0}))\in L^{\infty}([0,1]\times\mathbb{R}^{d}\times[0,T])$ and thus, the first term in the right hand side of \eqref{eq:contcomp1} goes to zero if $t\rightarrow t_{0}$ in $[0,T]$. Since $\widetilde{U}\in\mathcal{C}([0,T];H^{s})$, it is enough to obtain a uniform bound on the term 
\begin{align*}
\left\lVert\int_{0}^{1}Dg(\mathcal{U}(\tau;t,t_{0}))d\tau\right\rVert_{\dot{H}^{\sigma}},
\end{align*}
in order to conclude. First, observe that, for any $t,t_{0}$ and $\tau\in[0,1]$,  $|\mathcal{U}(\tau;t,t_{0})|\leq2\|\widetilde{U}\|_{\mathcal{C}([0,T];H^{s})}$ so that, the range of the function $[0,1]\times[0,T]\mapsto\mathcal{U}(\tau;t,t_{0})$ lies within a compact set of $\mathbb{R}^{N}$. Hence, for any $\delta>0$, there is a positive constant $M(\delta)>0$ such that $\|Dg\|_{\mathcal{C}^{[\delta]}}\leq M(\delta)$ when restricted to $R(\mathcal{U}(\tau;t,t_{0}))$. Therefore, by the fractional chain rule estimates, for any $\delta>0$ there are positive constants $C(\delta)$ and $M(\delta)$ such that,
\begin{align}
\label{eq:fracchainDG}
\mbox{for all}\quad t,t_{0}\in[0,T]~\mbox{and}~\tau\in[0,1],\quad\|Dg(\mathcal{U}(\tau;t,t_{0}))\|_{\dot{H}^{\delta}}\leq C(\delta)M(\delta)\|\widetilde{U}\|_{\mathcal{C}([0,T];H^{s})}.
\end{align}
Now, it is easy to show that
\begin{align*}
\mathcal{F}\left(\int_{0}^{1}Dg(\mathcal{U}(\tau;t,t_{0}))d\tau\right)=\int_{0}^{1}\mathcal{F}(Dg(\mathcal{U}(\tau;t,t_{0})))d\tau.
\end{align*}
Therefore, by Jensen's inequality and \eqref{eq:fracchainDG} it follows that, for any $\sigma\in(0,s]$ there are two positive constants $C(\sigma)$ and $M(\sigma)$ such that for any $t,t_{0}\in[0,T]$, 
\begin{equation}
\label{eq:integraluniform}
\begin{aligned}
\left\lVert\int_{0}^{1}Dg(\mathcal{U}(\tau;t,t_{0}))d\tau\right\rVert_{\dot{H}^{\sigma}}^{2}&\leq\int_{\mathbb{R}^{d}}\int_{0}^{1}|\xi|^{2\sigma}\left\lvert\mathcal{F}\left(Dg(\mathcal{U}(\tau;t,t_{0}))\right)(\xi)\right\rvert^{2}d\tau d\xi\leq C(\sigma)^{2}M(\sigma)^{2}\|\widetilde{U}\|_{\mathcal{C}([0,T];H^{s})}^{2}.
\end{aligned}
\end{equation}
In consequence, by taking the limit when $t\rightarrow t_{0}$ in the right hand side of \eqref{eq:contcomp1}, the result follows. 

Next, we prove $(b)$. By the continuity of $\mathcal{F}:L^{1}\rightarrow L^{\infty}$ and H\"older's inequality, for any $t\in[0,T]$, we have
\begin{align*}
\|\Lambda^{\sigma}h\|_{L^{\infty}}&\leq\|\mathcal{F}(\Lambda^{\sigma}h)\|_{L^{1}}\leq\int_{|\xi|\leq 1}|\xi|^{\sigma}|\widehat{h}(\xi)|d\xi+\int_{|\xi|\geq 1}|\xi|^{\sigma}|\widehat{h}(\xi)||\xi|^{s_{0}-\sigma}|\xi|^{-(s_{0}-\sigma)}d\xi\\
&\leq\lambda_{d}(B_{1}(0))^{1/2}\|h\|_{\dot{H}^{\sigma}}+\left(\int_{|\xi|\geq 1}|\xi|^{-2(s_{0}-\sigma)}d\xi\right)^{1/2}\|h\|_{\dot{H}^{s_{0}}}.
\end{align*}
Since $s_{0}-\sigma>\frac{d}{2}$, the last integral is finite and \eqref{eq:uniquenessrecovery} follows.

The proof of statement $(c)$ is a direct application of \eqref{eq:uniquenessrecovery} to the function $h=g(\widetilde{U})$. Indeed, by applying the fractional chain rule estimates together with \eqref{eq:uniquenessrecovery}, we find a constant $C=(\sigma,s,d)>0$ satisfying that,
\begin{align*}
\sup_{t\in[0,T]}\|\Lambda^{\sigma}g(\widetilde{U}(t))\|_{L^{\infty}}\leq C(\sigma,s,d)\|\widetilde{U}\|_{\mathcal{C}([0,T];H^{s})}.
\end{align*}
Moreover, for any $t,t_{0}\in[0,T]$, \eqref{eq:uniquenessrecovery} yields
\begin{align*}
\|\Lambda^{\sigma}g(\widetilde{U}(t))-\Lambda^{\sigma}g(\widetilde{U}(t_{0}))\|_{L^{\infty}}&\leq\|\mathcal{F}(g(\widetilde{U}(t)))-\mathcal{F}(g(\widetilde{U}(t_{0})))\|_{L^{1}}\\
&\leq C(\sigma,s,d)\left(\|g(\widetilde{U}(t))-g(\widetilde{U}(t_{0}))\|_{\dot{H}^{\sigma}}+\|g(\widetilde{U}(t))-g(\widetilde{U}(t_{0}))\|_{\dot{H}^{s}}\right).
\end{align*}
By the continuity proven in $(a)$, we reach the conclusion. 

Finally, to prove $(d)$ let $\widetilde{U}^{1},\widetilde{U}^{2}\in\mathcal{O}_{1}$ be arbitrary states in $\mathcal{O}_{1}$. By using the same argument that lead us to \eqref{eq:contcomp1} we obtain the estimate
\begin{align}
\|g(\widetilde{U}^{1})-g(\widetilde{U}^{2})\|_{\dot{H}^{\sigma}}&\leq\left\lVert\int_{0}^{1}Dg(\mathcal{V}(\tau))d\tau\right\rVert_{L^{\infty}}\|\widetilde{U}^{1}-\widetilde{U}^{2}\|_{\dot{H}^{\sigma}}+\left\lVert\int_{0}^{1}Dg(\mathcal{V}(\tau))d\tau\right\rVert_{\dot{H}^{\sigma}}\|\widetilde{U}^{1}-\widetilde{U}^{2}\|_{L^{\infty}},\label{eq:lipschitzrule}
\end{align}
where $\mathcal{V}(\tau)=\widetilde{U}_{2}+\tau(\widetilde{U}^{1}-\widetilde{U}^{2})$. The result is now a consequence of the fact that $Dg(\mathcal{V}(\tau))\in L^{\infty}(\mathbb{R}^{d}\times[0,T]\times[0,1])$ and \eqref{eq:integraluniform}. 
\end{proof}	
\section{Apendix B: Auxiliary commutator estimates}
\begin{lema}
\label{epsilonlimitest}
Let $\alpha\in(0,1)$ be given and let $s>\frac{d}{2}+1$ be an integer.  Assume that $Q$ is $p\times p$ matrix-valued function and that $w$ is a $\mathbb{R}^{p}$-valued vector field. Then, there is a positive constant $C$ depending only on $d$, $s$ and $\alpha$ such that the following statements are satisfied.
\begin{itemize}
\item [(i)] If  $w\in H^{s-\alpha}$ and $Q\in\widehat{H}^{s}$, then
\begin{align}
\|\Lambda^{s-\alpha}[\mathbb{J}_{\epsilon},Q]w\|_{L^{2}}\leq C\|Q\|_{\widehat{H}^{s}}\|w\|_{H^{s-1}}\label{eq:firstfeliestimate}
\end{align}
and $\|\Lambda^{s-\alpha}[\mathbb{J}_{\epsilon},Q]w\|_{L^{2}}\rightarrow 0$ as $\epsilon\rightarrow 0$. 
\item [(ii)] Set $s_{\alpha}:=\max\{1,2\alpha\}$. If $s>\frac{d}{2}+s_{\alpha}$, $Q\in\widehat{H}^{s}$ and $w\in H^{s-s_{\alpha}}$, then
\begin{align}
\|\Lambda^{s-\alpha}[\mathbb{J}_{\epsilon},Q]w\|_{L^{2}}\leq C\|Q\|_{\widehat{H}^{s}}\|w\|_{H^{s-s_{\alpha}}}\label{eq:secondfeliestimate}
\end{align}
and $\|\Lambda^{s-\alpha}[\mathbb{J}_{\epsilon},Q]w\|_{L^{2}}\rightarrow 0$ as $\epsilon\rightarrow 0$. 
\item [(iii)] If $Q\in\dot{H}^{\alpha}\cap\widehat{H}^{s}\cap\dot{H}^{s+\alpha}$ and $w\in H^{s}$ we have that
\begin{align}
\|\Lambda^{s+\alpha}[\mathbb{J}_{\epsilon},Q]w\|_{L^{2}}\leq C(\|Q\|_{\dot{H}^{\alpha}}+\|Q\|_{\dot{H}^{s+\alpha}})\|w\|_{H^{s}}\label{eq:repeteadestimate}
\end{align}
and $\|\Lambda^{s+\alpha}[\mathbb{J}_{\epsilon},Q]w\|_{L^{2}}\rightarrow 0$ as $\epsilon\rightarrow 0$. 
\end{itemize}
\end{lema}
\begin{proof}
$(i)$ Let us first assume that $w\in H^{s-\alpha}$. By \eqref{eq:maincommutatoridea},
\begin{align*}
\Lambda^{s-\alpha}[\mathbb{J}_{\epsilon},Q]w=\mathbb{J}_{\epsilon}[\Lambda^{s-\alpha},Q]w+[\mathbb{J}_{\epsilon},Q]\Lambda^{s-\alpha}w-[\Lambda^{s-\alpha},Q]\mathbb{J}_{\epsilon}w.
\end{align*}
By the fractional Leibniz rule and the duality argument that led us to \eqref{eq:dualityargument}, it follows that 
\begin{align*}
\|\Lambda^{s-\alpha}[\mathbb{J}_{\epsilon},Q]w\|_{L^{2}}\leq C\left(\|\nabla Q\|_{L^{\infty}}\|\Lambda^{s-\alpha-1}w\|_{L^{2}}+\|\Lambda^{s-\alpha}Q\|_{L^{2}}\|w\|_{L^{\infty}}+\|\nabla Q\|_{L^{\infty}}\|\Lambda^{s-\alpha}w\|_{H^{-1}}\right). 
\end{align*}
Then, estimate \eqref{eq:firstfeliestimate} is a consequence of applying the Sobolev embedding $H^{s-1}\hookrightarrow L^{\infty}$ into the last inequality.
	
$(ii)$ Now assume that $s>\frac{d}{2}+s_{\alpha}$. It is enough to prove the statement by assuming that $w\in\mathcal{S}$. Apply the decomposition
\begin{align*}
\Lambda^{s-\alpha}[\mathbb{J}_{\epsilon},Q]w&=\mathbb{J}_{\epsilon}[\Lambda^{s-\alpha},Q]w+[\mathbb{J}_{\epsilon},Q]\Lambda^{s-\alpha}w-[\Lambda^{s-\alpha},Q]\mathbb{J}_{\epsilon}w.
\end{align*}
Then, as in step $(i)$, it holds that
\begin{align*}
\|\Lambda^{s-\alpha}[\mathbb{J}_{\epsilon},Q]w\|_{L^{2}}&\leq C\left(\| Q\|_{\widehat{H}^{s}}\|w\|_{H^{s-\alpha-1}}+\|\Lambda^{s-\alpha}Q\|_{L^{2}}\|w\|_{L^{\infty}}\right).
\end{align*}
Therefore, by the embeddings $H^{s-s_{\alpha}}\hookrightarrow L^{\infty}$ and $H^{s-s_{\alpha}}\hookrightarrow H^{s-\alpha-1}$ we obtain that
\begin{align*}
\|\Lambda^{s-\alpha}[\mathbb{J}_{\epsilon},Q]w\|_{L^{2}}\leq C\|Q\|_{\widehat{H}^{s}}\|w\|_{H^{s-s_{\alpha}}}\quad\mbox{for all}\quad m\in\mathbb{N}.
\end{align*}
	
$(iii)$ Since $\|\Lambda^{s+\alpha}[\mathbb{J}_{\epsilon},Q]w\|_{L^{2}}\leq\|\Lambda^{\alpha}[\mathbb{J}_{\epsilon},Q]w\|_{H^{s}}$, \eqref{eq:repeteadestimate} is a consequence of theorem \ref{secondFcommest}.  
	
Now, we prove the convergence assertion in case $(iii)$.  Let $\{\varphi_{m}\}\subset\mathcal{S}$ such that, $\varphi_{m}\rightarrow w$ in $H^{s}$. In order to prove the convergence to zero in $L^{2}$, we fix $m\in\mathbb{N}$ and use the following decomposition,
\begin{align*}
\Lambda^{s+\alpha}[\mathbb{J}_{\epsilon},Q]\varphi_{m}&=\mathbb{J}_{\epsilon}[\Lambda^{s+\alpha},Q]\varphi_{m}+[\mathbb{J}_{\epsilon},Q]\Lambda^{s+\alpha}\varphi_{m}-[\Lambda^{s+\alpha},Q]\mathbb{J}_{\epsilon}\varphi_{m}\\
&=:I_{1,\epsilon}^{m}+I_{2,\epsilon}^{m}-I_{3,\epsilon}^{m}. 
\end{align*}
By \cite[Lemma 1.5.2]{milani}, $I_{2,\epsilon}^{m}\rightarrow 0$ in $L^{2}$ as $\epsilon\rightarrow 0$. Meanwhile, since $\Lambda^{s+\alpha}(Q\varphi_{m})$, $Q\Lambda^{s+\alpha}\varphi_{m}$ and $\Lambda^{s+\alpha}\varphi_{m}$ belong to $L^{2}$, it follows that
\begin{align*}
\mathbb{J}_{\epsilon}\Lambda^{s+\alpha}(Q\varphi_{m})\rightarrow \Lambda^{s+\alpha}(Q\varphi_{m}),\quad \mathbb{J}_{\epsilon}(Q\Lambda^{s+\alpha}\varphi_{m})\rightarrow Q\Lambda^{s+\alpha}\varphi_{m}\quad\mbox{and}\quad Q\Lambda^{s+\alpha}\mathbb{J}_{\epsilon}\varphi_{m}\rightarrow Q\Lambda^{s+\alpha}\varphi_{m}
\end{align*}
in $L^{2}$ as $\epsilon\to 0$. Furthermore, by the fractional Leibniz rule combined with the embedding $H^{s-1}\hookrightarrow L^{\infty}$ we have the estimate
\begin{align*}
\|\Lambda^{s+\alpha}(Q\mathbb{J}_{\epsilon}\varphi_{m})-\Lambda^{s+\alpha}(Q\varphi_{m})\|_{L^{2}}\leq C\left(\|Q\|_{L^{\infty}}\|\mathbb{J}_{\epsilon}\varphi_{m}-\varphi_{m}\|_{\dot{H}^{s+\alpha}}+\|Q\|_{\dot{H}^{s+\alpha}}\|\mathbb{J}_{\epsilon}\varphi_{m}-\varphi_{m}\|_{H^{s-1}}\right),
\end{align*}
which goes to zero as $\epsilon\rightarrow 0$. Therefore, for each $m\in\mathbb{N}$ it holds that
\begin{align*}
\lim_{\epsilon\rightarrow 0}\left(I_{1,\epsilon}^{m}-I_{3,\epsilon}^{m}\right)&=\lim_{\epsilon\rightarrow 0}\left(\mathbb{J}_{\epsilon}\Lambda^{s+\alpha}(Q\varphi_{m})-\Lambda^{s+\alpha}(Q\mathbb{J}_{\epsilon}\varphi_{m})-\mathbb{J}_{\epsilon}(Q\Lambda^{s+\alpha}\varphi_{m})+Q\Lambda^{s+\alpha}\mathbb{J}_{\epsilon}\varphi_{m}\right)\\
&=\Lambda^{s+\alpha}(Q\varphi_{m})-\Lambda^{s+\alpha}(Q\varphi_{m})-Q\Lambda^{s+\alpha}\varphi_{m}+Q\Lambda^{s+\alpha}\varphi_{m}=0\quad\mbox{in}\quad L^{2}. 
\end{align*}
Finally, observe that, as a consequence of \eqref{eq:repeteadestimate}, 
\begin{align*}
\|\Lambda^{s+\alpha}[\mathbb{J}_{\epsilon},Q]w-\Lambda^{s+\alpha}[\mathbb{J}_{\epsilon},Q]\varphi_{m}\|_{L^{2}}\leq C(\|Q\|_{\dot{H}^{\alpha}}+\|Q\|_{\dot{H}^{s+\alpha}})\|w-\varphi_{m}\|_{H^{s}}
\end{align*}
and the result follows. The convergence statements in $(i)$ and $(ii)$ are obtained by the same argument, replacing $\Lambda^{s+\alpha}$ by $\Lambda^{s-\alpha}$ and by using estimates \eqref{eq:firstfeliestimate} and \eqref{eq:secondfeliestimate}, respectively. 
\end{proof}

\begin{lema}
Let $s>\frac{d}{2}+1$ be an integer and $\alpha\in(0,1)$ be given. Assume that $f\in\dot{H}^{\alpha}\cap\widehat{H}^{s}\cap\dot{H}^{s+\alpha}$ and $g\in H^{m}$ for some $m\in\mathbb{N}$ with $1\leq m\leq s$. Then, there is a positive constant $C$ depending only on $\alpha$, $m$, $d$ and $s$, such that
\begin{align}
\label{eq:fraccommnon}
\|[\Lambda^{2\alpha},f]g\|_{H^{m-1}}\leq C(\|f\|_{\dot{H}^{\alpha}}+\|f\|_{\dot{H}^{s+\alpha}})\|g\|_{H^{m}}.
\end{align}
\end{lema}
\begin{proof}
First, notice that if $m=1$ the result follows directly from the $L^{2}$ estimates for the commutator of $\Lambda^{2\alpha}$. Indeed, if $0<\alpha\leq\frac{1}{2}$ we use \eqref{eq:dongliC2} and if $\frac{1}{2}<\alpha<1$ we apply \eqref{eq:fractionalcommutator2} with $p_{1}=\infty$, $p_{2}=2$, $p_{3}=p_{4}=4$ together with the embedding $H^{1}\hookrightarrow L^{4}$. 
	
Then, we assume $2\leq m\leq s$. Given that
	\begin{align*}
		\|[\Lambda^{2\alpha},f]g\|_{H^{m-1}}\leq C(\|[\Lambda^{2\alpha},f]g\|_{L^{2}}+\|\Lambda^{m-1}[\Lambda^{2\alpha},f]g\|_{L^{2}}),
	\end{align*}
	it is enough to prove the required estimate for the term $\Lambda^{m-1}[\Lambda^{2\alpha},f]g$. By using the identity
	\begin{align*}
		\Lambda^{m-1}[\Lambda^{2\alpha},f]g=[\Lambda^{m-1+2\alpha},f]g-[\Lambda^{m-1},f]\Lambda^{2\alpha}g,
	\end{align*}
	the result reduces to estimate the $L^{2}$ norm of each commutator separately. If $2\leq m\leq s-1$ we proceed as in the case above. If, on the other hand, $m=s$ we apply \eqref{eq:fractionalcommutator2} with $p_{1}=\infty$, $p_{2}=2$, $p_{3}=2$ and $p_{4}=\infty$ together with the embedding $H^{s}\hookrightarrow L^{\infty}$. 
\end{proof}
\section{Appendix C}
\label{AppendixC}
We will repeatedly use the following observation: for any $\sigma\in\mathbb{R}$, $\delta\in(0,1]$ and $f\in H^{\sigma+\delta}$, we have that, 
\begin{equation}
\label{eq:Cauchyrate}
\begin{aligned}
\|(\mathbb{J}_{\epsilon_{n+p}}-\mathbb{J}_{\epsilon_{n}})f\|_{H^{\sigma}}^{2}&=\int_{\mathbb{R}^{d}}(1+|\xi|^{2})^{\sigma}|\widehat{\eta}(\epsilon_{n+p}\xi)-\widehat{\eta}(\epsilon_{n}\xi)|^{2-2\delta}|\widehat{\eta}(\epsilon_{n+p}\xi)-\widehat{\eta}(\epsilon_{n}\xi)|^{2\delta}|\widehat{f}(\xi)|^{2}d\xi\\
&\leq C_{0}(\delta)^{2}|\epsilon_{n+p}-\epsilon_{n}|^{2\delta}\|f\|_{H^{\sigma+\delta}}^{2},
	\end{aligned}
\end{equation}
with  $C_{0}(\delta):=(2\|\widehat{\eta}\|_{L^{\infty}})^{1-\delta}\|\nabla \widehat{\eta}\|_{L^{\infty}}^{\delta}$. Similarly, since $\widehat{\eta}(0)=1$, the same argument yields
\begin{align}
	\|(\mathbb{J}_{\epsilon_{n}}-\mathbb{I})f\|_{H^{\sigma}}^{2}\leq C_{0}(\delta)\epsilon_{n}^{2\delta}\|f\|_{H^{\sigma+\delta}}^{2}.\label{eq:Cauchyrate2} 
\end{align}

\begin{proof}[Proof of Theorem~\ref{superapprox}]
Assume one of the hypotheses (E1)-(E2) holds. 
Let $U_{n}=(u_{n},v_{n})$ be the unique solution of \eqref{eq:LSn} constructed above. First, we show the existence of $U\in\cap_{\sigma\in[0,s)}X_{\sigma,\alpha,T}$ such that, 
\begin{equation}
\label{eq:Cauchylimit}
\mathbb{J}_{\epsilon_{n}}U_{n}\rightarrow U\quad\mbox{in}\quad X_{\sigma,\alpha,T}\quad\mbox{for all}\quad\sigma\in[0,s). 
\end{equation}
For any $p\in\mathbb{N}$, we define $V_{n,p}:=U_{n+p}-U_{n}$ and, in particular, $z_{n,p}:=v_{n+p}-v_{n}$. Subtracting the equations satisfied by $U_{n+p}$ and $U_{n}$, we get
\begin{equation}
\label{eq:Cauchydifference1}
\begin{aligned}
A^{0}\partial_{t}V_{n,p}+\mathbb{J}_{\epsilon_{n+p}}\left(A^{j}\partial_{j}\mathbb{J}_{\epsilon_{n+p}}V_{n,p}\right.&+\left.D\mathbb{J}_{\epsilon_{n+p}}V_{n,p}+B\Lambda^{2\alpha}\mathbb{J}_{\epsilon_{n+p}}V_{n,p}+H\mathbb{J}_{\epsilon_{n+p}}[\Lambda^{2\alpha},G]\mathbb{J}_{\epsilon_{n+p}}V_{n,p}\right)=F_{n,p}^{\alpha},\\
&V_{n,p}\rvert_{t=0}=(\mathbb{J}_{\epsilon_{n+p}}-\mathbb{J}_{\epsilon_{n}})U_{0},
\end{aligned}
\end{equation}
where the source term is given by
	\begin{equation}
		\label{eq:CauchydifferenceS1}
		\begin{aligned}
			F_{n,p}^{\alpha}&:=(\mathbb{J}_{\epsilon_{n+p}}-\mathbb{J}_{\epsilon_{n}})(F)-(\mathbb{J}_{\epsilon_{n+p}}-\mathbb{J}_{\epsilon_{n}})(A^{j}\partial_{j}\mathbb{J}_{\epsilon_{n}}U_{n})-(\mathbb{J}_{\epsilon_{n+p}}-\mathbb{J}_{\epsilon_{n}})(D\mathbb{J}_{\epsilon_{n}}U_{n})\\
			&-(\mathbb{J}_{\epsilon_{n+p}}-\mathbb{J}_{\epsilon_{n}})(B\Lambda^{2\alpha}\mathbb{J}_{\epsilon_{n}}U_{n})-\mathbb{J}_{\epsilon_{n+p}}\left(H(\mathbb{J}_{\epsilon_{n+p}}-\mathbb{J}_{\epsilon_{n}})[\Lambda^{2\alpha},G]\mathcal{L}\mathbb{J}_{\epsilon_{n}}U_{n}\right)\\
			&-(\mathbb{J}_{\epsilon_{n+p}}-\mathbb{J}_{\epsilon_{n}})\left(H\mathbb{J}_{\epsilon_{n}}[\Lambda^{2\alpha},G]\mathcal{L}\mathbb{J}_{\epsilon_{n}}U_{n}\right)-\mathbb{J}_{\epsilon_{n+p}}\left(A^{j}\partial_{j}(\mathbb{J}_{\epsilon_{n+p}}-\mathbb{J}_{\epsilon_{n}})U_{n}\right)\\
			&-\mathbb{J}_{\epsilon_{n+p}}\left(D(\mathbb{J}_{\epsilon_{n+p}}-\mathbb{J}_{\epsilon_{n}})U_{n}\right)-\mathbb{J}_{\epsilon_{n+p}}\left(B\Lambda^{2\alpha}(\mathbb{J}_{\epsilon_{n+p}}-\mathbb{J}_{\epsilon_{n}})U_{n}\right)\\
			&-\mathbb{J}_{\epsilon_{n+p}}\left(H\mathbb{J}_{\epsilon_{n+p}}[\Lambda^{2\alpha},G]\mathcal{L}(\mathbb{J}_{\epsilon_{n+p}}-\mathbb{J}_{\epsilon_{n}})U_{n}\right).
		\end{aligned}
	\end{equation}
	Moreover, we decompose  $F_{n,p}^{\alpha}=(f_{n,p}^{1,\alpha},f_{n,p}^{2,\alpha})^{\top}$ where $f_{n,p}^{1,\alpha}$ and $f_{n,p}^{2,\alpha}$ take values in $\mathbb{R}^{k}$ and $\mathbb{R}^{m}$, respectively. 
	
	According to Theorem \ref{energyestimatesF}, the sequence $\{U_{n}\}$ satisfies the estimate, 
	\begin{align}
		\|U_{n}(t)\|_{H^{s}}^{2}+\int_{0}^{t}\|\Lambda^{\alpha}\mathbb{J}_{\epsilon_{n}}v_{n}(\tau)\|_{H^{s}}^{2}d\tau\leq K(t)\Phi_{\alpha}^{s}(U_{0},F,t)\quad\mbox{for all}\quad t\in[0,T].\label{eq:boundHs}
	\end{align}
	Furthermore, since $V_{n,p}$ satisfies \eqref{eq:Cauchydifference1}, we may apply $m=s-2$ in \eqref{eq:fractionalenergyT} to obtain, 
	\begin{align}
		\mbox{for all}\quad t\in[0,T],\quad\|\mathbb{J}_{\epsilon_{n+p}}V_{n,p}(t)\|_{H^{s-2}}^{2}+\int_{0}^{t}\|\Lambda^{\alpha}\mathbb{J}_{\epsilon_{n+p}}z_{n,p}(\tau)\|_{H^{s-2}}^{2}d\tau\leq K(t)\Phi_{\alpha}^{s-2}(V_{n,p}(0),F_{n,p}^{\alpha},t). \label{eq:Cauchyestimate}
	\end{align}
	We first prove that, as a consequence of \eqref{eq:boundHs} and \eqref{eq:Cauchyestimate}, $\mathbb{J}_{\epsilon_{n+p}}V_{n,p}$ converges to zero as $n\rightarrow\infty$ with respect to the norm of the space $X_{s-2,\alpha,T}^{\infty}$. Employing \eqref{eq:Cauchyrate} with $\sigma=s-2$ and $\delta=1$, together with the Sobolev product estimates, yields
	\begin{equation}
\label{eq:firstsource}
\begin{aligned}
\|f_{n,p}^{1,\alpha}\|_{H^{s-2}}&\leq C_{0} |\epsilon_{n+p}-\epsilon_{n}|\left(\|f_{1}\|_{H^{s-1}}+\|A^{j}\partial_{j}\mathbb{J}_{\epsilon_{n}}U_{n}\|_{H^{s-1}}+\|D\mathbb{J}_{\epsilon_{n}}U_{n}\|_{H^{s-1}}\right)\\
&+\|A^{j}(\mathbb{J}_{\epsilon_{n+p}}-\mathbb{J}_{\epsilon_{n}})\partial_{j}U_{n}\|_{H^{s-2}}+\|D(\mathbb{J}_{\epsilon_{n+p}}-\mathbb{J}_{\epsilon_{n}})U_{n}\|_{H^{s-2}}\\
&\leq C_{0}C(\mathcal{O}_{1},M)|\epsilon_{n+p}-\epsilon_{n}|\left(\|f_{1}\|_{H^{s-1}}+\|U_{n}\|_{H^{s}}\right).
\end{aligned}
\end{equation}
Let, $\mathcal{L}U_{n}=(y_{n},w_{n})^{\top}\in\mathbb{R}^{k}\times\mathbb{R}^{m}$. To estimate $f_{n,p}^{2,\alpha}$ we apply \eqref{eq:Cauchyrate} with $\sigma=s-2-\alpha$ and $\delta=\alpha$, combined with conditions (C5) and (C6), so that
\begin{align}
\|f_{n,p}^{2,\alpha}\|_{H^{s-2-\alpha}}&\leq C_{0}(\alpha)C(\mathcal{O}_{1},M)|\epsilon_{n+p}-\epsilon_{n}|^{\alpha}\left(\|f_{2}\|_{H^{s-2}}+\|U_{n}\|_{H^{s}}\right)+C_{0}(\alpha)|\epsilon_{n+p}-\epsilon_{n}|^{\alpha}\|B_{0}\Lambda^{2\alpha}\mathbb{J}_{\epsilon_{n}}v_{n}\|_{H^{s-2}}\nonumber\\
&+\|H_{0}(\mathbb{J}_{\epsilon_{n+p}}-\mathbb{J}_{\epsilon_{n}})[\Lambda^{2\alpha},G_{0}]\mathbb{J}_{\epsilon_{n}}w_{n}\|_{H^{s-2-\alpha}}+C_{0}(\alpha)|\epsilon_{n+p}-\epsilon_{n}|^{\alpha}\|H_{0}\mathbb{J}_{\epsilon_{n}}[\Lambda^{2\alpha},G_{0}]\mathbb{J}_{\epsilon_{n}}w_{n}\|_{H^{s-2}}\nonumber\\
&+\|B_{0}\Lambda^{2\alpha}(\mathbb{J}_{\epsilon_{n+p}}-\mathbb{J}_{\epsilon_{n}})v_{n}\|_{H^{s-2-\alpha}}+\|H_{0}\mathbb{J}_{\epsilon_{n+p}}[\Lambda^{2\alpha},G_{0}](\mathbb{J}_{\epsilon_{n+p}}-\mathbb{J}_{\epsilon_{n}})w_{n}\|_{H^{s-2-\alpha}}\nonumber\\
&\leq C_{0}(\alpha)C(\mathcal{O}_{1},M)|\epsilon_{n+p}-\epsilon_{n}|^{\alpha}\left(\|f_{2}\|_{H^{s-2}}+\|U_{n}\|_{H^{s}}\right)+C(\mathcal{O}_{1},M)R_{n,p}^{\alpha}\label{eq:secondsourcealmost}
\end{align}
with
	\begin{align*}
		R_{n,p}^{\alpha}&:=\|(\mathbb{J}_{\epsilon_{n+p}}-\mathbb{J}_{\epsilon_{n}})[\Lambda^{2\alpha},G_{0}]\mathbb{J}_{\epsilon_{n}}w_{n}\|_{H^{s-2}}+\|\Lambda^{2\alpha}(\mathbb{J}_{\epsilon_{n+p}}-\mathbb{J}_{\epsilon_{n}})w_{n}\|_{H^{s-2}}\\
		&+\|\mathbb{J}_{\epsilon_{n+p}}[\Lambda^{2\alpha},G_{0}](\mathbb{J}_{\epsilon_{n+p}}-\mathbb{J}_{\epsilon_{n}})w_{n}\|_{H^{s-2}}.
	\end{align*}
	We first consider the case $\alpha\in(0,\tfrac{1}{2}]$. Applying \eqref{eq:Cauchyrate} with $\sigma=s-2$ and $\delta=1$, implies
\begin{equation}
\label{eq:Rnhalf}
\begin{aligned}
R_{n,p}^{\alpha}&\leq C_{0}|\epsilon_{n+p}-\epsilon_{n}|\left(\|[\Lambda^{2\alpha},G_{0}]\mathbb{J}_{\epsilon_{n}}w_{n}\|_{H^{s-1}}+\|\Lambda^{2\alpha}w_{n}\|_{H^{s-1}}\right)+\|\Lambda^{2\alpha}\left(G_{0}(\mathbb{J}_{\epsilon_{n+p}}-\mathbb{J}_{\epsilon_{n}})w_{n}\right)\|_{H^{s-2}}\\
&+\|G_{0}(\mathbb{J}_{\epsilon_{n+p}}-\mathbb{J}_{\epsilon_{n}})\Lambda^{2\alpha}w_{n}\|_{H^{s-2}}\\
&\leq C(\mathcal{O}_{1},M)C_{0}|\epsilon_{n+p}-\epsilon_{n}|\|w_{n}\|_{H^{s}}.
\end{aligned}
\end{equation}
	In contrast, if $\alpha\in(\tfrac{1}{2},1)$, we set $\sigma=s-2$ and $\delta=2-2\alpha$ in \eqref{eq:Cauchyrate} to have
	\begin{align*}
		R_{n,p}^{\alpha}&\leq C_{0}(2-2\alpha)|\epsilon_{n+p}-\epsilon_{n}|^{2-2\alpha}\left(\|[\Lambda^{2\alpha},G_{0}]\mathbb{J}_{\epsilon_{n}}w_{n}\|_{H^{s-2\alpha}}+\|\Lambda^{2\alpha}w_{n}\|_{H^{s-2\alpha}}\right)\\
		&+\|\Lambda^{2\alpha}\left(G_{0}(\mathbb{J}_{\epsilon_{n+p}}-\mathbb{J}_{\epsilon_{n}})w_{n}\right)\|_{H^{s-2}}+\|G_{0}(\mathbb{J}_{\epsilon_{n+p}}-\mathbb{J}_{\epsilon_{n}})\Lambda^{2\alpha}w_{n}\|_{H^{s-2}}\\
		&\leq C_{0}(2-2\alpha)|\epsilon_{n+p}-\epsilon_{n}|^{2-2\alpha}\left(\|\Lambda^{2\alpha}(G_{0}\mathbb{J}_{\epsilon_{n}}w_{n})\|_{H^{s-2\alpha}}+\|G_{0}\mathbb{J}_{\epsilon_{n}}\Lambda^{2\alpha}w_{n}\|_{H^{s-2\alpha}}+C(\mathcal{O}_{1},M)\|\Lambda^{2\alpha}w_{n}\|_{H^{s-2\alpha}}\right)\\
		&+\|G_{0}(\mathbb{J}_{\epsilon_{n+p}}-\mathbb{J}_{\epsilon_{n}})w_{n}\|_{H^{s}}\\
		&\leq C_{0}(2-2\alpha)|\epsilon_{n+p}-\epsilon_{n}|^{2-2\alpha}\left(C(\mathcal{O}_{1},M)\|w_{n}\|_{H^{s}}+\|G_{0}\mathbb{J}_{\epsilon_{n}}\Lambda^{2\alpha}w_{n}\|_{H^{s-2\alpha}}\right)\\
		&+C(\mathcal{O}_{1},M)\|(\mathbb{J}_{\epsilon_{n+p}}-\mathbb{J}_{\epsilon_{n}})w_{n}\|_{H^{s}}.
	\end{align*}
Applying \eqref{eq:leibemb}, we obtain
\begin{align*}
\|G_{0}\mathbb{J}_{\epsilon_{n}}\Lambda^{2\alpha}w_{n}\|_{H^{s-2\alpha}}&\leq\|G_{0}\mathbb{J}_{\epsilon_{n}}\Lambda^{2\alpha}w_{n}\|_{L^{2}}+\|\Lambda^{s-2\alpha}\left(G_{0}\mathbb{J}_{\epsilon_{n}}\Lambda^{2\alpha}w_{n}\right)\|_{L^{2}}\\
&\leq C(\mathcal{O}_{1},M)\|w_{n}\|_{H^{s}}+C\|\Lambda^{s-2\alpha}G_{0}\|_{H^{1}}\|\Lambda^{2\alpha}w_{n}\|_{H^{1}}.
\end{align*}
	Since $2\alpha>1$, we have $\|\Lambda^{s-2\alpha}G_{0}\|_{H^{1}}\leq  C(\mathcal{O}_{1},M)$ and given that $s\geq 3$, we deduce the bound
	\begin{align*}
		\|G_{0}\mathbb{J}_{\epsilon_{n}}\Lambda^{2\alpha}w_{n}\|_{H^{s-2\alpha}}\leq C(\mathcal{O}_{1},M)\|w_{n}\|_{H^{s}}.
	\end{align*}
	Hence, for $\alpha\in(\tfrac{1}{2},1)$,
	\begin{align}
		R_{n,p}^{\alpha}\leq C_{3}(\mathcal{O}_{1},M)\left(|\epsilon_{n+p}-\epsilon_{n}|^{2-2\alpha}\|w_{n}\|_{H^{s}}+\|(\mathbb{J}_{\epsilon_{n+p}}-\mathbb{J}_{\epsilon_{n}})w_{n}\|_{H^{s}}\right).\label{eq:Rnmorehalf}
	\end{align}
	Substituting \eqref{eq:Rnhalf} and \eqref{eq:Rnmorehalf} into \eqref{eq:secondsourcealmost}, we arrive at
	\begin{equation}
		\label{eq:finalsourcesecond}
		\begin{aligned}
			\|f_{n,p}^{2,\alpha}\|_{H^{s-2-\alpha}}&\leq C_{3}(\mathcal{O}_{1},M)\left(|\epsilon_{n+p}-\epsilon_{n}|^{\alpha}+|\epsilon_{n+p}-\epsilon_{n}|^{2-2\alpha}+|\epsilon_{n+p}-\epsilon_{n}|\right)\left(\|f_{2}\|_{H^{s-2}}+\|U_{n}\|_{H^{s}}\right)\\
			&+C_{3}(\mathcal{O}_{1},M)\|(\mathbb{J}_{\epsilon_{n+p}}-\mathbb{J}_{\epsilon_{n}})U_{n}\|_{H^{s}}.
		\end{aligned}
	\end{equation}
Since $\epsilon_{n}\rightarrow 0$, we may assume that $|\epsilon_{n+p}-\epsilon_{n}|<1$. Combining \eqref{eq:firstsource} and \eqref{eq:finalsourcesecond} with \eqref{eq:Cauchyestimate}, we find a constant $C_{4}(\mathcal{O}_{1},M)$ such that, for any $t\in[0,T]$
\begin{align*}
&\|\mathbb{J}_{\epsilon_{n+p}}V_{n,p}(t)\|_{H^{s-2}}^{2}+\int_{0}^{t}\|\Lambda^{\alpha}\mathbb{J}_{\epsilon_{n+p}}z_{n,p}(\tau)\|_{H^{s-2}}^{2}d\tau\\
&\leq C_{4}(\mathcal{O}_{1},M)K(T)\left(|\epsilon_{n+p}-\epsilon_{n}|^{2\alpha}+|\epsilon_{n+p}-\epsilon_{n}|^{2(2-2\alpha)}\right)\left(\Phi_{\alpha}^{s}(U_{0},F,T)+T\sup_{t\in[0,T]}\|U_{n}(t)\|_{H^{s}}^{2}\right)\\
&+C_{4}(\mathcal{O}_{1},M)K(T)\int_{0}^{T}\|(\mathbb{J}_{\epsilon_{n+p}}-\mathbb{J}_{\epsilon_{n}})U_{n}(t)\|_{H^{s}}^{2}d\tau.
\end{align*}
Notice that we cannot apply \eqref{eq:Cauchyrate} to the last term of the previous inequality. Nonetheless, for every $t\in[0,T]$, $\|(\mathbb{J}_{\epsilon_{n+p}}-\mathbb{J}_{\epsilon_{n}})U_{n}(t)\|_{H^{s}}\rightarrow 0$ as $n\rightarrow\infty$ and 
\begin{align*}
\|(\mathbb{J}_{\epsilon_{n+p}}-\mathbb{J}_{\epsilon_{n}})U_{n}(t)\|_{H^{s}}^{2}\leq 2\|U_{n}(t)\|_{H^{s}}^{2}\leq 2 K(T)\Phi_{\alpha}^{s}(U_{0},F,T).
\end{align*}
In consequence, by the dominated convergence theorem, 
\begin{equation}
\label{eq:firstcauchyzero}
\sup_{t\in[0,T]}\|\mathbb{J}_{\epsilon_{n+p}}V_{n,p}(t)\|_{H^{s-2}}^{2}+\int_{0}^{T}\|\Lambda^{\alpha}\mathbb{J}_{\epsilon_{n+p}}z_{n,p}(\tau)\|_{H^{s-2}}^{2}d\tau\rightarrow 0\quad\mbox{as}\quad n\rightarrow\infty\quad\mbox{for any}\quad p\in\mathbb{N}_{0}.
\end{equation}
Using the triangle inequality,
\begin{align*}
E_{s-2,\alpha,T}(\mathbb{J}_{\epsilon_{n+p}}U_{n+p}-\mathbb{J}_{\epsilon_{n}}U_{n})\leq E_{s-2,\alpha,T}(\mathbb{J}_{\epsilon_{n+p}}V_{n,p})+E_{s-2,\alpha,T}((\mathbb{J}_{\epsilon_{n+p}}-\mathbb{J}_{\epsilon_{n}})U_{n}).
\end{align*}
By \eqref{eq:firstcauchyzero},  $E_{s-2,\alpha,T}(\mathbb{J}_{\epsilon_{n+p}}V_{n,p})\rightarrow 0$ as $n\rightarrow\infty$, uniformly in $p\in\mathbb{N}_{0}$. Furthermore, by \eqref{eq:Cauchyrate} with $\sigma=s-2$ and $\delta=1$ and the bound in \eqref{eq:boundHs}, 
\begin{align*}
E_{s-2,\alpha,T}^{2}((\mathbb{J}_{\epsilon_{n+p}}-\mathbb{J}_{\epsilon_{n}})U_{n})\leq|\epsilon_{n+p}-\epsilon_{n}|^{2}K(T)\Phi_{\alpha}^{s}(U_{0},F,T)\rightarrow 0\quad\mbox{as}\quad n\rightarrow\infty.
\end{align*}
	Hence, for any $p\in\mathbb{N}_{0}$, $E_{s-2,\alpha,T}(\mathbb{J}_{\epsilon_{n+p}}U_{n+p}-\mathbb{J}_{\epsilon_{n}}U_{n})\rightarrow 0$ as $n\rightarrow\infty$. Since, by Theorem \ref{existenceUniqLSn}, $\{U_{n}\}\subset\mathcal{C}([0,T];H^{s})$, we conclude that $\{\mathbb{J}_{\epsilon_{n}}U_{n}\}$ is a Cauchy sequence in $X_{s-2,\alpha,T}$. Consequently, $\{\mathbb{J}_{\epsilon_{n}}U_{n}\}$ is a Cauchy sequence in $X_{\sigma,\alpha,T}$ for every $\sigma\in[0,s-2]$. 
	
	It remains to consider the case $\sigma\in(s-2,s)$. Let $\theta\in(0,1)$ be such that $\sigma=(1-\theta)(s-2)+\theta s$. By the interpolation between non-homogeneous Sobolev spaces and H\"older's inequality with $p=(1-\theta)^{-1}$ and $q=\theta^{-1}$, we have, for any $t\in[0,T]$,
\begin{align*}
&E_{\sigma,\alpha}^{2}(\mathbb{J}_{\epsilon_{n}}U_{n}(t)-\mathbb{J}_{\epsilon_{m}}U_{m}(t))\leq\|\mathbb{J}_{\epsilon_{n}}U_{n}(t)-\mathbb{J}_{\epsilon_{m}}U_{m}(t)\|_{H^{s-2}}^{2(1-\theta)}\|\mathbb{J}_{\epsilon_{n}}U_{n}(t)-\mathbb{J}_{\epsilon_{m}}U_{m}(t)\|_{H^{s}}^{2\theta}\\
&+\left(\int_{0}^{t}\|\Lambda^{\alpha}\left(\mathbb{J}_{\epsilon_{n}}v_{n}(\tau)-\mathbb{J}_{\epsilon_{m}}v_{m}(\tau)\right)\|_{H^{s-2}}^{2}d\tau\right)^{1-\theta}\left(\int_{0}^{t}\|\Lambda^{\alpha}\left(\mathbb{J}_{\epsilon_{n}}v_{n}(\tau)-\mathbb{J}_{\epsilon_{m}}v_{m}(\tau)\right)\|_{H^{s}}^{2}d\tau\right)^{\theta}.
	\end{align*}
The bound in \eqref{eq:boundHs} shows that $\{\mathbb{J}_{\epsilon_{n}}U_{n}\}$ is bounded in $X_{s,\alpha,T}^{\infty}$, while the previous argument proves that it is Cauchy in $X_{s-2,\alpha,T}$. Hence, by interpolation, we deduce that $E_{\sigma}^{2}(\mathbb{J}_{\epsilon_{n}}U_{n}(t)-\mathbb{J}_{\epsilon_{m}}U_{m}(t))\rightarrow 0$ as $n,m\rightarrow\infty$  for any $t\in[0,T]$. Therefore, $\{\mathbb{J}_{\epsilon_{n}}U_{n}\}$ is a Cauchy sequence in $X_{\sigma,\alpha,T}$ for every $\sigma\in[0,s)$ and \eqref{eq:Cauchylimit} follows. 

Next, we establish the convergence of the sequence $\{\partial_{t}U_{n}\}$ in a suitable space.
By Theorem \ref{existenceUniqLSn}, the difference $V_{n,p}=U_{n+p}-U_{n}$ satisfies \eqref{eq:Cauchydifference1} pointwise in $\mathcal{C}([0,T];H^{s})$, with source term given by \eqref{eq:CauchydifferenceS1}. Taking the $H^{s-3}$-norm in \eqref{eq:Cauchydifference1} and using the Sobolev product estimates, we obtain
\begin{equation}
\label{eq:Cauchytimederivative1}
		\begin{aligned}
			\|\partial_{t}V_{n,p}\|_{H^{s-3}}\leq\|f_{n,p}^{1,\alpha}\|_{H^{s-2}}+\|f_{n,p}^{2,\alpha}\|_{H^{s-3}}+C(\mathcal{O}_{1},M)\left[\|\mathbb{J}_{\epsilon_{n+p}}V_{n,p}\|_{H^{s-2}}+\|\mathbb{J}_{\epsilon_{n+p}}V_{n,p}\|_{H^{s-1}}\right].
		\end{aligned}
	\end{equation}
	By estimate \eqref{eq:firstsource}, 
	\begin{align*}
		\|\partial_{t}V_{n,p}\|_{H^{s-3}}&\leq\|f_{n,p}^{2,\alpha}\|_{H^{s-3}}+C_{0}C(\mathcal{O}_{1},M)|\epsilon_{n+p}-\epsilon_{n}|(\|f_{1}\|_{H^{s-1}}+\|U_{n}\|_{H^{s}})\\
		&+C(\mathcal{O}_{1},M)\left[\|\mathbb{J}_{\epsilon_{n+p}}V_{n,p}\|_{H^{s-2}}+\|\mathbb{J}_{\epsilon_{n+p}}V_{n,p}\|_{H^{s-1}}\right].
	\end{align*}
	Combining the Sobolev product estimates with \eqref{eq:Cauchyrate}, with $\sigma=s-3$ and $\delta=1$, and applying the resulting estimate to all terms in $f_{n,p}^{2,\alpha}$ except the last one, we obtain
	\begin{align*}
		\|f_{n,p}^{2,\alpha}\|_{H^{s-3}}\leq C_{0}C(\mathcal{O}_{1},M)|\epsilon_{n+p}-\epsilon_{n}|(\|f_{2}\|_{H^{s-2}}+\|U_{n}\|_{H^{s}})+C(\mathcal{O}_{1},M)\|[\Lambda^{2\alpha},G_{0}](\mathbb{J}_{\epsilon_{n+p}}-\mathbb{J}_{\epsilon_{n}})w_{n}\|_{H^{s-3}},
	\end{align*}
	where, as in the previous theorem, $w_{n}$ denotes the last $m$ components of $\mathcal{L}U_{n}$. Expanding the commutator in the remaining term and applying the triangle inequality, yields
	\begin{align*}
		\|[\Lambda^{2\alpha},G_{0}](\mathbb{J}_{\epsilon_{n+p}}-\mathbb{J}_{\epsilon_{n}})w_{n}\|_{H^{s-3}}\leq C(\mathcal{O}_{1},M)\|(\mathbb{J}_{\epsilon_{n+p}}-\mathbb{J}_{\epsilon_{n}})w_{n}\|_{H^{s-1}}.
	\end{align*}
Then, \eqref{eq:Cauchyrate} with $\sigma=s-1$ and $\delta=1$, implies
\begin{align*}
\|f_{n,p}^{2,\alpha}\|_{H^{s-3}}\leq C_{0}C(\mathcal{O}_{1},M)|\epsilon_{n+p}-\epsilon_{n}|(\|f_{2}\|_{H^{s-2}}+\|U_{n}\|_{H^{s}})
\end{align*}
and we obtain
\begin{align*}
\sup_{t\in[0,T]}\|\partial_{t}V_{n,p}(t)\|_{H^{s-3}}&\leq C_{3}(\mathcal{O}_{1},M)|\epsilon_{n+p}-\epsilon_{n}|(\|f_{1}\|_{\mathcal{C}([0,T];H^{s-1})}+\|f_{2}\|_{\mathcal{C}([0,T];H^{s-2\alpha})}+\|U_{n}\|_{\mathcal{C}([0,T];H^{s})})\\
&+C_{3}(\mathcal{O}_{1},M)\left[\|\mathbb{J}_{\epsilon_{n+p}}V_{n,p}\|_{\mathcal{C}([0,T];H^{s-2})}+\|\mathbb{J}_{\epsilon_{n+p}}V_{n,p}\|_{\mathcal{C}([0,T];H^{s-1})}\right].
\end{align*}
As a consequence of \eqref{eq:parabolicenergyestimate} and \eqref{eq:Cauchylimit} it follows that $\{\partial_{t}U_{n}\}$ is a Cauchy sequence in $\mathcal{C}([0,T];H^{s-3})$. Therefore, by \eqref{eq:Cauchylimit}, there exists $U\in\cap_{\sigma\in[0,s)}\mathcal{C}([0,T];H^{\sigma})$ and $V\in\mathcal{C}([0,T];H^{s-3})$ such that
\begin{equation}
\label{eq:Cauchyconvergence1}
\mathbb{J}_{\epsilon_{n}}U_{n}\rightarrow U=(u,v)\quad\mbox{in}\quad\mathcal{C}([0,T];H^{\sigma})\quad\mbox{for any}\quad\sigma\in[0,s)\quad\mbox{and}\quad\partial_{t}U_{n}\rightarrow V\quad\mbox{in}\quad\mathcal{C}([0,T];H^{s-3}).
\end{equation}
Moreover, by theorem \ref{existenceUniqLSn}, the identity
	\begin{align*}
		\mathbb{J}_{\epsilon_{n}}U_{n}(t)-\mathbb{J}_{\epsilon_{n}}U_{n}(0)=\int_{0}^{t}\mathbb{J}_{\epsilon_{n}}\partial_{t}U_{n}(\tau)d\tau
	\end{align*}
	holds in the space $\mathcal{C}([0,T];H^{s-3})$ for any $n\in\mathbb{N}$. By taking the limit in the norm of the space $\mathcal{C}([0,T];H^{s-3})$ and using \eqref{eq:Cauchyconvergence1}, it follows that $\partial_{t}U=V$. 
	
We now show that $U$ is a solution of the Cauchy problem \eqref{eq:linearizedsystem1}-\eqref{eq:linearizedinitialdata} by passing to the limit in each term of equation \eqref{eq:LSn} in the space $H^{s-3}$. Using \eqref{eq:Cauchyrate2} with $\sigma=s-3$ and $\delta=1$, we obtain
\begin{align*}
&\|\mathbb{J}_{\epsilon_{n}}(A^{j}\partial_{j}\mathbb{J}_{\epsilon_{n}}U_{n})-A^{j}\partial_{j}U\|_{H^{s-3}},\|\mathbb{J}_{\epsilon_{n}}D\mathbb{J}_{\epsilon_{n}}U_{n}-DU\|_{H^{s-3}},\|\mathbb{J}_{\epsilon_{n}}B\Lambda^{2\alpha}\mathbb{J}_{\epsilon_{n}}U_{n}-B\Lambda^{2\alpha}U\|_{H^{s-3}}\\
&\leq C(\mathcal{O}_{1},M)\left(C_{0}\epsilon_{n}\|U_{n}\|_{H^{s}}+\|\mathbb{J}_{\epsilon_{n}}U_{n}-U\|_{H^{s-1}}\right).
\end{align*}
Likewise for the commutator term,
\begin{align*}
&\|\mathbb{J}_{\epsilon_{n}}\left(H\mathbb{J}_{\epsilon_{n}}[\Lambda^{2\alpha},G]\mathbb{J}_{\epsilon_{n}}\mathcal{L}U_{n}\right)-H[\Lambda^{2\alpha},G]\mathcal{L}U\|_{H^{s-3}}\\
&\leq\|(\mathbb{J}_{\epsilon_{n}}-\mathbb{I})H_{0}\mathbb{J}_{\epsilon_{n}}[\Lambda^{2\alpha},G_{0}]\mathcal{L}\mathbb{J}_{\epsilon_{n}}U_{n}\|_{H^{s-3}}+\|H_{0}(\mathbb{J}_{\epsilon_{n}}-\mathbb{I})[\Lambda^{2\alpha},G_{0}]\mathcal{L}\mathbb{J}_{\epsilon_{n}}U_{n}\|_{H^{s-3}}+\|H_{0}[\Lambda^{2\alpha},G_{0}]\mathcal{L}(\mathbb{J}_{\epsilon_{n}}U_{n}-U)\|_{H^{s-3}}\\
&\leq C_{1}(\mathcal{O}_{1},M)\left(C_{0}\epsilon_{n}\|U_{n}\|_{H^{s}}+\|\mathbb{J}_{\epsilon_{n}}U_{n}-U\|_{H^{s-1}}\right). 
\end{align*}
	The right hand side of each of the previous inequalities converges to zero when $n\rightarrow\infty$ by  \eqref{eq:Cauchyconvergence1} and the uniform boundedness of $\{U_{n}\}$ in $L^{\infty}([0,T];H^{s})$.
	Finally, the assumptions of the theorem and \eqref{eq:Cauchyconvergence1} imply 
\begin{align*}
\lim_{n\rightarrow\infty}\|\mathbb{J}_{\epsilon_{n}}F-F\|_{\mathcal{C}([0,T];H^{s-3})},~\|A^{0}\partial_{t}U_{n}-A^{0}\partial_{t}U\|_{\mathcal{C}([0,T];H^{s-3})},\|\mathbb{J}_{\epsilon_{n}}U_{0}-U_{0}\|_{H^{s-3}}=0.
\end{align*}
Therefore, $U$ is a solution of equation \eqref{eq:linearizedsystem1} that satisfies the initial condition \eqref{eq:linearizedinitialdata} in the space $\mathcal{C}([0,T];H^{s-3})$.
	
Now, by theorem \ref{energyestimatesF} and \eqref{eq:Cauchylimit}, we can use a lower semi-continuity argument to show \eqref{eq:estimateforapprox} whenever $m=s$. If $m\in[1,s)\cap\mathbb{N}$, pass to the limit in \eqref{eq:parabolicenergyestimate} and \eqref{eq:estimateforapprox} follows.
	
	Now assume that $\alpha\in(0,\tfrac{1}{2}]$ and let $\sigma\in[0,s)$. In this case, the embedding $H^{\sigma-2\alpha}\hookrightarrow H^{\sigma-1}$ implies that, $f_{2}$, $B\Lambda^{2\alpha}U$ and $H[\Lambda^{2\alpha},G]U$ belong to the space $\mathcal{C}([0,T];H^{\sigma-1})$. Since $U$ is a solution of \eqref{eq:linearizedsystem1}, it follows that $U\in\mathcal{C}([0,T];H^{\sigma})\cap\mathcal{C}^{1}([0,T];H^{\sigma-1})$ for every $\sigma\in[0,s)$. On the other hand, if $\alpha\in[\tfrac{1}{2},1)$ the previous conclusion remains true for the first component $u$ of $U$, that is, $u\in\mathcal{C}([0,T];H^{\sigma})\cap\mathcal{C}^{1}([0,T];H^{\sigma-1})$. Meanwhile, for the second component $v$ of $U$, we can only guarantee that $v\in\mathcal{C}([0,T];H^{\sigma})\cap\mathcal{C}^{1}([0,T];H^{\sigma-2\alpha})$ and the result follows. 
\end{proof}

\section*{Acknowledgments}
The author is grateful to Professors Ram\'on Plaza and Luis L\'opez for their encouragement and support during the preparation of this manuscript. This work was supported by C\'atedra IIMAS from August 2024 to December of the same year and by Estancias Posdoctorales por M\'exico 2024 SECIHTI, grant 589929, from February 2025 to August of the same year.

	\bibliographystyle{plain} 
	\bibliography{dissref}
	
\end{document}